\documentclass[smallcondensed,envcountsame,envcountreset,numbook]{svjour3m}     
\smartqed  
\usepackage{latexsym}
\usepackage{amsmath}
\usepackage{amssymb}
\usepackage[colorlinks]{hyperref}
\usepackage{paralist}

\usepackage{graphicx}

\newtheorem{defn}[theorem]{Definition}
\newtheorem{assumption}[theorem]{Assumption}

\newtheorem{rem}[theorem]{Remark}

\newcommand{\rd}{\color{red}}
\renewcommand{\rd}{}

\newcommand{\Rd}{\color{red}}
\renewcommand{\Rd}{}

\newcommand{\bl}{\color{blue}}
\renewcommand{\bl}{}
\newcommand{\Bl}{\color{blue}}
\renewcommand{\Bl}{}

\newcommand{\MG}{\color{magenta}}
\renewcommand{\MG}{}

\newcommand{\bn}{\color[rgb]{0.5,0.25,0.0}}
\renewcommand{\bn}{}
\newcommand{\Bn}{\color[rgb]{0.5,0.25,0.0}}
\renewcommand{\Bn}{}

\newcommand{\Gr}{\color[rgb]{0.0,0.45,0.0}}
\renewcommand{\Gr}{}

\newcommand{\gr}{\color[rgb]{0.0,0.45,0.0}}
\renewcommand{\gr}{}
\newcommand{\gR}{\color[rgb]{0.0,0.45,0.0}}
\renewcommand{\gR}{}

\newcommand{\bs}{\boldsymbol}

\newcommand{\comment}[1]{}
\allowdisplaybreaks

\newcommand{\R}{\mathbb R}

\journalname{}
\begin{document}
%
%






%
%
%
%
%
%
%
%
%


\title{Non-homogeneous Dirichlet-transmission
problems
for the anisotropic Stokes and Navier-Stokes systems in Lipschitz domains with transversal interfaces}

\titlerunning{Non-homogeneous Dirichlet-transmission problems for the anisotropic Navier-Stokes}

\author{Mirela Kohr \and Sergey E. Mikhailov$^*$\thanks{$^*$ Corresponding author} \and Wolfgang L. Wendland}

\authorrunning{M. Kohr, S.E. Mikhailov and W.L. Wendland}

\institute{M. Kohr
           \at
              Faculty of Mathematics and Computer Science, Babe\c{s}-Bolyai University,\\
              1 M. Kog\u{a}lniceanu Str., 400084 Cluj-Napoca, Romania \\
              \email{mkohr@math.ubbcluj.ro}
          \and
           S.E. Mikhailov
           \at
           Department of Mathematics, Brunel University London,
             Uxbridge, UB8 3PH,\\ United Kingdom\\
             \email{sergey.mikhailov@brunel.ac.uk}
           \and
           W.L. Wendland
           \at
              Institut f\"ur Angewandte Analysis und Numerische
              Simulation, Universit\"at Stuttgart,\\
              Pfaffenwaldring, 57, 70569 Stuttgart, Germany\\
              \email{wendland@mathematik.uni-stuttgart.de}
}
\date{\today}

\maketitle


\begin{abstract}
This paper is build around the stationary anisotropic Stokes and Navier-Stokes systems with an $L^\infty$-tensor coefficient satisfying an ellipticity condition in terms of symmetric matrices in ${\mathbb R}^{n\times n}$ with zero matrix traces. We analyze, in $L^2$-based Sobolev spaces, the non-homogeneous boundary value problems of Dirichlet-transmission type for the anisotropic Stokes and Navier-Stokes systems in a compressible framework in a bounded Lipschitz domain with a Lipschitz interface  in ${\mathbb R}^n$, $n\ge 2$ ($n=2,3$ for the nonlinear problems).
The  transversal interface intersects the boundary of the Lipschitz domain.
First, we use a  mixed variational approach to prove well-posedness results for the linear anisotropic Stokes system.
Then we show the existence of a weak solution  for the nonlinear anisotropic Navier-Stokes system by implementing the Leray-Schauder fixed point theorem and using various results and estimates from the linear case, as well as  the Leray-Hopf and some other norm inequalities. Explicit conditions for uniqueness of solutions to the nonlinear problems are also provided.

\keywords{Anisotropic Stokes and Navier-Stokes systems with $L^\infty$ coefficients \and  variational problem \and $L^2$-based Sobolev spaces \and Dirichlet-transmission problems \and existence results \and fixed point theorems.}
\subclass{35J57 \and 35Q30 \and  31C \and 46E35\and 76D \and 76M}
\end{abstract}





\section{\bf Introduction}

Variational methods have been intensively used in the analysis of elliptic boundary problems, in particular, boundary value problems for the Stokes and Navier-Stokes equations (see, e.g., \cite{Co,Ga-Sa,Sayas-book}).
Employing variational methods, Angot \cite{Angot-2,Angot-3} analyzed a well-posedness of some Stokes/Brinkman problems with constant isotropic viscosity and a family of embedded jump conditions on an immersed (transversal) interface with weak regularity assumptions.

The authors in \cite{K-L-W} combined a layer potential approach with the Leray-Schauder fixed point theorem and proved existence results for a nonlinear Neumann-transmission problem for the Stokes and Brinkman systems in $L^p$, Sobolev, and Besov spaces.

Dong and Kim \cite{Dong-Kim} obtained regularity results for the Stokes system with measurable coefficients in one direction (see also \cite{Choi-Dong-Kim-JMFM}). Korobkov, Pileckas and Russo \cite{Korobkov} analyzed the flux problem in the theory of steady Navier-Stokes equations with constant coefficients and non-homogeneous boundary conditions.
Amrouche and Rodr\'{i}guez-Bellido \cite{Amrouche-Bellido} proved the existence of a very weak solution for the {\bn non-}homogeneous Dirichlet problem for the compressible Navier-Stokes system in a bounded domain of the class $C^{1,1}$ in ${\mathbb R}^3$.

An alternative integral approach, which reduces boundary value problems for the Stokes system with variable coefficients and a large spectrum of other variable-coefficient elliptic partial differential equations to {boundary-domain integral equations} (BDIEs), by employing explicit parametrix-based integral potentials, was developed in \cite{CMN-1,CMN-2,Ch-Mi-Na,
Mikh-18,FP-Mik2019}.

Mazzucato and Nistor \cite{Ma-Ni} obtained well-posedness and regularity results in weighted Sobolev spaces for the anisotropic linear elasticity equations with mixed boundary conditions on polyhedral domains. Brewster et al. \cite{B-M-M-M} used a variational approach to show well-posedness of Dirichlet, Neumann and mixed boundary problems for higher order divergence-form elliptic equations with $L^\infty$ coefficients in locally $(\epsilon,\delta )$-domains and in Besov and Bessel potential spaces.

The authors in \cite{K-M-W-2} developed a variational analysis in the pseudostress setting for transmission problems with internal interfaces in weighted Sobolev spaces for the anisotropic Stokes and Navier-Stokes systems with $L^\infty $ strongly elliptic coefficient tensor, see also \cite{Dong-Kim}.
Note that in \cite{K-M-W-2} and \cite{Dong-Kim} it was assumed that the coefficients of the viscosity tensor satisfy a stronger ellipticity condition than in \eqref{mu}, for all matrices in ${\mathbb R}^{n\times n}$ (not only for symmetric and with zero-trace, see \cite[Eqs. (2)-(3)]{K-M-W-2}).
Such a condition allowed to explore the associated non-symmetric pseudostress setting (see also \cite{K-M-W}, and \cite{K-W,K-W1} 
for the Stokes and Navier-Stokes systems with non-smooth coefficients in compact Riemannian setting).
The authors extended in \cite{KMW-DCDS2021} and \cite{KMW-LP} their variational analysis to other transmission and exterior boundary problems with internal interfaces for the anisotropic Stokes and Navier-Stokes systems by assuming that the corresponding $L^{\infty }$ viscosity tensor coefficient satisfies the ellipticity condition only in terms of symmetric matrices in ${\mathbb R}^{n\times n}$ with zero traces, that is, the strong ellipticity condition \eqref{mu}.
Only homogeneous Dirichlet conditions and zero velocity jumps were considered in the (nonlinear) Navier-Stokes problems in \cite{K-M-W}-\cite{K-W1}.

In this paper we
investigate {\it non-homogeneous} Dirichlet-transmission problems for the anisotropic Stokes and Navier-Stokes systems in a bounded Lipschitz domain of ${\mathbb R}^n$ ($n=2,3$ for the nonlinear problems) with {\it a transversal Lipschitz interface that intersects the boundary of the domain.}
As in \cite{KMW-DCDS2021} and \cite{KMW-LP}, we impose the ellipticity condition \eqref{mu}, which is less restrictive than in \cite{K-M-W-2} and \cite{Dong-Kim}.
We show well-posedness results for the linear problems, as well as existence results for the nonlinear problems in $L^2$-based Sobolev spaces.
First, we explore equivalent mixed variational formulations and prove the well-posedness of linear Dirichlet-transmission problems for the anisotropic Stokes system in a compressible framework in bounded Lipschitz domains of ${\mathbb R}^n$ with transversal Lipschitz interfaces and given data in $L^2$-based Sobolev spaces.
Next, we use well-posedness results in the linear case and the Leray-Schauder fixed point theorem and show the existence of a weak solution of the Dirichlet problem for the anisotropic Navier-Stokes system with general non-homogeneous data in $L^2$-based Sobolev spaces in a bounded Lipschitz domain in ${\mathbb R}^n$, $n=2,3$.
Finally, we prove the existence of weak solutions ${\bf u}$ of the Dirichlet-transmission problems for the anisotropic Navier-Stokes system in a bounded Lipschitz domain in ${\mathbb R}^n$, $n=2,3$, with transversal Lipschitz interface and data in $L^2$-based Sobolev spaces.

In addition to their mathematical interest, the anisotropic Stokes and Navier-Stokes transmission problems analyzed in this paper are also motivated
by industrial applications related to the flow of immiscible fluids, liquid crystals, and flows of non-homogeneous fluids
with variable anisotropic viscosity tensors depending on physical properties of the fluids (cf., e.g., \cite{duffy}, \cite[Chapter 3]{Malkin}).

\section{Anisotropic Stokes system with elliptic $L^\infty $ viscosity tensor coefficient} 

Let $\Omega \subseteq\mathbb R^n$, $n\ge 2$, be an open set, and let
$\boldsymbol{\mathfrak L}$ denote a second order differential operator in the component-wise divergence form\footnote{The standard notation $\partial _\beta $ for the first order partial derivative $\displaystyle\frac{\partial }{\partial x_\beta }$, $\beta =1,\ldots ,n$, and the Einstein summation rule on repeated indices are used all along the paper.},

\begin{align}
\label{L-oper}
&(\boldsymbol{\mathfrak L}{\bf u})_i:=
\partial _\alpha\big(a_{ij}^{\alpha \beta }E_{j\beta }({\bf u})\big),\ \ i=1,\ldots ,n,
\end{align}
where ${\bf u}\!=\!(u_1,\ldots ,u_n)^\top$, and $E_{j\beta }({\bf u})\!:=\!\frac{1}{2}(\partial _ju_\beta +\partial _\beta u_j)$ are the {entries} of the symmetric part ${\mathbb E}({\bf u})$ of $\nabla {\bf u}$ (the gradient of ${\bf u}$).
The coefficients $a_{ij}^{\alpha \beta }$ are essentially bounded, measurable, real-valued functions constituting the tensor coefficient ${\mathbb A}$, that is,
\begin{align}
\label{Stokes-1}
\!\!\!\!\!{\mathbb A}:=\!\left({a_{ij}^{\alpha \beta }}\right)_{1\leq i,j,\alpha ,\beta \leq n},\  a_{ij}^{\alpha \beta }\in L^\infty(\Omega ),\ 1\leq i,j,\alpha ,\beta \leq n,
\end{align}
and satisfying the following symmetry conditions
\begin{align}
\label{Stokes-sym}
a_{ij}^{\alpha \beta }(x)=a_{\alpha j}^{i\beta }(x)=a_{i\beta }^{\alpha j}(x),\ \ x\in \Omega
\end{align}
(see also \cite[(3.1),(3.3)]{Oleinik}). In addition, we require that ${\mathbb A}$ satisfies the ellipticity condition only in terms of all {\it symmetric} matrices in ${\mathbb R}^{n\times n}$ with {\it zero matrix trace}. This ellipticity condition has been first used in \cite{KMW-DCDS2021,KMW-LP}. Thus, we assume that there exists a constant $C_{\mathbb A} >0$ such that, for almost all $x\in \Omega$,
\begin{align}
\label{mu}
a_{ij}^{\alpha \beta }(x)\xi _{i\alpha }\xi _{j\beta }\geq C_{\mathbb A}^{-1}|\boldsymbol\xi|^2\,,
\ \ &\forall\ \boldsymbol\xi =(\xi _{i\alpha })_{i,\alpha =1,\ldots ,n}\in {\mathbb R}^{n\times n}\nonumber\\
&\mbox{ such that }\, \boldsymbol\xi=\boldsymbol\xi^\top \mbox{ and }
\sum_{i=1}^n\xi _{ii}=0,
\end{align}
where $|\boldsymbol\xi |^2=\xi _{i\alpha }\xi _{i\alpha }$, and the superscript $\top $ denotes the transpose of a matrix. The tensor coefficient ${\mathbb A}$ is endowed with the norm
\begin{align}
\label{A}
\|{\mathbb A}\|:=\max\left\{\|a_{ij}^{\alpha \beta }\|_{\rd L^\infty (\Omega)}:i,j,\alpha ,\beta =1\ldots ,n\right\}.
\end{align}

The symmetry conditions \eqref{Stokes-sym} lead to the following equivalent forms of the operator $\boldsymbol{\mathfrak L}$,
{\begin{equation}
\label{Stokes-0}
\begin{array}{lll}
(\boldsymbol{\mathfrak L}{\bf u})_i=\partial _\alpha\big(a_{ij}^{\alpha \beta }\partial _\beta u_j\big),\ \ i=1,\ldots ,n;\quad
\boldsymbol{\mathfrak L}{\bf u}=\partial _\alpha\left(A^{\alpha \beta }\partial _\beta {\bf u}\right),
\end{array}
\end{equation}
and the tensor coefficient ${\mathbb A}$ can be considered as consisting of $n\times n$ matrix valued functions $A^{\alpha \beta }$, i.e.,
\begin{align}
\label{Stokes-1A}
\!\!\!\!\!{\mathbb A}\!=\!\left(A^{\alpha \beta }\right)_{1\leq \alpha ,\beta \leq n},\
A^{\alpha \beta }:=\left(a_{ij}^{\alpha \beta }\right)_{1\leq i,j\leq n},\ 1\leq \alpha ,\beta \leq n.
\end{align}}

Let ${\bf u}$ be an unknown vector field, $\pi $ be an unknown scalar field, ${\bf f}$ be a given vector field and $g$ be a given scalar field defined in $\Omega $. Then the equations
\begin{equation}
\label{Stokes}
\begin{array}{lll}
\boldsymbol{\mathcal L}({\bf u},\pi ):=\boldsymbol{\mathfrak L}{\bf u}-\nabla \pi={\bf f},\ {\rm{div}}\ {\bf u}=g \mbox{ in } \Omega
\end{array}
\end{equation}
determine the {\it Stokes system with variable anisotropic viscosity tensor coefficient ${\mathbb A}=\left(A^{\alpha \beta }\right)_{1\leq \alpha ,\beta \leq n}$ in a compressible framework}.

According to \eqref{Stokes-0} and \eqref{L-oper}, the Stokes operator $\boldsymbol{\mathcal L}$ can be written in any of the following equivalent forms
\begin{align}
\label{Stokes-new}
&\boldsymbol{\mathcal L}({\bf u},\pi )=\partial _\alpha\left(A^{\alpha \beta }\partial _\beta {\bf u}\right)-\nabla \pi ,\\
&\left(\boldsymbol{\mathcal L}({\bf u},\pi )\right)_i=\partial _\alpha\big(a_{ij}^{\alpha \beta }E_{j\beta }({\bf u})\big)-\partial _i\pi ,\ i=1,\ldots ,n\,.
\end{align}
In addition, the following nonlinear system
\begin{align}
\label{anisotropic-NS}
\boldsymbol{\mathcal L}({\bf u},\pi )-{({\bf u}\cdot \nabla ){\bf u}}={\bf f}\,, \ \ {\rm{div}} \, {\bf u}=g {\rd \mbox{ in } \Omega }
\end{align}
is called the {\it anisotropic Navier-Stokes system with variable viscosity tensor coefficient ${\mathbb A}\!=\!\left(A^{\alpha \beta }\right)_{1\leq \alpha ,\beta \leq n}$ {\Rd in a compressible framework}}.

If ${\rm{div}}\ {\bf u}=0$ in \eqref{Stokes} and \eqref{Stokes-new} one obtains the {\it anisotropic Stokes and Navier-Stokes systems in the incompressible case}.

In the {\it isotropic case}, the tensor ${\mathbb A}$ in \eqref{Stokes-1} has the following entries
\begin{align}
\label{isotropic}
a_{ij}^{\alpha \beta}{(x)}={\lambda{(x)} \delta _{i\alpha }\delta _{j\beta }}+\mu{(x)} \left(\delta_{\alpha j}\delta _{\beta i}+\delta_{\alpha \beta }\delta _{ij}\right),\ 1\leq i,j, \alpha ,\beta \leq n\,,
\end{align}
where $\lambda,\mu\in L^{\infty }(\Omega)$, and
$c_\mu^{-1}\leq\mu(x) \leq {c_\mu} \mbox{ for a.e. } x\in\Omega $,
with some constant $c_\mu>0$ (cf., e.g., Appendix III, Part I, Section 1 in \cite{Temam}).
Then it is immediate that condition \eqref{mu} is fulfilled (see also \cite{KMW-LP}) and thus our results apply also to the Stokes system in the isotropic case.

\section{Functional framework and preliminaries}
\label{preliminaries}
\setcounter{equation}{0}

Given a Banach space ${\mathcal X}$, its topological dual is denoted by ${\mathcal X}'$, and the notation $\langle \cdot ,\cdot \rangle _X$ means the duality pairing of two dual spaces defined on a set $X\subseteq {\mathbb R}^n$.

\subsection{\bf Sobolev spaces on Lipschitz domains in ${\mathbb R}^n$}

Let $n\geq 2$ and let $\Omega $ be a bounded Lipschitz domain in ${\mathbb R}^n$, i.e., an open and connected set with a connected boundary $\partial \Omega $.
Let ${\mathcal D}(\Omega ):=C^{\infty }_{0}(\Omega )$ denote the space of infinitely differentiable functions with compact support in $\Omega $, equipped with the inductive limit topology. Let ${\mathcal D}'(\Omega )$ denote the corresponding space of distributions on $\Omega $, i.e., the dual of the space ${\mathcal D}(\Omega )$.
Let $L^2(\Omega )$ be the Lebesgue space of
square-integrable functions on $\Omega $, and $L^\infty(\Omega )$ be the space of (equivalence classes of) essentially bounded measurable functions on $\Omega $.
Let also
\begin{align}
\label{L2-0}
L_0^2(\Omega ):=\{f\in L^2(\Omega ):\langle f,1\rangle _{\Omega }=0\}\,.
\end{align}
The dual of $L_0^2(\Omega )$ is the space $L^2(\Omega )/{\mathbb R}$.
The Sobolev space $H^1({\Omega })$ is defined as
\begin{align}
\label{bessel-potential2}
H^1({\Omega }):=\big\{f\in L^2({\Omega }):\nabla f\in L^2({\Omega })^n\big\},
\end{align}
and is endowed with the norm
\begin{align}
\label{Sobolev}
\|f\|^2_{H^1({\Omega })}=\|f\|^2_{L^2({\Omega })}+\|\nabla f\|^2_{L^2(\Omega )^n}\,.
\end{align}
The space $\widetilde{H}^1(\Omega )$ is the closure of ${\mathcal D}(\Omega )$ in $H^1({\mathbb R}^n)$, and can be also described as
\begin{align}
\label{2.5}
&\widetilde{H}^1(\Omega ):=\big\{\widetilde{f}\in H^1({\mathbb R}^n):{\rm{supp}}\,
\widetilde{f}\subseteq \overline{\Omega }\big\},
\end{align}
where ${\rm{supp}}f:=\overline{\{x\in {\mathbb R}^n:f(x)\neq 0\}}$. The dual of $\widetilde{H}^1(\Omega )$ is the space $H^{-1}(\Omega )$.
Then the following equivalent characterization of the spaces $H^{\pm 1}(\Omega )$ holds
\begin{align}
\label{spaces-Sobolev-inverse}
&{H^{\pm 1}}(\Omega )=\{f\in {\mathcal D}'(\Omega ):\exists \, F\in {H^{\pm 1}}({\mathbb R}^n) \mbox{ such that } F|_{\Omega }=f\}\,,
\end{align}
where {$|_X=r_{{X}}$ 
is the restriction operator of functions or distributions to a set $X$}.

Let $\mathring{H}^1(\Omega )$ be the closure of ${\mathcal D}(\Omega )$ in ${H}^1(\Omega )$.
The space $\mathring{H}^1(\Omega )$ can be equivalently described as the space of all functions in ${H}^1(\Omega )$ with null traces on the boundary of $\Omega $,
\begin{align}
\mathring{H}^1(\Omega ):=\{f\in {H}^1(\Omega ):\gamma_{_\Omega }f=0 \mbox{ on } \partial \Omega \},
\end{align}
where $\gamma_{_{\Omega }}:H^1(\Omega )\to H^{\frac{1}{2}}(\partial\Omega )$ is the trace operator, see Theorem~\ref{trace-operator1}.
Note that the spaces $\widetilde{H}^1({\Omega })$ and $\mathring{H}^1(\Omega )$ can be identified isomorphically (see, e.g., \cite[Theorem 3.33]{Lean}).

The dual of ${H}^{1}(\Omega )$ is denoted by $\widetilde{H}^{-1}(\Omega )$, and is a space of distributions. (Note that $\widetilde{H}^{-1}({\mathbb R}^n)={H}^{-1}({\mathbb R}^n)$.) Moreover, the following spaces can be isomorphically identified (cf., e.g.,
\cite[Theorem 3.14]{Lean})
\begin{equation}\label{duality-spaces}
\big({H}^{1}(\Omega )\big)'=\widetilde{H}^{-1}(\Omega ),\ \ {H}^{-1}(\Omega )
=\big(\widetilde{H}^{1}(\Omega )\big)'\,.
\end{equation}

Let $s\in (0,1)$. Then the boundary Sobolev space $H^s(\partial \Omega )$ is defined by
\begin{align}
\label{Sobolev-boundary}
H^s(\partial \Omega ):=\left\{f\in L^2(\partial \Omega ): \int _{\partial \Omega }\int_{\partial \Omega }\frac{|f({\bf x})-f({\bf y})|^2}{|{\bf x}-{\bf y}|^{n-1+2s}}d\sigma _{{\bf x}}d\sigma _{{\bf y}}<\infty \right\}\,,
\end{align}
where $\sigma _{\bf y}$ is the surface measure on $\partial \Omega $ (see, e.g., \cite[Proposition 2.5.1]{M-W}). The dual of $H^s(\partial \Omega )$ is the space $H^{-s}(\partial \Omega )$, and $H^0(\partial \Omega )\!=\!L^2(\partial \Omega )$.

By $H^1({\Omega })^n$, $\widetilde{H}^1({\Omega })^n$, and $H^{s}(\partial \Omega )^n$ we denote the spaces of vector-valued functions whose components belong to the Sobolev spaces $H^1({\Omega })$, $\widetilde{H}^1({\Omega })$, and $H^{s}(\partial \Omega )$, respectively.

For further properties of Sobolev spaces we refer the reader to \cite{H-W,Lean,M-W}.

We will need the following well known result (see, e.g., \cite[Lemma 2.5]{LS1976}, \cite{Bogovskii}, \cite[Theorem 3.1]{Am-Ciarlet}), for which we will provide several generalizations further on.
\begin{proposition}
\label{LS-prop}
Let $\Omega $ be a bounded Lipschitz domain in ${\mathbb R}^n$, $n\geq 2$, with connected boundary. Then the divergence operator
$
{\rm{div}}:\mathring{H}^1(\Omega)^n\to {L_0^2(\Omega)}
$
is bounded, linear and surjective. It has a bounded, linear right inverse ${\mathcal R}_{\Omega }:{L^2_0(\Omega)}\to \mathring{H}^{1}(\Omega )^n$. Thus, there exists a constant $C=C(\Omega ,n)>0$ such that
\begin{align}
\label{Bog2}
{\rm{div}}({\mathcal R}_{\Omega }f)=f,\ \|{\mathcal R}_{\Omega }f\|_{{H}^{1}(\Omega )^n}\leq C\|f\|_{{L^2(\Omega )}},\ \forall \, f\in {L^2_0(\Omega )}.
\end{align}
\end{proposition}

\section{Dirichlet problems for the anisotropic compressible Stokes system in bounded Lipschitz domains}
Dindo\u{s} and Mitrea \cite{D-M} 
obtained well-posedness results in Sobolev and Besov spaces for the Dirichlet problem for the Stokes and Navier-Stokes systems with smooth coefficients in Lipschitz domains on compact Riemannian manifolds. Mitrea and Wright \cite{M-W} obtained well-posedness results in Sobolev and Besov spaces for Dirichlet problems for the Stokes system with constant coefficients in Lipschitz domains in ${\mathbb R}^n$  (see also the references therein, and \cite{Amrouche-Bellido} for Dirichlet problems for the Stokes, Oseen and Navier-Stokes systems with constant coefficients in a non-solenoidal framework).
Dirichlet problems for the anisotropic Stokes system in exterior Lipschitz domains and in ${\mathbb R}^n$, $n\geq 3$, have been studied in \cite{K-M-W-2} by using both variational and potential approaches  (see also \cite{Choi-Lee}, \cite{KMW-LP} and \cite{KMW-DCDS2021}).

\subsection{\bf Mixed variational formulation for the anisotropic Stokes system in bounded Lipschitz domains and partly homogeneous Dirichlet problem}
\label{N-S-D}
Let $\Omega \subset {\mathbb R}^n$, $n\geq 2$, be a bounded Lipschitz domain with connected boundary $\partial \Omega $.
Recall that $\mathring{H}^1(\Omega )^n$ is the closure of the space ${\mathcal D}(\Omega )^n$ in ${H}^1(\Omega )^n$ and that 
\begin{align}
\label{seminorm-Omega}
|{\bf u}|_{{H}^1(\Omega )^n}:=\|\nabla {\bf u}\|_{L^2(\Omega )^{n\times n}}
\end{align}
is a norm on the space $\mathring{H}^1(\Omega )$, equivalent to the norm
\begin{align}
\label{norm-H1}
\|{\bf u}\|_{H^1(\Omega )^n}=\|{\bf u}\|_{L^2(\Omega )^n}+\|\nabla {\bf u}\|_{L^2(\Omega  )^{n\times n}}
\end{align}
(cf., e.g., \cite[Theorem II.5.1 and Remark II.6.2]{Galdi}), {\gr that is,
\begin{align}
\label{norm-ineq}
\|{\bf u}\|_{H^1(\Omega )^n}\le\mathring C\|\nabla {\bf u}\|_{L^2(\Omega  )^{n\times n}}\quad \forall\ {\bf u}\in \mathring H^1(\Omega )^n
\end{align}
for some constant $\mathring C=\mathring C(\Omega,n)>0$.}
Let also ${H}^{-1}(\Omega )^n=\big(\mathring{H}^1(\Omega )^n\big)'$ and let
$|\!|\!|\cdot |\!|\!|_{H^{-1}(\Omega )^n}$ denote the corresponding norm on ${H}^{-1}(\Omega )^n$ generated by the semi-norm \eqref{seminorm-Omega}, i.e.,
\begin{align}
\label{norm-3e}
|\!|\!|{\bf g}|\!|\!|_{{H}^{-1}(\Omega )^n}:=\sup_{{\bf v}\in \mathring{H}^1(\Omega )^n,\ \|\nabla {\bf v}\|_{L^2(\Omega )^{n\times n}}=1}|\langle {\bf g},{\bf v}\rangle _{\Omega }|\,, \quad \forall \ {\bf g}\in {H}^{-1}(\Omega )^n.
\end{align}
This implies that
\begin{align*}
|\langle {\bf g},{\bf v}\rangle _{\Omega }|\le|\!|\!|{\bf g}|\!|\!|_{{H}^{-1}(\Omega )^n}\, \|\nabla {\bf v}\|_{L^2(\Omega )^{n\times n}}\,, \quad \forall \ {\bf g}\in {H}^{-1}(\Omega )^n,\ \forall \ {\bf v}\in \mathring{H}^1(\Omega )^n,
\end{align*}
and
\begin{align}
\label{norm-3e0}
\bn\|{\bf g}\|_{{H}^{-1}(\Omega )^n}\le |\!|\!|{\bf g}|\!|\!|_{{H}^{-1}(\Omega )^n}\,, \quad \forall \ {\bf g}\in {H}^{-1}(\Omega )^n.
\end{align}

Let $a_{{\mathbb A};\Omega }:\mathring{H}^1(\Omega )^n\times \mathring{H}^{1}(\Omega )^n\to {\mathbb R}$ and $b_{\Omega }:\mathring{H}^1(\Omega )^n\times {L^2(\Omega )/{\mathbb R}}\to {\mathbb R}$ be the bilinear forms given by {\Bn  \eqref{a-v-A}, \eqref{b-v-A}, i.e.,}
\begin{align}
\label{a-v}
&a_{{\mathbb A};\Omega }({\bf u},{\bf v}):
=\left\langle a_{ij}^{\alpha \beta }E_{j\beta }({\bf u}),E_{i\alpha }({\bf v})\right\rangle _{\Omega },\ \forall \, {\bf u}, {\bf v}\in \mathring{H}^{1}(\Omega )^n\,,\\
\label{b-v}
&b_{\Omega }({\bf v},q):=-\langle {\rm{div}}\, {\bf v},q\rangle _{\Omega },\ \forall \, {\bf v}\in \mathring{H}^1(\Omega )^n,\ \forall \, q\in {L^2(\Omega )/{\mathbb R}}\,.
\end{align}
Let us also introduce the following spaces of divergence-free vector fields
\begin{align*}
&{H}^1_{\rm{div}}(\Omega )^n:=\{{\bf w}\in {H}^1(\Omega )^n:{\rm{div}}\, {\bf w}=0 \mbox{ in } \Omega \}\,,\\
&\mathring{H}^1_{\rm{div}}(\Omega )^n:=\{{\bf w}\in \mathring{H}^1(\Omega )^n:{\rm{div}}\, {\bf w}=0
\mbox{ in } \Omega \}
\end{align*}
and note the characterisation
\begin{align*}
\mathring{H}^1_{\rm div}(\Omega )^n
&=\left\{{\bf w}\in \mathring{H}^1(\Omega )^n: b_{\Omega }({\bf w},q)=0,\ \forall \, q\in  L^2(\Omega )/{\mathbb R}\right\}.
\end{align*}

The H\"{o}lder inequality implies that there exists a constant ${\mathcal C}>0$,
such that
\begin{align}
\label{a-1-v}
|a_{{\mathbb A};\Omega }({\bf u},{\bf v})|
\leq {\mathcal C}\|\nabla {\bf u}\|_{L^2(\Omega )^{n\times n}}\|\nabla {\bf v}\|_{L^2(\Omega )^{n\times n}},\ \forall\, {\bf u},{\bf v}\in \mathring{H}^1(\Omega )^n.
\end{align}
Thus, the bilinear form $a_{{\mathbb A};\Omega }(\cdot ,\cdot ):\mathring{H}^1(\Omega )^n\times \mathring{H}^1(\Omega )^n\to {\mathbb R}$ is bounded. Moreover, the Korn first inequality applied to functions in $\mathring H^1(\Omega )^n$,
\begin{align}
\label{Korn3-R3}
\|\nabla {\bf v}\|_{L^2(\Omega )^{n\times n}}\leq 2^{\frac{1}{2}}\|\mathbb E ({\bf v})\|_{L^2(\Omega )^{n\times n}}
\end{align}
(cf. \cite[Theorem 10.1]{Lean}) combined with the ellipticity condition \eqref{mu} and the property that the semi-norm $\|\nabla (\cdot )\|_{L^2(\Omega )^{3\times 3}}$
is a norm in $\mathring H^1(\Omega )^3$ equivalent to the norm $\|\cdot \|_{H^1(\Omega )^3}$,
shows that the bilinear form
$a_{{\mathbb A};\Omega }(\cdot ,\cdot ): \mathring{H}_{\rm{div}}^1(\Omega )^n\times \mathring{H}_{\rm{div}}^1(\Omega )^n\to {\mathbb R}$
 is coercive, that is,
\begin{align}
\label{a-1-v2-S}
a_{{\mathbb A};\Omega }({\bf v},{\bf v})
\geq \frac{1}{2}C_{\mathbb A}^{-1}\|\nabla {\bf u}\|_{L^2(\Omega )^{n\times n}}^2\,,\ \ \forall \, {\bf v}\in \mathring{H}_{\rm{div}}^1(\Omega )^n\,.
\end{align}
On the other hand, the surjectivity of the operator
${\rm{div}}:\mathring{H}^1(\Omega )^n\to {\bl L_0^2(\Omega )}$
(see Proposition \ref{LS-prop}) shows that the bilinear and bounded form $b_{\Omega }:\mathring{H}^1(\Omega )^n\!\times \!L_0^2(\Omega )\!\to \!{\mathbb R}$ satisfies the inf-sup condition (see Lemma \ref{surj-inj-inf-sup}(ii)).
Then Theorem \ref{B-B} and Remark \ref{RB.2} lead to the following well-posedness result, whose detailed proof can be consulted in \cite[Theorem 3.2]{KMW-DCDS2021} (see also \cite[Lemma 3.1]{K-M-W-2}).
\begin{theorem}\label{lemma-a47-1-Stokes}
Let conditions \eqref{Stokes-1}-\eqref{mu} hold. Let $a_{{\mathbb A};\Omega }$ and $b_{\Omega }$ be the bilinear forms defined in \eqref{a-v} and \eqref{b-v}.
\begin{itemize}
\item[$(i)$]
Then for all given data $\boldsymbol{\mathfrak F} \in {H}^{-1}(\Omega )^n$ and $g\in L_0^2(\Omega )$, the variational problem
\begin{align}
\label{transmission-S-variational-dl-3-equiv-0-2}
\left\{\begin{array}{ll}
a_{{\mathbb A};\Omega }({\bf u},{\bf v})+b_{\Omega }({\bf v},\pi )=\langle \boldsymbol{\mathfrak F} ,{\bf v}\rangle _{\Omega }\,, & \forall \, {\bf v}\in \mathring{H}^1(\Omega )^n,\\
b_{\Omega }({\bf u},q)=-\langle g,q\rangle _{\Omega }\,, & \forall \, q\in {L^2(\Omega )/{\mathbb R}} 
\end{array}
\right.
\end{align}
for $({\bf u},\pi )\in {\mathring{H}^1(\Omega )^n}\times {L^2(\Omega )/{\mathbb R}}$ is well-posed, i.e., \eqref{transmission-S-variational-dl-3-equiv-0-2} has a unique solution and there exists a constant $C>0$ {\bl depending only on $\|{\mathbb A}\|$, $\Omega $ and $n$}, such that
\begin{align}
\label{estimate-1-wp-S-2}
\|{\bf u}\|_{H^1(\Omega )^n}+\|\pi \|_{L^2(\Omega )/\R}\leq
C\left(\|\boldsymbol{\mathfrak F} \|_{{H}^{-1}(\Omega )^n}+\|g\|_{{L^2(\Omega )}}\right).
\end{align}
\item[$(ii)$]
Moreover, the pair $({\bf u},\pi )$ is the unique solution in ${{H}^1(\Omega )^n}\times  L^2(\Omega )/{\mathbb R}$ of the Dirichlet problem for the anisotropic Stokes system
\begin{align}
\label{Dirichlet-homog}
\left\{\begin{array}{ll}
\boldsymbol{\mathcal L}({\bf u},\pi )=-\boldsymbol{\mathfrak F},\quad
{\rm{div}}\, {\bf u}={g} & \mbox{ in } \Omega \,,\\
\gamma _{_{\Omega}}{\bf u}=0 & \mbox{ on } \partial \Omega \,,
\end{array}
\right.
\end{align}
\item[$(iii)$]
The solution can be represented in the form
$(\mathbf u,\pi)=\bs{\mathfrak U}(\boldsymbol{\mathfrak F},g),$
where $\bs{\mathfrak U}:{H}^{-1}(\Omega )^n\times L_0^2(\Omega )\to {{H}^1(\Omega )^n}\times L^2(\Omega )/{\mathbb R}$ is a linear continuous operator.
\end{itemize}
\end{theorem}

\subsection{\bf Non-homogeneous Dirichlet problem}\label{S3.2}
Let us consider the following non-homogeneous Dirichlet problem for the anisotropic Stokes system
\begin{align}
\label{Dirichlet-nonhomog}
\left\{\begin{array}{ll}
\boldsymbol{\mathcal L}({\bf u},\pi )=-\bs{\mathfrak F},\quad
{\rm{div}}\, {\bf u}=g & \mbox{ in } \Omega \,,\\
\gamma _{_{\Omega }}{\bf u}=\bs\varphi & \mbox{ on } \partial \Omega \,,
\end{array}
\right.
\end{align}
for the unknowns $({\bf u},\pi )\in {H}^1(\Omega )^n\times {L^2(\Omega )/{\mathbb R}}$, with the given data
$(\bs{\mathfrak F},g,\bs\varphi)\in{H}^{-1}(\Omega)^n\times L^2(\Omega)\times H^{\frac{1}{2}}(\partial \Omega )^n$,
which satisfy the compatibility condition
\begin{align}
\label{hD}
\int _{\Omega }g(x) dx=\int _{\partial \Omega }\bs\varphi\cdot \boldsymbol\nu d\sigma \,,
\end{align}
where $\boldsymbol\nu$ is the exterior unit normal to $\partial\Omega$.

To analyse the Dirichlet problem \eqref{Dirichlet-nonhomog}, we need the following well-known Bogovskii-type result (see, e.g., \cite{Bogovskii}, \cite{Gi-Ra}, and the proof of Theorem 3.2 in \cite{Amrouche-Bellido}).
\begin{lemma}\label{B-D}
For any
$(g,\boldsymbol \varphi)\in L^2(\Omega)\times H^{\frac{1}{2}}(\partial \Omega )^n$
satisfying condition \eqref{hD}, there exists ${\bf v}\in H^1(\Omega )^n$ such that
\begin{align}
\left\{\begin{array}{ll}
{\rm{div}}\, {\bf v}=g \mbox{ in } \Omega \\
\gamma _{_\Omega }{\bf v}=\bs\varphi \mbox{ on } \partial \Omega \,,
\end{array}
\right.
\end{align}
and there exists a  constant $c=c(\Omega,n)>0$ such that
\begin{align}
\label{v-0-0}
{\|{\bf v}\|_{H^1(\Omega )^n}\leq c\Big(\|g\|_{L^2(\Omega )}+\|\bs\varphi\|_{H^{\frac{1}{2}}(\partial \Omega )^n}\Big)}\,.
\end{align}
\end{lemma}
\begin{theorem}\label{Stokes-Dirichlet-nonhom}
Let conditions \eqref{Stokes-1}-\eqref{mu} hold.
Then for all given data
$(\bs{\mathfrak F},g,\bs\varphi)\in{H}^{-1}(\Omega)^n\times L^2(\Omega)\times H^{\frac{1}{2}}(\partial \Omega )^n$
satisfying condition \eqref{hD},
the Dirichlet problem \eqref{Dirichlet-nonhomog} has a unique solution
$({\bf u},\pi )\in {H}^1(\Omega )^n\times {L^2(\Omega )/{\mathbb R}}$
and there exists a constant $C=C(\Omega,C_{\mathbb A},n)>0$ such that
\begin{align}
\label{estimate-Dnh}
\|{\bf u}\|_{H^1(\Omega )^n}+\|\pi \|_{L^2(\Omega )/\R}\leq
C\left(\|\bs{\mathfrak F} \|_{{H}^{-1}(\Omega )^n}+\|g\|_{L^2(\Omega )}
+\|\bs\varphi\|_{H^{\frac{1}{2}}(\partial \Omega )^n}
\right).
\end{align}
\end{theorem}
\begin{proof}
Let ${\bf v}\in {H}^1(\Omega)^n$ be the function given by Lemma~\ref{B-D}.
For the velocity-pressure couple $({\bf v},0)$, let us also define
\begin{align}\label{checkLD}
\check{\bs{\mathfrak F}}:={\MG-}\boldsymbol{\mathcal L}({\bf v},0)=-\boldsymbol{\mathfrak L}{\bf v}\in {H}^{-1}(\Omega)^n
\end{align}
(cf. notations for $\boldsymbol{\mathfrak L}{\bf v}$ in \eqref{Stokes-0} and \eqref{L-oper}).
Then  the fully non-homogeneous Dirichlet problem \eqref{Dirichlet-nonhomog} reduces to the following Dirichlet problem with homogeneous Dirichlet condition, for the new function ${\bf w} :={\bf u} -{\bf v}$,
\begin{align}
\label{Dirichlet-homog-w}
\left\{\begin{array}{ll}
\boldsymbol{\mathcal L}({\bf w},\pi )=-(\boldsymbol{\mathfrak F}-\check{\bs{\mathfrak F}}),\quad
{\rm{div}}\, {\bf w}=0 & \mbox{ in } \Omega \,,\\
\gamma _{_{\Omega }}{\bf w}=0 & \mbox{ on } \partial \Omega \,.
\end{array}
\right.
\end{align}

Theorem \ref{lemma-a47-1-Stokes} implies that the Dirichlet problem \eqref{Dirichlet-homog-w} has a unique solution $({\bf w},\pi)$ in the space $({\bf u},\pi )\in {H}^1(\Omega )^n\times {L^2(\Omega )/{\mathbb R}}$
and depends continuously on the given data of this problem.
Finally, the well-posedness of problem \eqref{Dirichlet-homog-w}
implies that the couple
$({\bf u}={\bf w} + {\bf v},\ \pi)$
determines a solution of the full non-homogeneous Dirichlet problem \eqref{Dirichlet-nonhomog} in the space
${H}^1(\Omega )^n\times {L^2(\Omega )/{\mathbb R}}$, and estimate \eqref{estimate-Dnh} holds.
This solution is unique by the uniqueness statement in Theorem  \ref{lemma-a47-1-Stokes}.
\qed
\end{proof}

\section{Dirichlet-transmission problems for the anisotropic compressible Stokes system in bounded Lipschitz domains with transversal interfaces}
Mitrea and Wright \cite{M-W} obtained well-posedness results in Sobolev and Besov spaces for transmission problems for the isotropic Stokes system with constant coefficients in Lipschitz domains in ${\mathbb R}^n$  (see also the references therein).
Various transmission problems for the anisotropic Stokes system  in Lipschitz domains in ${\mathbb R}^n$, $n\geq 3$, with internal interface and homogeneous conditions for traces, have been studied in \cite{K-M-W-2} by using both variational and potential approaches (see also \cite{KMW-DCDS2021} and \cite{KMW-LP}).

In this section we show the well-posedness of boundary value problems of Dirichlet-transmission type for the anisotropic Stokes system in a compressible framework in a bounded Lipschitz domain with transversal Lipschitz interfaces satisfying the following assumption.
\begin{assumption}
\label{interface-Sigma}
Let $n\!\geq \!2$ and $\Omega \subset \mathbb R^n$ be a bounded Lipschitz domain with connected boundary $\partial \Omega $. The domain $\Omega $ is divided into two disjoint Lipschitz sub-domains $\Omega ^+$ and $\Omega ^-$ by an $(n-1)$-dimensional Lipschitz open interface ${\Sigma }$, such that $\partial \Sigma =\overline\Sigma\cap\partial\Omega$ is {a non-empty} $(n-2)$-dimensional Lipschitz manifold. In this case $\overline{\Sigma }$ intersects $\partial \Omega $ transversally and $\Omega =\Omega ^+\cup \Sigma \cup \Omega ^-$.
Let the remaining boundaries $\Gamma ^+=\partial \Omega ^+\setminus \overline{\Sigma }$ and $\Gamma ^-=\partial \Omega ^-\setminus \overline{\Sigma }$ of $\partial \Omega ^+$ and $\partial \Omega ^-$, respectively, be not empty, see Fig.~\ref{Fig1}.
\end{assumption}
Thus, $\Gamma ^+$ and $\Gamma ^-$ are relatively open subsets of $\partial \Omega $.

\begin{figure}[h]
\begin{center}
        \includegraphics[width=0.45\textwidth]{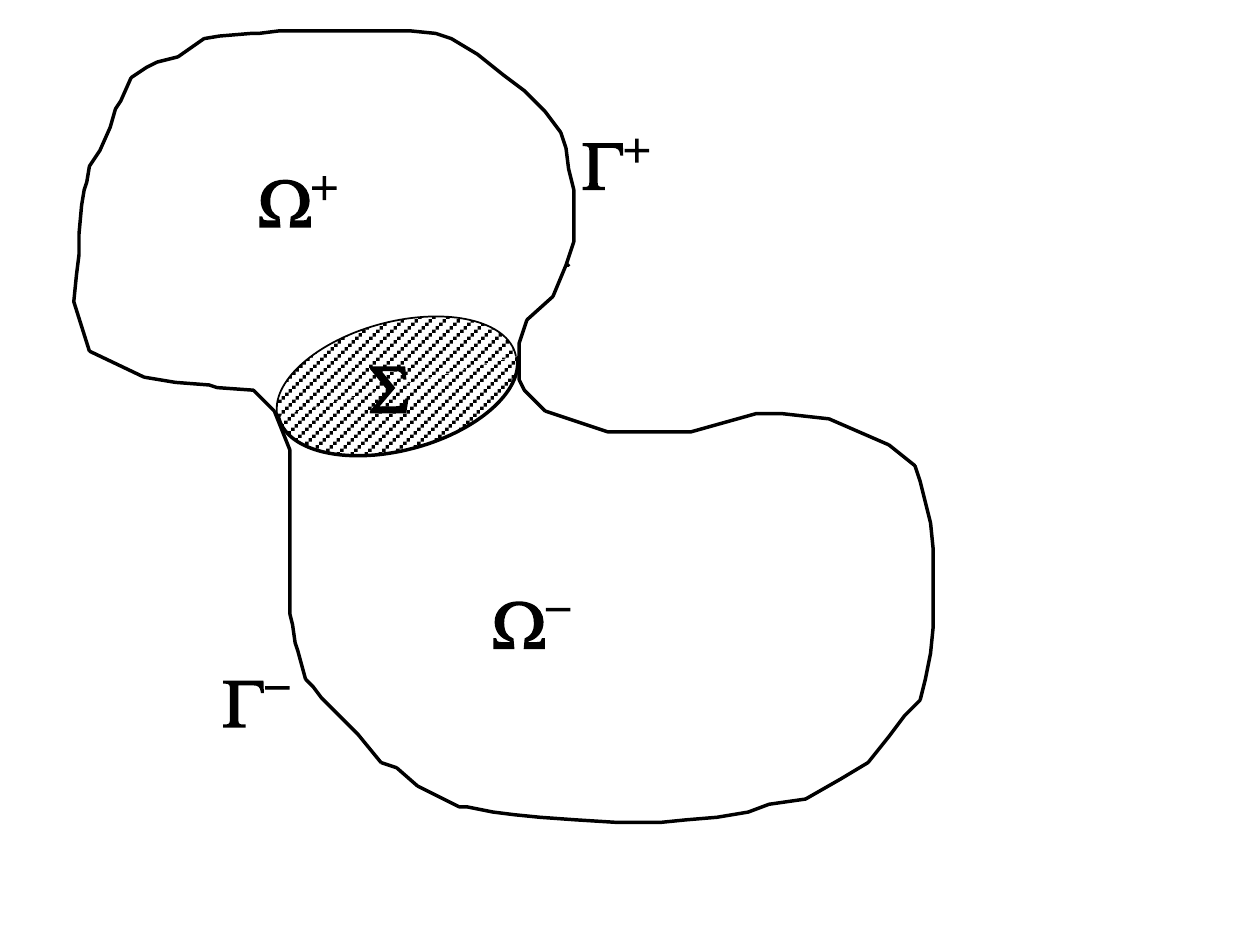}
	\caption{Bounded composite domain $\Omega =\Omega ^+\cup \Sigma \cup \Omega ^-$ with the interface $\Sigma$, for $n=3$.}
	\label{Fig1}
\end{center}
\end{figure}

\vspace{-9ex}

\subsection{\bf Sobolev spaces on bounded domains with partially vanishing traces}
Let $\Omega' \subset {\mathbb R}^n$ $(n\geq 2)$ be a bounded Lipschitz domain with connected boundary $\partial\Omega'$. Let $D$ and $N$ be relatively open subsets of $\partial\Omega'$, such that $D$ has positive $(n-1)$-Hausdorff measure, $D\cap N=\emptyset $, $\overline{D}\cup \overline{N}=\partial\Omega'$, and $\overline{D}\cap \overline{N}=\partial{D}=\partial{N}$ is an $(n-2)$-dimensional closed Lipschitz submanifold of $\partial\Omega'$.

We need the following space defined on the Lipschitz domains $\Omega'$
\begin{align}
C_{D}^{\infty }({\Omega'})^n:=\left\{\boldsymbol\varphi |_{_{\Omega'} }:\boldsymbol\varphi \in C^\infty ({\mathbb R}^n)^n,\, {\rm{supp}}\, (\boldsymbol\varphi )\cap \overline{D}=\emptyset \right\}\,,
\end{align}
and let $H_{D}^1(\Omega')^n$ be the closure of $C_{D}^{\infty }({\Omega'})^n$ in $H^1(\Omega')^n$. The space $H_{D}^1(\Omega')^n$ can be equivalently characterized as \begin{align}
\label{H-D-a}
{H}^1_{D}(\Omega')^n=\left\{{\bf v}\in {H}^1(\Omega')^n:\left(\gamma _{_{\Omega'}}{\bf v}\right)|_{_{D}}\!=\!{\bf 0}\ \mbox{ on } D\right\}
\end{align}
(cf. \cite[Corollary 3.11]{B-M-M-M}).
Let also
\begin{align}
\label{H1-div-OmegaD}
{H}^1_{D;\rm div}(\Omega ')^n
&:=\left\{{\bf w}\in{H}^1_D(\Omega ')^n:{\rm div}\, {\bf w}=0\right\}.
\end{align}

Let $\Xi $ be {a relatively} open $(n-1)$-dimensional subset of $\partial\Omega'$, e.g., $D$ or $N$. Let $r_{_{\Xi }}$ denote the operator of restriction of distributions from $\partial\Omega'$ to  $\Xi $.
Then the boundary Sobolev spaces on $N$ are defined by
\begin{align}
\label{D-N-spaces-a}
&H^{\frac{1}{2}}(\Xi )^n:=\left\{\boldsymbol\phi |_{_{\Xi }}: \boldsymbol\phi \in H^{\frac{1}{2}}(\partial\Omega')^n\right\}\,,\\
&\widetilde{H}^{\frac{1}{2}}(\Xi )^n:=\left\{{\boldsymbol\phi }\in H^{\frac{1}{2}}(\partial\Omega')^n\!:{\boldsymbol\phi }\!=\!{\bf 0}\!\ \mbox{ on } \partial\Omega'\setminus \overline {\Xi }\right\}\,,
\\
\label{duality-D-N-a}
&H^{-\frac{1}{2}}(\Xi )^n:=\big(\widetilde{H}^{\frac{1}{2}}(\Xi )^n\big)',\
\widetilde{H}^{-\frac{1}{2}}(\Xi )^n:=\big(H^{\frac{1}{2}}(\Xi )^n\big)'
\end{align}
(cf., e.g., \cite{Lean}, \cite[Definition 4.8, Theorem 5.1]{B-M-M-M}).

\begin{lemma}
\label{mixt-trace-D}
The trace operator $\gamma_{_{\Omega'}}:{H}^1_{D}(\Omega')^n\!\to \! \widetilde{H}^{\frac{1}{2}}(N)^n$ is bounded, linear and surjective, with a $($non-unique$)$ bounded, linear right inverse
$\gamma_{_{\Omega'}}^{-1}:\widetilde{H}^{\frac{1}{2}}(N)^n\!\to \!{H}^1_{D}(\Omega')^n.$
\end{lemma}
\begin{proof}
Let $\gamma_{_{\Omega'}}^{-1}:{H}^{\frac{1}{2}}(\Gamma_0)^n\!\to \!{H}^1(\Omega')^n$ be a bounded right inverse of the trace operator $\gamma_{_{\Omega'}}:{H}^1(\Omega')^n\!\to \!{H}^{\frac{1}{2}}(\Gamma_0)^n$ (see Theorem \ref{trace-operator1}).
Consequently, we have $\gamma_{_{\Omega'}}\gamma_{_{\Omega'}}^{-1}\bs\phi=\bs\phi$, for any $\bs\phi \in {H}^{\frac{1}{2}}(\Gamma_0)^n$. Therefore, if $\bs\phi\in\widetilde{H}^{\frac{1}{2}}(N)^n$, that is, $\bs\phi=0$ on $D$, then $\gamma_{_{\Omega'}}\gamma_{_{\Omega'}}^{-1}\bs\phi=0$ on $D$, and, thus, $\gamma_{_{\Omega'}}^{-1}\bs\phi\in {H}^1_{D}(\Omega')^n$. Hence, the existence of a right inverse $\gamma_{_{\Omega'}}^{-1}:{H}^{\frac{1}{2}}(\Gamma_0)^n\to {H}^1(\Omega')^n$ of the trace operator $\gamma_{_{\Omega'}}:{H}^1(\Omega')^n\to {H}^{\frac{1}{2}}(\Gamma_0)^n$ assures the existence of a bounded right inverse $\gamma_{_{\Omega'}}^{-1}:\widetilde{H}^{\frac{1}{2}}(N)^n\to {H}^1_{D}(\Omega')^n$ of the operator $\gamma_{_{\Omega'}}:{H}^1_{D}(\Omega')^n\to \widetilde{H}^{\frac{1}{2}}(N)^n$.
\qed
\end{proof}

\subsection{\bf Sobolev spaces, conormal derivatives and Green's {\Bn identity} in a bounded Lipschitz domain with a transversal Lipschitz interface}
In the sequel, $\Omega $ is a bounded Lipschitz domain in ${\mathbb R}^n$, $n\geq 2$, satisfying Assumption \ref{interface-Sigma}.
We need the following spaces defined on the domains $\Omega $, $\Omega ^+$ and $\Omega ^-$,
\begin{align}
\label{5.8}
&{H}_{\Gamma ^\pm }^1(\Omega )^n:=\left\{{\bf v}\in H^1(\Omega )^n:(\gamma _{_{\Omega }}{\bf v})|_{\Gamma ^\pm }={\bf 0} \mbox{ on } \Gamma ^\pm \right\},\\
&{H}_{\Gamma ^\pm }^1(\Omega ^\pm )^n:=\left\{{\bf v}^\pm \in H^1(\Omega ^\pm )^n: (\gamma _{_{\Omega ^\pm }}{\bf v}^\pm )|_{\Gamma ^\pm }={\bf 0} \mbox{ on } \Gamma ^\pm \right\}\,,
\end{align}
where ${\gamma _{_{\Omega ^\pm }}}:H^1(\Omega ^\pm )\to H^{\frac{1}{2}}(\partial \Omega ^\pm )$ are the trace operators acting on functions defined on the domains $\Omega ^\pm $.
The spaces ${H}_{\Gamma ^\pm }^1(\Omega ^\pm )^n$ can be equivalently described as
\begin{align}
&{H}_{\Gamma ^\pm }^1(\Omega ^\pm )^n:=\left\{{\bf v}|_{\Omega ^\pm}: {\bf v}\in H_{\Gamma ^\pm }^1(\Omega )^n\right\}\,,
\end{align}
and the space $\mathring{H}^1(\Omega )^n=\left\{{\bf w}\in H^1(\Omega ^\pm )^n:\gamma _{_\Omega }{\bf w}=0 \mbox{ on } \partial \Omega \right\}$ can be identified with the space
\begin{align}
\label{H-equiv}
\Big\{\big({\bf v}^+,{\bf v}^-\big)\in {H}_{\Gamma ^+}^1(\Omega ^+)^n\times {H}_{\Gamma ^-}^1(\Omega ^-)^n: \big(\gamma _{_{\Omega ^+}}{\bf v}^+\big)\big|_{\Sigma  }=\big(\gamma _{_{\Omega ^-}}{\bf v}^-\big)\big|_{\Sigma }\Big\}.
\end{align}
This property is an immediate consequence of Lemma \ref{extention}.

Let us next introduce the Sobolev spaces on the interface $\Sigma $ (cf., e.g.,
\cite{B-M-M-M,Lean}). First, define the space
\begin{align}
\label{D-N-spaces}
&{H^{\frac{1}{2}}(\Sigma )^n}:={\bl \left\{\bs\phi  \in L^2(\Sigma )^n: \exists \, \bs\phi  ^+\in H^{\frac{1}{2}}(\partial \Omega ^+)^n \mbox{ such that } \bs\phi  =\bs\phi  ^+\big|_{_{\Sigma }}\right\}}\,,
\end{align}
which can be identified with the space
\begin{align}
\label{D-N-spaces-1}
&\left\{\bs\phi  \in L^2(\Sigma )^n: \exists \, \bs\phi  ^-\in H^{\frac{1}{2}}(\partial \Omega ^-)^n \mbox{ such that } \bs\phi  =\bs\phi  ^-\big|_{_{\Sigma }}\right\}
\end{align}
due to the equivalence of both of them  to the space defined as in \eqref{Sobolev-boundary}, with $\Sigma $ instead of $\partial \Omega $
(see also Lemma \ref{L3} and Theorem \ref{trace-operator1}).

Let us also consider the space
\begin{align}
&{\widetilde{H}^{\frac{1}{2}}(\Sigma{\gR;}\partial\Omega^+)^n}
:=\left\{{\widetilde{\bs\phi}  ^+}\in H^{\frac{1}{2}}(\partial \Omega ^+)^n\!:{\rm{supp}}\, {\widetilde{\bs\phi}  ^+}\subseteq \overline{\Sigma }\right\}\,,
\end{align}
which, by Lemma \ref{L3}(ii), can be identified with the space
\begin{align}
\label{D-N-spaces-tilde-1}
\widetilde{H}^{\frac{1}{2}}(\Sigma{\gR;}\partial\Omega^-)^n:=
\left\{{\widetilde{\bs\phi}  ^-}\in H^{\frac{1}{2}}(\partial \Omega ^-)^n\!:{\rm{supp}}\, {\widetilde{\bs\phi}  ^-}\subseteq \overline{\Sigma }\right\}\,,
\end{align}
and both spaces inherit their norms from the spaces $H^{\frac{1}{2}}(\partial \Omega^\pm)^n$, respectively.
Let us denote by ${H}_\bullet^{\frac{1}{2}}(\Sigma)^n$ the space consisting of all functions
$\bs\phi \in{H}^{\frac{1}{2}}(\Sigma)^n$ such that their extensions by zero on
$\partial\Omega^+$, $\mathring E_{_{\Sigma\to\partial\Omega^+}}\bs\phi $, belong to $\widetilde{H}^{\frac{1}{2}}(\Sigma{\gR;}\partial\Omega^+)^n$, i.e.,
$$
{H}_\bullet^{\frac{1}{2}}(\Sigma)^n:=\{\bs\phi\in {H}^{\frac{1}{2}}(\Sigma)^n :
\mathring E_{_{\Sigma\to\partial\Omega^+}}\bs\phi \in\widetilde{H}^{\frac{1}{2}}(\Sigma{\gR;}\partial\Omega^+)^n\}.
$$
Lemma \ref{L3}(ii) shows that the space ${H}_\bullet^{\frac{1}{2}}(\Sigma)^n$ can be also defined as
$${H}_\bullet^{\frac{1}{2}}(\Sigma)^n:=\{\bs\phi\in {H}^{\frac{1}{2}}(\Sigma)^n :
\mathring E_{_{\Sigma\to\partial\Omega^-}}\bs\phi \in\widetilde{H}^{\frac{1}{2}}(\Sigma{\gR;}\partial\Omega^-)^n\}.
$$
The space ${H}_\bullet^{\frac{1}{2}}(\Sigma)^n$ can be endowed with the norm
$$\|\bs\phi\|_{{H}_\bullet^{\frac{1}{2}}(\Sigma)^n}
=\max\left\{\|\mathring E_{_{\Sigma\to\partial\Omega^+}}\bs\phi\|_{{H}^{\frac{1}{2}}(\partial\Omega^+)^n},
\|\mathring E_{_{\Sigma\to\partial\Omega^-}}\bs\phi\|_{{H}^{\frac{1}{2}}(\partial\Omega^-)^n}\right\}.
$$
Note that the operators of extension by zero
$\mathring E_{_{\Sigma\to\partial\Omega^\pm}}: {H}_\bullet^{\frac{1}{2}}(\Sigma)^n \to\widetilde{H}^{\frac{1}{2}}(\Sigma{\gR;}\partial\Omega^\pm)^n$
are continuous,
and due to, e.g., \cite[Theorem 2.10(i)]{Mikh} also surjective, implying that
the space ${H}_\bullet^{\frac{1}{2}}(\Sigma)^n$ can be identified with the spaces $\widetilde{H}^{\frac{1}{2}}(\Sigma{\gR;}\partial\Omega^\pm)^n$.
In addition,
by Theorem~\ref{Tbl}
the space ${H}_\bullet^{\frac{1}{2}}(\Sigma)^n$ can be characterized as the weighted space ${H}_{00}^{\frac{1}{2}}(\Sigma)$ consisting of functions $\bs\phi\in\mathring{H}^{\frac{1}{2}}(\Sigma)^n$,  such that $\delta^{-\frac{1}{2}}\bs\phi\in L^2(\Sigma)^n$, where $\delta(x)$ is the distance from $x\gR\in\Sigma$ to the boundary $\partial\Sigma$.
The counterpart of ${H}_{00}^{\frac{1}{2}}(\cdot)$ on smooth domains in $\R^n$ has been considered in \cite[Chapter 1, Theorem 11.7]{LiMa1} and on Lipschitz domains
 in Corollary 1.4.4.10 in \cite{Grisvard1985}.
Note also that the space
${H}_\bullet^{\frac{1}{2}}(\cdot)$ is similar to {the space} $L_{s,z}^p(\cdot)$ in \cite[Eq. (2.212)]{MM2013Spr}, cf. also \cite[p.3]{Mikh-18}.

\begin{lemma}
\label{gamma-Sigma-surj}
The operator $\gamma _{_{\Sigma}} :\mathring{H}^1(\Omega )^n\to {H}_\bullet^{\frac{1}{2}}(\Sigma )^n$ given by
\begin{align}
\label{gamma-Sigma}
\gamma _{_{\Sigma}}{\bf v}:=\left(\gamma _{_{\Omega ^+}}\left({\bf v}|_{_{\Omega ^+}}\right)\right)\big|_{\Sigma }=\left(\gamma _{_{\Omega ^-}}\left({\bf v}|_{_{\Omega ^-}}\right)\right)\big|_{\Sigma }\,, \ \ \forall \ {\bf v}\in \mathring{H}^1(\Omega )^n\,,
\end{align}
is linear, bounded and surjective.
\end{lemma}
\begin{proof}
The linearity and boundedness of the operator $\gamma _{_{\Sigma}}$ are immediate consequences of the linearity and boundedness of the trace operators
\begin{align}\label{E4.17}
\gamma _{_{\Omega ^\pm }}:H_{\Gamma ^\pm }^1(\Omega ^\pm )^n\to \widetilde{H}^{\frac{1}{2}}(\Sigma{\gR;}\partial\Omega^\pm)^n,
\end{align}
 cf. Lemma \ref{mixt-trace-D}.
The equality of restrictions to $\Sigma$ of the traces from $\Omega^+$ and $\Omega^-$ in \eqref{gamma-Sigma} follows from the inclusion
${\bf v}\in \mathring{H}^1(\Omega )^n$, see Lemma~\ref{extention}(ii).
The surjectivity of operators \eqref{E4.17} implies the surjectivity of the operator
$\gamma _{_{\Sigma}} :\mathring{H}^1(\Omega )^n\to {H}_\bullet^{\frac{1}{2}}(\Sigma )^n$.
To this end, let $\boldsymbol \phi \in {H}_\bullet^{\frac{1}{2}}(\Sigma )^n$.
Then  $\mathring E_{_{\Sigma\to\partial\Omega^\pm}}\bs\phi \in \widetilde{H}^{\frac{1}{2}}(\Sigma{\gR;}\partial\Omega^\pm)^n$ and by Lemma \ref{mixt-trace-D}, there exist ${\bf v}^\pm \in H_{\Gamma ^\pm }^1(\Omega ^\pm )^n$ such that
$\gamma _{_{\Omega^\pm}}{\bf v}^\pm =\mathring E_{_{\Sigma\to\partial\Omega^\pm}}\bs\phi$ on $\partial\Omega^\pm$.
Hence, we obtain that $\gamma _{_{\Sigma}}{\bf v}^+=\gamma _{_{\Sigma}}{\bf v}^-$ on $\Sigma $. According to Lemma \ref{extention}(i) there exists a unique function ${\bf v}\in {H}^1(\Omega )^n$ such that ${\bf v}|_{\Omega ^\pm }={\bf v}^\pm $. Moreover, since $\gamma _{_{\Omega ^\pm }}{\bf v}^\pm =0$ on $\Gamma ^\pm $, we deduce that $\gamma _{_{\Omega }}{\bf v}=0$ (a.e.) on $\partial \Omega $, and hence ${\bf v}\in \mathring{H}^1(\Omega )^n$.
\qed
\end{proof}

Lemma~\ref{gamma-Sigma-surj} implies that the space ${H}_\bullet^{\frac{1}{2}}(\Sigma )^n$ can be also characterised as
\begin{align}
{H}_\bullet^{\frac{1}{2}}(\Sigma )^n=\Big\{{\bs\phi }\in L^2(\Sigma ): \, & {\Bl \exists \, {\bf v}\in \mathring{H}^1(\Omega )^n \mbox{ such that }}\nonumber\\
&{\Bl {\bs\phi }=\left(\gamma _{_{\Omega ^+}}\left({\bf v}|_{_{\Omega ^+}}\right)\right)\big|_{\Sigma }=\left(\gamma _{_{\Omega ^-}}\left({\bf v}|_{_{\Omega ^-}}\right)\right)\big|_{\Sigma }}\Big\}\,.
\end{align}

In addition, we consider the spaces
\begin{align}
&H^{-\frac{1}{2}}(\Sigma )^n:=\big({H}_\bullet^{\frac{1}{2}}(\Sigma )^n\big)',\
\widetilde{H}^{-\frac{1}{2}}(\Sigma )^n:=\big(H^{\frac{1}{2}}(\Sigma )^n\big)'\,.
\end{align}

In Appendix \ref{A-GCN} we provide a definition of the generalized conormal derivative, associated with anisotropic Stokes operator $\boldsymbol{\mathcal L}$, on the entire boundary of the domain. If we need the conormal derivative only on a part of the boundary of the domain, we do not not need the extension of the PDE right hand side to the `tilde-space' on the rest of the boundary.
To this end, we consider the following counterpart of Definition \ref{conormal-derivative-var-Brinkman} in the case of
$\tilde{\bf f}^{\pm} \in \left(H_{\Gamma ^\pm }^1({\Omega ^\pm })^n\right)'$.
\begin{defn}
\label{conormal-derivative-var-Brinkman-pm}
Let Assumption $\ref{interface-Sigma}$ and condition \eqref{Stokes-1} hold. Let
\begin{align}
\label{Hu-pi-f}
{\pmb{H}}_{\Gamma ^\pm }^1(\Omega ^{\pm},{\boldsymbol{\mathcal L}}):=\Big\{({\bf u}^{\pm},\pi^{\pm},\tilde{\bf f}^{\pm})
&\in H^1({\Omega ^\pm })^n\times L^2({\Omega ^\pm })\times \left(H_{\Gamma ^\pm }^1({\Omega ^\pm })^n\right)':\nonumber\\
&{\boldsymbol{\mathcal L}}({\bf u}^\pm ,\pi ^\pm )=\tilde{\bf f}^{\pm}|_{\Omega ^\pm } \mbox{ in } {\Omega ^\pm }\Big\}.
\end{align}
If $({\bf u}^{\pm},\pi^{\pm},{\tilde{\bf f}}^{\pm})\in {\pmb{H}}_{\Gamma ^\pm }^1(\Omega ^{\pm},{\boldsymbol{\mathcal L}})$, then the formula
 \begin{align}
\label{conormal-derivative-var-Brinkman-3pm}
&\left\langle \big({\bf t}_{\Omega ^\pm}({\bf u}^\pm ,\pi^\pm ;\tilde{\bf f}^\pm )\big)|_{_{\Sigma }},{\bs\Phi}^\pm \right\rangle _{_{\Sigma }}:=
\left\langle{a_{ij}^{\alpha \beta }E_{j\beta }({\bf u}^\pm),E_{i\alpha }} (\gamma^{-1}_{_{\Omega ^\pm }}{\bs\Phi}^\pm )\right\rangle _{\Omega ^\pm }\nonumber\\
&\hspace{3em}-\left\langle {\pi }^\pm ,{\rm{div}}(\gamma^{-1}_{_{\Omega ^\pm }}{\bs\Phi}^\pm )\right\rangle _{\Omega ^\pm }
+\left\langle \tilde{\bf f}^\pm ,\gamma^{-1}_{_{\Omega ^\pm }}{\bs\Phi}^\pm \right\rangle _{{\Omega ^\pm }}, \ \forall\, {\bs\Phi}^\pm \in {H}_\bullet^{\frac{1}{2}}(\Sigma )^n\,,
\end{align}
defines the generalized conormal derivatives $\big({\bf t}_{\Omega ^\pm}({\bf u}^\pm ,\pi^\pm ;\tilde{\bf f}^\pm )\big)|_{_{\Sigma }}\in H^{-\frac{1}{2}}(\Sigma )^n$,
where $\gamma^{-1}_{_{\Omega ^\pm }}:{H}_\bullet^{\frac{1}{2}}(\Sigma )^n\to H_{\Gamma ^\pm }^{1}({\Omega ^\pm })^n$ are bounded right {\Bl inverses} of the trace operators
$\gamma _{_{\Omega ^\pm }}:H_{\Gamma ^\pm }^{1}(\Omega ^\pm )^n\to {H}_\bullet^{\frac{1}{2}}(\Sigma )^n$.
\end{defn}
Note that, in view of Lemma \ref{mixt-trace-D}, all duality pairings in formula  \eqref{conormal-derivative-var-Brinkman-3pm} are well-defined. Moreover,
as in Lemma \ref{lem-add1}, we have the following result, whose proof is omitted for the brevity  (cf. \cite[Proposition 8.1]{M-M1} for the Laplace operator, \cite[Lemma 7.6]{K-W1} for extensions to compact Riemannian manifolds, and \cite[Definition 7.1]{B-M-M-M} in the case of higher order elliptic operators).
\begin{lemma}
\label{lemma-add-new-1H}
Let Assumption $\ref{interface-Sigma}$ and conditions \eqref{Stokes-1} and \eqref{Stokes-sym} hold.
\begin{itemize}
\item[$(i)$]
The generalized conormal derivative operators
${{\bf t}_{\Omega ^\pm }}:{\pmb{H}}_{\Gamma ^\pm }^1(\Omega ^{\pm},{\boldsymbol{\mathcal L}})
\to H^{-\frac{1}{2}}(\Sigma)^n$
are linear and bounded, and definition \eqref{conormal-derivative-var-Brinkman-3pm} does not depend on the particular choice of  right inverses
$\gamma^{-1}_{_{\Omega ^\pm }}:{H}_\bullet^{\frac{1}{2}}(\Sigma )^n
\to H_{\Gamma ^\pm }^{1}({\Omega ^\pm })^n$
of the trace operators
$\gamma _{_{\Omega ^\pm }}:H_{\Gamma ^\pm }^{1}(\Omega^\pm)^n\to {H}_\bullet^{\frac{1}{2}}(\Sigma )^n$.
\item[$(ii)$]
Let $({\bf u}^{\pm},\pi^{\pm},{\tilde{\bf f}}^{\pm})\in {\pmb{H}}_{\Gamma ^\pm }^1(\Omega ^{\pm},{\boldsymbol{\mathcal L}})$.
Let
$\pi \in L^2(\Omega )$ and $\tilde{\bf f}\in H^{-1}(\Omega )$ be such that
\begin{align}
\label{u-pi-f}
\pi|_{\Omega ^\pm }=\pi ^\pm ,\ \ \tilde{\bf f}=\tilde{\bf f}^++\tilde{\bf f}^-\,.
\end{align}
Then the following first Green identities hold
\begin{align}
\label{Green-particular+-}
\big\langle &\big({\bf t}_{\Omega ^\pm}({\bf u}^\pm,\pi^\pm;\tilde{\bf f}^\pm)\big)|_{_{\Sigma }}\,
,\gamma _{_{\Sigma}}{\bf w}^\pm\big\rangle _{_{\Sigma }}
=\left\langle a_{ij}^{\alpha \beta }E_{j\beta }({\bf u}^\pm),E_{i\alpha }({\bf w}^\pm)\right\rangle _{\Omega^\pm}
\nonumber\\
&\hspace{7em}
-\langle {\pi},{\rm{div}}\, {\bf w}^\pm \rangle_{\Omega^\pm}+\langle \tilde{\bf f},{\bf w}^\pm\rangle_{\Omega^\pm}\,,
\quad \forall\ {\bf w}^\pm\in H_{\Gamma ^\pm }^{1}({\Omega ^\pm })^n\,,
\end{align}
and hence
\begin{align}
\label{Green-particular}
\big\langle &\big({\bf t}_{\Omega ^+}({\bf u}^+,\pi^+;\tilde{\bf f}^+)\big)|_{_{\Sigma }}\,
+\, \big({{\bf t}_{\Omega ^-}}({\bf u}^-,\pi^-;\tilde{\bf f}^-)\big)|_{_{\Sigma }}\,,\gamma _{_{\Sigma}}{\bf w}\big\rangle _{_{\Sigma }}
\nonumber\\
&=\left\langle a_{ij}^{\alpha \beta }E_{j\beta }({\bf u}^+),E_{i\alpha }({\bf w})\right\rangle _{\Omega^+}
+\left\langle a_{ij}^{\alpha \beta }E_{j\beta }({\bf u}^-),E_{i\alpha }({\bf w})\right\rangle _{\Omega^-}
\nonumber\\
&\hspace{10em}-\langle {\pi},{\rm{div}}\, {\bf w} \rangle_{\Omega }+\langle \tilde{\bf f},{\bf w}\rangle_{\Omega }\,,
\quad \forall\ {\bf w}\in \mathring{H}^{1}(\Omega )^n\,.
\end{align}
\end{itemize}
\end{lemma}

Note that the existence of a function
$\pi \in L^2(\Omega )$ as in \eqref{u-pi-f} follows immediately from Lemma \ref{extention}, while Lemma \ref{identification} shows that $\tilde{\bf f}$ defined in \eqref{u-pi-f} belongs indeed to the space $H^{-1}(\Omega )^n$.

\subsection{\bf Dirichlet-transmission problem with homogeneous Dirichlet conditions}
\label{S4.2.1}
First, we analyze the following Dirichlet-transmission problem for the anisotropic Stokes system in $\Omega $ with {homogeneous Dirichlet conditions}
\begin{equation}
\label{Dirichlet-var-Stokes-a}
\left\{
\begin{array}{ll}
\boldsymbol{\mathcal L}({\bf u}^+,\pi ^+)=\widetilde{\bf f}^+|_{\Omega ^+},\ {\rm{div}}\, {\bf u}^+={g|_{\Omega ^+}}
& \mbox{ in } \Omega ^+\,,\\
\boldsymbol{\mathcal L}({\bf u}^-,\pi ^-)=\widetilde{\bf f}^-|_{\Omega ^-},\ {\rm{div}}\, {\bf u}^-={g|_{\Omega ^-}}
& \mbox{ in } \Omega ^-\,,\\
{(\gamma _{_{\Omega ^+}}{\bf u}^+)|_{_{\Sigma }}}=(\gamma _{_{\Omega ^-}}{\bf u}^{-})|_{_{\Sigma }} &  \mbox{ on } \Sigma \,,\\
({\bf t}_{\Omega ^+}({\bf u}^+,\pi ^+;\widetilde{\bf f}^+))|_{_{\Sigma }}\, +\, ({\bf t}_{\Omega ^-}({\bf u}^-,\pi ^-;\widetilde{\bf f}^-))|_{_{\Sigma }}
={\bs\psi_{_{\Sigma }}} &  \mbox{ on } \Sigma \,,\\
({\gamma _{_{\Omega ^+}}}{\bf u}^+)|_{\Gamma ^+}={\bf 0} &  \mbox{ on } \Gamma ^+\,,\\
({\gamma _{_{\Omega ^-}}}{\bf u}^-)|_{\Gamma ^-}={\bf 0} &  \mbox{ on } \Gamma ^-
\end{array}\right.
\end{equation}
and the given data $(\widetilde{\bf f}^+,\widetilde{\bf f}^-,{g},{\bs\psi_{_{\Sigma }}}) \in\mathfrak Y^0$. The space
\begin{align}
\mathfrak Y^0:=
\left({H}_{\Gamma ^+}^1(\Omega ^+)^n\right)'{\times \left({H}_{\Gamma ^-}^1(\Omega ^-)^n\right)'}
\times {L^2_0(\Omega)}\times H^{-\frac{1}{2}}(\Sigma )^n
\end{align}
is endowed with the norm
\begin{align*}
\|(\widetilde{\bf f}^+,\widetilde{\bf f}^-,g,{\bs\psi_{_{\Sigma }}})\|_{\mathfrak Y^0}:=
\|\widetilde{\bf f}^+\|_{\left({H}_{\Gamma ^+}^1(\Omega ^+)^n\right)'}
&+\|\widetilde{\bf f}^-\|_{\left({H}_{\Gamma ^-}^1(\Omega ^-)^n\right)'}\\
&+\|g\|_{L^2(\Omega )}
+\|{\bs\psi_{_{\Sigma }}}\|_{H^{-\frac{1}{2}}(\Sigma )^n}.
\end{align*}
Note that the conormal derivative operators ${\bf t}_{\Omega^+}$ and ${\bf t}_{\Omega^-}$,
as introduced in Definition~\ref{conormal-derivative-var-Brinkman}, correspond to the outward unit normal vectors to $\Omega^+$ and $\Omega ^-$, respectively, that have opposite directions on $\Sigma $.
However, if one would consider the conormal derivatives with respect to unit normal vectors of the same direction on $\Sigma$, then the sum in the corresponding transmission condition in \eqref{Dirichlet-var-Stokes-a} would be replaced by the difference, leading to the {\it jump} of the conormal derivatives as, e.g., in \cite{K-M-W-2,KMW-DCDS2021,KMW-LP}.

We show that \eqref{Dirichlet-var-Stokes-a} has a unique solution $({\bf u}^+,\pi ^+,{\bf u}^-,\pi ^-)$ in the space
\begin{align}
\label{Omega-pm}
{\mathfrak X}_{\Omega ^+,\Omega ^-}:=\big\{({\bf v}^+,q^+,{\bf v}^-,q^-):\
&{\bf v}^+\in {H}^1(\Omega ^+)^n,\
{\bf v}^-\in {H}^1(\Omega ^-)^n,\nonumber\\
& q^+=q|_{\Omega ^+}, \ q^-=q|_{\Omega ^-},\  q\in L^2(\Omega)/{\mathbb R}
\big\}\,,
\end{align}
endowed with the norm
$$\|({\bf v}^+,q^+,{\bf v}^-,q^-)\|_{{\mathfrak X}_{\Omega ^+,\Omega ^-}}=
\|{\bf v}^+\|_{{H}^1(\Omega ^+)^n}
+\|{\bf v}^-\|_{{H}^1(\Omega ^-)^n}
+\| q\|_{L^2(\Omega)/{\mathbb R}}.$$
(The choice of the space $L^2(\Omega )/{\mathbb R}$ for the pressure is only for convenience, and one may consider the space $L_0^2(\Omega )$ as well.)

Let $({\bf u}^+,\pi ^+,{\bf u}^-,\pi ^-)\in\Gr {\mathfrak X}_{\Omega ^+,\Omega ^-}$
and ${\bf u}^+$ and ${\bf u}^-$ satisfy the homogeneous interface condition for traces in \eqref{Dirichlet-var-Stokes-a},
${(\gamma _{_{\Omega ^+}}{\bf u}^+)|_{_{\Sigma }}}-{(\gamma _{_{\Omega ^-}}{\bf u}^{-})|_{_{\Sigma }}}={\bf 0}$ on $\Sigma $.
Then Lemma \ref{extention} implies that there exists a unique pair $({\bf u},\pi)\in  H^1(\Omega )^n\times  L^2(\Omega )/{\mathbb R}$ such that
\begin{align}
\label{u}
{\bf u}|_{\Omega ^+}={\bf u}^+,\quad
{\bf u}|_{\Omega^-}={\bf u}^-,\quad
\pi|_{\Omega ^+} =\pi ^+,\quad
\pi|_{\Omega^-} =\pi^-\,.
\end{align}
Assuming that ${\bf u}^+$ and ${\bf u}^-$ also satisfy the homogeneous Dirichlet condition in \eqref{Dirichlet-var-Stokes-a} 
we have that ${\bf u}\in \mathring{H}^1(\Omega )^n$.
Therefore, $({\bf u},\pi)\in
\mathring{H}^1(\Omega )^n\times L^2(\Omega )/\mathbb R$.

Note that the membership of $\widetilde{\bf f}^+$ to $\left({H}_{\Gamma ^+}^1(\Omega ^+)^n\right)'$ and the identification of this space with the space defined {by \eqref{equivalence-set} in Lemma \ref{identification}} imply that $\widetilde{\bf f}^+$ can be considered also as a distribution in $H^{-1}(\Omega )^n$.
Similarly, the assumption $\widetilde{\bf f}^-\in \left({H}_{\Gamma ^-}^1(\Omega ^-)^n\right)'$ implies that $\widetilde{\bf f}^-$ can be considered as a distribution in $H^{-1}(\Omega )^n$.

Let also $\boldsymbol{\mathfrak F}\in {H}^{-1}(\Omega )^n$ be such that
\begin{align}
\label{F-a}
\langle \boldsymbol{\mathfrak F},{\bf v}\rangle _{\Omega }&:=-\big\langle \tilde{\bf f}^+,{\bf v}\big\rangle _{\Omega ^+}\!-\!\big\langle \tilde{\bf f}^-,{\bf v}\big\rangle _{\Omega ^-}+\langle {\bs\psi_{_{\Sigma }}},\gamma _{_{\Sigma }}{\bf v}\rangle _{\Sigma }\nonumber\\
&=-\big\langle \tilde{\bf f}^++\tilde{\bf f}^-,{\bf v}\big\rangle _{\Omega }+\langle \gamma _{_{\Sigma }}^*{\bs\psi_{_{\Sigma }}},{\bf v}\rangle _{\Omega }\,,\ \ \forall \, {\bf v}\in \mathring{H}^{1}(\Omega )^n\,,
\end{align}
that is,
$\boldsymbol{\mathfrak F}=-(\tilde{\bf f}^++\tilde{\bf f}^-)+\gamma _{_\Sigma }^*{\bs\psi_{_{\Sigma }}}$.
Note that ${\gamma _{_{\Sigma }}^*}:H^{-\frac{1}{2}}(\Sigma )^n\to {H}^{-1}(\Omega )^n$ is {the adjoint of the trace operator $\gamma _{_{\Sigma }}:\mathring{H}^{1}(\Omega )^n\to {\widetilde{H}^{\frac{1}{2}}(\Sigma )^n}$ defined by \eqref{gamma-Sigma}, and the support of $\gamma _{_{\Sigma }}^*{\bs\psi_{_{\Sigma }}}$ is a subset of $\overline\Sigma $}.

Now, we can show the well-posedness of the Dirichlet-transmission problem \eqref{Dirichlet-var-Stokes-a} (see also \cite[Theorem 1.2]{Angot-2}, \cite[Corollary 3.1]{Angot-3} for interface problems involving the Stokes and Brinkman systems
in Lipschitz domains with transversal interfaces and jump
conditions in the isotropic case \eqref{isotropic}).
\begin{theorem}
\label{T-p}
Let Assumption $\ref{interface-Sigma}$  and conditions \eqref{Stokes-1}-\eqref{mu} hold.
\begin{enumerate}
\item[$(i)$]
Then for all $(\widetilde{\bf f}^+,\widetilde{\bf f}^-,g,{\bs\psi_{_{\Sigma }}})\in \mathfrak Y^0$ the Dirichlet-transmission problem \eqref{Dirichlet-var-Stokes-a} has a unique solution $({\bf u}^+,\pi^+,{\bf u}^-,\pi^-)$ in the space ${\mathfrak X}_{\Omega^+,\Omega^-}$, and there exists a positive constant $C=C(\Omega^+,\Omega^-, C_{\mathbb A},n)$ such that
\begin{align}\label{E3.27}
&\|({\bf u}^+,\pi^+,{\bf u}^-,\pi^-)\|_{{\mathfrak X}_{\Omega^+,\Omega^-}}
\le C\|(\widetilde{\bf f}^+,\widetilde{\bf f}^-,g,{\bs\psi_{_{\Sigma }}})\|_{\mathfrak Y^0}
\end{align}
\item[$(ii)$]
The solution can be represented in the form
$({\bf u}^+,\pi^+,{\bf u}^-,\pi^-)=\bs{\mathfrak U}^0(\widetilde{\bf f}^+,\widetilde{\bf f}^-,g,{\bs\psi_{_{\Sigma }}}),$
where $\bs{\mathfrak U}^0:\mathfrak Y^0\to {\mathfrak X}_{\Omega^+,\Omega^-}$ is a linear continuous operator.
\end{enumerate}
\end{theorem}
\begin{proof}
Let us prove that the Dirichlet-transmission problem \eqref{Dirichlet-var-Stokes-a}
with the unknowns $\left({\bf u}^+,\pi ^+,{\bf u}^-,\pi ^-\right)\in {\mathfrak X}_{\Omega ^+,\Omega ^-}$ is equivalent,
in the sense of relations \eqref{u}, to the variational problem \eqref{transmission-S-variational-dl-3-equiv-0-2} with the unknowns
$({\bf u},\pi )\in\Gr\mathring{H}^1(\Omega )^n\times L^2(\Omega )/\mathbb R$, and
with $\boldsymbol{\mathfrak{F}}\in H^{-1}(\Omega )^n$ given by \eqref{F-a}.

First, assume that $\left({\bf u}^+,\pi ^+,{\bf u}^-,\pi ^-\right)\in {\mathfrak X}_{\Omega ^+,\Omega ^-}$ satisfies the Dirichlet-transmission problem \eqref{Dirichlet-var-Stokes-a}.
Let $({\bf u},\pi )\in \mathring{H}^1(\Omega )^n\times {\bl L_0^2(\Omega )}$
be the pair defined by formula \eqref{u} (cf. Lemma \ref{extention}). Then the first equation of variational problem \eqref{transmission-S-variational-dl-3-equiv-0-2} follows from the Green identity \eqref{Green-particular} and relation \eqref{F-a} for $\boldsymbol{\mathfrak F}$.
The second equation in \eqref{transmission-S-variational-dl-3-equiv-0-2} follows from the equations ${\rm{div}}\, {\bf u}^\pm =g|_{\Omega ^\pm }$ in $\Omega ^\pm $ and the inclusion ${\rm{div}}\, {\bf u}\in L^2(\Omega)$.

Conversely, assume that $({\bf u},\pi )\in \mathring{H}^1(\Omega )^n\times  L^2(\Omega )/{\mathbb R}$ satisfies the variational problem \eqref{transmission-S-variational-dl-3-equiv-0-2}
and let $({\bf u}^\pm ,\pi ^\pm )=({\bf u}|_{\Omega ^\pm },\pi |_{\Omega ^\pm })$.
Then the first equation in \eqref{transmission-S-variational-dl-3-equiv-0-2} can be written as
\begin{align}
\label{conormal-derivative-particular-a}
&\big\langle a_{ij}^{\alpha \beta }E_{j\beta }({\bf u}^+),E_{i\alpha }({\bf w}^+)\big\rangle _{\Omega ^+}-\big\langle \pi ^+,{\rm{div}}\, {\bf w}^+\big\rangle _{\Omega ^+}\nonumber\\
&+\big\langle a_{ij}^{\alpha \beta }E_{j\beta }({\bf u}^-),E_{i\alpha }({\bf w}^-)\big\rangle _{\Omega ^-}-\big\langle \pi ^-,{\rm{div}}\, {\bf w}^-\big\rangle _{\Omega ^-}\nonumber\\
&-\big\langle \tilde{\bf f}^+,{\bf w}\big\rangle _{\Omega ^+}-\big\langle \tilde{\bf f}^-,{\bf w}\big\rangle _{\Omega ^-}+\langle {\bs\psi_{_{\Sigma }}},\gamma _{_{\Sigma }}{\bf w}\rangle _{\Sigma }=0\ \ \forall \, {\bf w}\in \mathring{H}^1(\Omega )^n\,.
\end{align}
Since
the spaces
${\mathcal D}(\Omega ^\pm )^n$ are subspaces of $\mathring{H}^{1}(\Omega)^n$,
the (distributional form of) the anisotropic Stokes equation in \eqref{Dirichlet-var-Stokes-a}, in each of the domains $\Omega ^+$ and $\Omega ^-$, follows from equation \eqref{conormal-derivative-particular-a} written for all ${\bf w}\in {\mathcal D}(\Omega ^+)^n$ and
${\bf w}\in {\mathcal D}(\Omega ^-)^n$, respectively.
A similar argument
yields that the second variational equation in \eqref{transmission-S-variational-dl-3-equiv-0-2} implies the divergence equation ${\rm{div}}\, {\bf u}^\pm =g^\pm $ in $\Omega ^\pm $.
Thus,
$\left({\bf u}^+,\pi ^+,{\bf u}^-,\pi ^-\right)$
satisfies the anisotropic Stokes system in $\Omega ^+\cup \Omega ^-$, 
the Dirichlet boundary condition $(\gamma _{_{\Omega ^\pm }}{\bf u}^\pm )|_{_{\Gamma ^\pm }}={\bf 0}$ on $\Gamma ^\pm $, and the interface condition $(\gamma _{_{\Omega ^+}}{\bf u}^+)|_{_{\Sigma }}={(\gamma _{_{\Omega ^-}}{\bf u}^{-})|_{_{\Sigma }}}$ on $\Sigma $.
Then substituting \eqref{conormal-derivative-particular-a} into the Green identity \eqref{Green-particular}, we obtain the equation
\begin{align}
\label{Sigma-ab}
\big\langle \big({\bf t}_{\Omega ^+}({\bf u}^+,\pi^+;\tilde{\bf f}^+)
{+}{{\bf t}_{\Omega ^-}}({\bf u}^-,\pi^-;\tilde{\bf f}^-)\big)\big|_{\Sigma },
(\gamma _{_{\Omega }}{\bf w})|_{_{\Sigma }}\big\rangle _{_{\Sigma }}
=\big\langle {\bs\psi_{_{\Sigma }}},{\gamma _{_{\Sigma }}{\bf w}}\big\rangle _\Sigma \,.
\end{align}
In view of {Lemma \ref{gamma-Sigma-surj}}, formula \eqref{Sigma-ab} {\Gr implies}
\begin{align}
\label{Sigma-aa}
\!\!\!\!\big\langle \big({\bf t}_{\Omega ^+}({\bf u}^+,\pi^+;\tilde{\bf f}^+)
{+}{{\bf t}_{\Omega ^-}}({\bf u}^-,\pi^-;\tilde{\bf f}^-)\big)\big|_{\Sigma },\bs\phi  \big\rangle _{_{\Sigma }}\!
=\!\big\langle {\bs\psi_{_{\Sigma }}},\bs\phi  \big\rangle _\Sigma \,,\
\forall \, {\bs\phi  \!\in \! \widetilde{H}^{\frac{1}{2}}(\Sigma )^n}.
\end{align}
Therefore,
$\big({\bf t}_{\Omega ^+}({\bf u}^+,\pi^+;\tilde{\bf f}^+)
{+}{{\bf t}_{\Omega ^-}}({\bf u}^-,\pi^-;\tilde{\bf f}^-)\big)\big|_{\Sigma }
={\bs\psi_{_{\Sigma }}}$ on $\Sigma $.

Consequently, the Dirichlet-transmission problem \eqref{Dirichlet-var-Stokes-a} and the variational problem \eqref{transmission-S-variational-dl-3-equiv-0-2} are equivalent, as asserted.
By Theorem \ref{lemma-a47-1-Stokes}, the
variational problem \eqref{transmission-S-variational-dl-3-equiv-0-2} in the space
$\mathring{H}^1(\Omega )^n\times L^2(\Omega )/\mathbb R$ is well-posed.
Hence the proved equivalence implies the well-posedness of problem \eqref{Dirichlet-var-Stokes-a} in the space
${\mathfrak X}_{\Omega ^+,\Omega ^-}$,
and estimate \eqref{E3.27} follows from \eqref{estimate-1-wp-S-2} and  \eqref{F-a}.
Together with Theorem \ref{lemma-a47-1-Stokes}(iii) this also implies the representation of item (ii).
\qed
\end{proof}

\subsection{\bf Dirichlet-transmission problem with non-homogeneous interface and Dirichlet conditions}\label{S3.3.2}
Let the space $\mathfrak Y_\bullet$ consists of all elements
\begin{multline*}
(\widetilde{\bf f}^+,\widetilde{\bf f}^-,g^+,g^-,\boldsymbol \varphi_{_{\Sigma }}, {\boldsymbol \psi}_{_{\Sigma }},\bs\varphi)\in
\\
\left({H}_{\Gamma ^+}^1(\Omega ^+)^n\right)'\times \left({H}_{\Gamma ^-}^1(\Omega ^-)^n\right)'
\times  L^2(\Omega^+)\times L^2(\Omega^-)
\times {{H}_\bullet^{\frac{1}{2}}(\Sigma )^n}\!
\times \!{{H}^{-\frac{1}{2}}(\Sigma )^n}\!\times \!H^{\frac{1}{2}}(\partial \Omega )^n
\end{multline*}
such that $g^+$, $g^-$, $\bs\varphi $, and ${\boldsymbol \varphi }_{_{\Sigma }}$ satisfy the compatibility condition
\begin{align}
\label{hg}
{\int _{\Omega^+}g^+ dx + \int _{\Omega^-}g^-dx}=\int _{\partial \Omega }{\bs\varphi}\cdot \boldsymbol\nu d\sigma
+\int _{\Sigma}{\boldsymbol \varphi }_{_{\Sigma }}\cdot\boldsymbol\nu _{_\Sigma }d\sigma \,,
\end{align}
where $\boldsymbol\nu _{_{\Sigma }}$ is the unit normal to $\Sigma $ oriented from $\Omega ^+$ to $\Omega ^-$.
The space $\mathfrak Y_\bullet$ is endowed with the norm
\begin{multline*}
\|(\widetilde{\bf f}^+,\widetilde{\bf f}^-,{g^+,g^-},{\boldsymbol \varphi}_{_{\Sigma }},{\boldsymbol \psi}_{_{\Sigma }}, \bs\varphi)\|_{\mathfrak Y_\bullet}
:= \|\widetilde{\bf f}^+\|_{\left({H}_{\Gamma ^+}^1(\Omega ^+)^n\right)'}+
\|\widetilde{\bf f}^-\|_{\left({H}_{\Gamma ^-}^1(\Omega ^-)^n\right)'}\\
+\|g^+\|_{L^2(\Omega^+)}+\|g^-\|_{L^2(\Omega^-)}
+\|\bs\varphi_{_{\Sigma }}\|_{{H}_\bullet^{\frac{1}{2}}(\Sigma)^n}
+\|\bs\psi_{_{\Sigma }}\|_{H^{-\frac{1}{2}}(\Sigma )^n}
+\|\bs\varphi\|_{H^{\frac{1}{2}}(\partial\Omega^n)}.
\end{multline*}
Let us also define the space $\mathcal M_\bullet$  consisting of all elements
\begin{align*}
(g^+,g^-,{\boldsymbol \varphi}_{_{\Sigma }},{\boldsymbol \varphi})\in
L^2(\Omega^+)\times L^2(\Omega^-)\times {H}_\bullet^{\frac{1}{2}}(\Sigma )^n\times H^{\frac{1}{2}}(\partial \Omega )^n
\end{align*}
satisfying the compatibility condition \eqref{hg}, with the norm
\begin{align}
\label{M-norm}
&\|(g^+,g^-,{\boldsymbol \varphi}_{_{\Sigma }},{\boldsymbol \varphi})\|_{\mathcal M_\bullet}:=
\|g^+\|_{L^2(\Omega^+)}+\|g^-\|_{L^2(\Omega^-)}+\|{\boldsymbol \varphi}_{_{\Sigma }}\|_{{H}_\bullet^{\frac{1}{2}}(\Sigma )^n}
+\|\bs\varphi \|_{H^{\frac{1}{2}}(\partial \Omega )^n}.
\end{align}

Let us consider the following non-homogeneous Dirichlet-transmission problem
\begin{equation}
\label{Dirichlet-var-Stokes-nn}
\left\{
\begin{array}{ll}
\boldsymbol{\mathcal L}({\bf u}^+,\pi ^+)=\widetilde{\bf f}^+|_{\Omega ^+},\quad {\rm{div}}\, {\bf u}^+={g^+} & \mbox{ in } \Omega ^+\,,\\
\boldsymbol{\mathcal L}({\bf u}^-,\pi ^-)=\widetilde{\bf f}^-|_{\Omega ^-},\quad {\rm{div}}\, {\bf u}^-={g^-} & \mbox{ in } \Omega ^-\,,\\
(\gamma _{_{\Omega ^+}}{\bf u}^+)|_{_{\Sigma }}-{(\gamma _{_{\Omega ^-}}{\bf u}^{-})|_{_{\Sigma }}}={\boldsymbol \varphi}_{_{\Sigma }} & \mbox{ on } \Sigma \,,\\
{\big({\bf t}_{\Omega ^+}({\bf u}^+,\pi ^+;\widetilde{\bf f}^+)\big)|_{_{\Sigma }}}\,  +\, {\big({\bf t}_{\Omega ^-}({\bf u}^-,\pi ^-;\widetilde{\bf f}^-)\big)|_{_{\Sigma }}}={\boldsymbol \psi}_{_{\Sigma }} &  \mbox{ on } \Sigma \,,\\
({\gamma _{_{\Omega ^+}}}{\bf u}^+)|_{\Gamma ^+}={{\bs\varphi}|_{\Gamma ^+}} &  \mbox{ on } \Gamma ^+\,,\\
({\gamma _{_{\Omega ^-}}}{\bf u}^-)|_{\Gamma ^-}={{\bs\varphi}|_{\Gamma ^-}} &  \mbox{ on } \Gamma ^-\,,
\end{array}\right.
\end{equation}
with the unknown functions $({\bf u}^+,\pi ^+,{\bf u}^-,\pi ^-)$ in the space ${\mathfrak X}_{\Omega ^+,\Omega ^-}$ defined in \eqref{Omega-pm},
and with the given data
$(\widetilde{\bf f}^+,\widetilde{\bf f}^-,g^+,g^-,{\boldsymbol \varphi }_{_{\Sigma }}, {\boldsymbol \psi}_{_{\Sigma }},\bs\varphi )\in \mathfrak Y_\bullet$.

In order to analyze the non-homogeneous Dirichlet-transmission problem, we need the following Bogovskii-type transmission result.
\begin{lemma}\label{L.BT}
Let Assumption $\ref{interface-Sigma}$ hold.
Then for all given data
$({g^+,g^-},\boldsymbol \varphi _{_\Sigma },{\bs\varphi})\in
{\mathcal M_\bullet}$
there exist ${\bf v}^\pm\in {H}^1(\Omega ^\pm)^n$
such that
\begin{align}
\label{BT1}
{\rm{div}}\, {\bf v}^+={g^+} & \mbox{ in } \Omega ^+,  \\
\label{BT2}
{\rm{div}}\, {\bf v}^-={g^-}  & \mbox{ in } \Omega ^- , \\
\label{BT3}
(\gamma _{_{\Omega ^+}}{\bf v}^+)|_{_{\Sigma }}-{(\gamma _{_{\Omega ^-}}{\bf v}^-)|_{_{\Sigma }}}
=\boldsymbol \varphi_{_{\Sigma }} &\mbox{ on } \Sigma \,,\\
\label{BT4}
({\gamma _{_{\Omega ^+}}}{\bf v}^+)|_{\Gamma ^+}={{\bs\varphi}|_{\Gamma ^+}} &  \mbox{ on } \Gamma ^+\,,\\
\label{BT5}
({\gamma _{_{\Omega ^-}}}{\bf v}^-)|_{\Gamma ^-}={{\bs\varphi}|_{\Gamma ^-}} &  \mbox{ on } \Gamma ^-\,,
\end{align}
and, moreover,
\begin{align}
\label{BT6}
\|{\bf v}^\pm\|_{H^1(\Omega ^\pm)^n}\leq
C_\Sigma\|(g^+,g^-,{\boldsymbol \varphi}_{_{\Sigma }},{\boldsymbol \varphi})\|_{\mathcal M_\bullet}
\end{align}
with some constant $C_\Sigma=C_\Sigma(\Omega^+,\Omega^- ,n)>0$.
\end{lemma}
\begin{proof}
Let
$\gamma^{-1}_{\Omega^\pm}: H^{\frac{1}{2}}(\partial \Omega^\pm)^n\to H^1(\Omega ^\pm )^n$
be some continuous right inverses to the corresponding trace operators and  $E_{_{\Gamma^\pm\to\partial\Omega}}:H^{\frac{1}{2}}(\Gamma^\pm)^n\to H^{\frac{1}{2}}(\partial\Omega)$ some continuous extension operators.
Let us introduce the functions
\begin{align}
\label{E3.49}
&{\bf v}^+_1:=r_{_{\Omega^+}}\gamma^{-1}_{\Omega}{\bs\varphi}
+\frac{1}{2}
\gamma^{-1}_{\Omega^+}\mathring E_{_{\Sigma\to\partial\Omega^+}}
\boldsymbol \varphi_{_{\Sigma }},\\
&{\bf v}^-_1:=r_{_{\Omega^-}}\gamma^{-1}_{\Omega}{\bs\varphi}
-\frac{1}{2}
\gamma^{-1}_{\Omega^-}\mathring E_{_{\Sigma\to\partial\Omega^-}}
\boldsymbol \varphi_{_{\Sigma }},
\end{align}
where $ \gamma^{-1}_{\Omega }: H^{\frac{1}{2}}(\partial \Omega )^n\to H^1(\Omega )^n$ and $\gamma^{-1}_{\Omega^\pm}: H^{\frac{1}{2}}(\partial \Omega^\pm)^n\to H^1(\Omega ^\pm )^n$
are continuous right inverses to the corresponding trace operators, while $r_{_{\Omega^\pm}}:H^1(\Omega )^n\to H^1(\Omega^\pm)^n$ are restriction operators to the corresponding domains, and each of them is a continuous operator.
Then ${\bf v}^\pm_1$ belong to $H^1(\Omega^\pm)^n$, respectively, and satisfy transmission and boundary conditions \eqref{BT3}-\eqref{BT5}.
Let us now define
\begin{align}\label{gpm}
g^\pm_1:={\rm{div}}\, {\bf v}_1^\pm\in L^2(\Omega^\pm)
\end{align}
Then by the divergence Theorem and condition \eqref{hg} we obtain
\begin{multline}
\label{hg1}
\int _{\Omega^+}g_1^+(x) dx+\int _{\Omega^-}g_1^-(x) dx
=\int _{\Omega^+}{\rm{div}}\, {\bf v}_1^+(x) dx+\int _{\Omega^-}{\rm{div}}\, {\bf v}_1^-(x) dx\\
=\int _{\partial\Omega^+}\gamma_{_{\Omega^+}}{\bf v}_1^+\cdot \boldsymbol\nu d\sigma
+\int _{\partial\Omega^-}\gamma_{_{\Omega^-}}{\bf v}_1^-\cdot \boldsymbol\nu d\sigma\\
=\int _{\Gamma^+}\gamma_{_{\Omega^+}}{\bf v}_1^+\cdot \boldsymbol\nu d\sigma
+\int _{\Gamma^-}\gamma_{_{\Omega^-}}{\bf v}_1^-\cdot \boldsymbol\nu d\sigma
+\int _{\Sigma}(\gamma_{_{\Omega^+}}{\bf v}_1^+-\gamma_{_{\Omega^-}}{\bf v}_1^-)\cdot \boldsymbol\nu _{_{\Sigma }}d\sigma\\
=\int _{\partial \Omega }{\bs\varphi}\cdot \boldsymbol\nu d\sigma
+\int _{\Sigma}{\boldsymbol \varphi }_{_{\Sigma }}\cdot\boldsymbol\nu _{_\Sigma }d\sigma \,
=\int _{\Omega^+}g^+ dx +\int _{\Omega_-}g^-dx.
\end{multline}
Let $g_2\in L^2(\Omega)$ be such that $g_2|_{\Omega_\pm}=g^\pm-g^\pm_1$.
Hence $g_2$ belongs to the space $L_0^2(\Omega)$ defined in \eqref{L2-0}.
Then by Proposition~\ref{LS-prop} there exists ${\bf v}_2\in \mathring H^1(\Omega)$ satisfying
\begin{align}
{\rm{div}}\, {\bf v}_2=g_2 \mbox{ in } \Omega \,,
\end{align}
and a constant $C=C(\Omega,n)>0$ such that
\begin{align}
\label{BT7}
\|{\bf v}_2\|_{H^1(\Omega )^n}\le
C\|g-g_1\|_{L^2(\Omega )}\le C\|g_2\|_{L^2(\Omega )}.
\end{align}
Finally, choosing
$
{\bf v}^\pm:={\bf v}^\pm_1 + r_{_{\Omega^\pm}}{\bf v}_2
$
and using inequality \eqref{BT7} and continuity of the operators involved in \eqref{E3.49}-\eqref{gpm}, we obtain the desired result.
\qed
\end{proof}

Let us define the operators
$\check{\bs{\mathfrak L}}^\pm: {H}^1(\Omega^\pm)^n\to \widetilde{H}^{-1}(\Omega^\pm)^n$, cf. \eqref{L-oper}, as
\begin{align}
\label{checkL}
(\check{\bs{\mathfrak L}}^\pm{\bf v}^\pm)_i
:=\partial _\alpha\mathring E_{\Omega^\pm}\big(a_{ij}^{\alpha \beta }E_{j\beta }({\bf v}^\pm)\big),\ \ i=1,\ldots ,n,\quad
\forall\ {\bf v}^\pm\in {H}^1(\Omega^\pm)^n\,,
\end{align}
where $\mathring E_{\Omega^\pm}$ are the operators of zero extensions from $\Omega^\pm$ onto $\R^n$.
Then {\Rd by \eqref{Stokes-new} we have that} $\big(\check{\bs{\mathfrak L}}^\pm{\bf v}^\pm\big)|_{\Omega^\pm}=\boldsymbol{\mathcal L}({\bf v}^\pm,0)$ in $\Omega^\pm$.
\begin{theorem}
\label{T-nn}
Let Assumption $\ref{interface-Sigma}$ {\Gr and
conditions \eqref{Stokes-1}-\eqref{mu} hold.}
Then for all
\linebreak
$(\widetilde{\bf f}^+,\widetilde{\bf f}^-,{g^+,g^-},{\boldsymbol \varphi }_{_{\Sigma }}, {\boldsymbol \psi}_{_{\Sigma }}, {\bs\varphi})\in \mathfrak Y_\bullet$ the Dirichlet-transmission problem \eqref{Dirichlet-var-Stokes-nn} has a unique solution $({\bf u}^+,\pi ^+,{\bf u}^-,\pi ^-)$ in the space ${\mathfrak X}_{\Omega ^+,\Omega ^-}$ and there exists a constant $C=C({\Omega^+,\Omega^-} ,C_{\mathbb A},n)>0$ such that
\begin{align*}
\|({\bf u}^+,\pi ^+,{\bf u}^-,\pi ^-)\|_{{\mathfrak X}_{\Omega ^+,\Omega ^-}}
\leq C\|(\widetilde{\bf f}^+,\widetilde{\bf f}^-,{g^+,g^-},{\boldsymbol \varphi}_{_{\Sigma }},{\boldsymbol \psi}_{_{\Sigma }}, {\bs\varphi})\|_{\mathfrak Y_\bullet}
\end{align*}
\end{theorem}
\begin{proof}
Let ${\bf v}^\pm\in {H}^1(\Omega^\pm)^n$  be the functions given by Lemma~\ref{L.BT}.
Defining ${\bf \check f}^\pm:=\check{\bs{\mathfrak L}}^\pm{\bf v}^\pm \in \widetilde{H}^{-1}(\Omega^\pm)^n$, we obtain that ${\bf \check f}^\pm|_{\Omega^\pm}=\boldsymbol{\mathcal L}({\bf v}^\pm,0)$ in $\Omega^\pm$
and ${\bf t}_{\Omega^\pm}({\bf v}^\pm,0;\check{\bf f}^\pm)={\bf 0}$ by Definition~\ref{conormal-derivative-var-Brinkman}.

Then the fully non-homogeneous Dirichlet-transmission problem \eqref{Dirichlet-var-Stokes-nn} reduces to the {following} Dirichlet-transmission problem with homogeneous {Dirichlet} conditions on $\Gamma^\pm$ and homogeneous interface condition for the traces across $\Sigma $, for the new functions ${\bf w}^\pm :={\bf u}^\pm -{\bf v}^\pm$.
\begin{equation}
\label{Dirichlet-var-Stokes-nnn}
\!\!\!\!\!\left\{
\begin{array}{llll}
\boldsymbol{\mathcal L}({\bf w}^+,\pi ^+)=(\widetilde{\bf f}^+ -\check{\bf f}^+)|_{\Omega^+},\
{\rm{div}}\, {\bf w}^+=0 & \mbox{ in } \Omega^+\\
\boldsymbol{\mathcal L}({\bf w}^-,\pi ^-)=(\widetilde{\bf f}^- -\check{\bf f}^-)|_{\Omega^-},\
{\rm{div}}\, {\bf w}^-=0 & \mbox{ in } \Omega^-\\
(\gamma _{_{\Omega^+}}{\bf w}^+)|_{_{\Sigma }} =(\gamma _{_{\Omega^-}}{\bf w}^{-})|_{_{\Sigma }}
& \mbox{ on } \Sigma \\
\big({\bf t}_{\Omega^+}({\bf w}^+,\pi ^+;\widetilde{\bf f}^+ -\check{\bf f}^+)\big)|_{_{\Sigma }}\,
+\, \big({\bf t}_{\Omega^-}({\bf w}^-,\pi ^-;\widetilde{\bf f}^- -\check{\bf f}^-)\big)|_{_{\Sigma }} = {\boldsymbol \psi}_{_{\Sigma }}
&\mbox{ on } \Sigma \\
\!({\gamma _{_{\Omega^+}}}{\bf w}^+)|_{\Gamma ^+}={\bf 0}&  \mbox{ on } \Gamma ^+\\
\!({\gamma _{_{\Omega^-}}}{\bf w}^-)|_{\Gamma ^-}={\bf 0}&  \mbox{ on } \Gamma ^-.
\end{array}\right.
\end{equation}

Theorem \ref{T-p} implies that the Dirichlet-transmission problem \eqref{Dirichlet-var-Stokes-nnn} has a unique solution $({\bf w}^+,\pi ^+,{\bf w}^-,\pi ^-)$ in the space ${\mathfrak X}_{\Omega^+,\Omega^-}$ and depends continuously on the given data of this problem.
Finally, the well-posedness of problem \eqref{Dirichlet-var-Stokes-nnn} implies that the functions $({\bf u}^\pm={\bf w}^\pm + {\bf v}^\pm,\ \pi ^\pm)$ determine a solution of the full non-homogeneous Dirichlet-transmission problem \eqref{Dirichlet-var-Stokes-nn} in the space ${\mathfrak X}_{\Omega^+,\Omega^-}$, and depends continuously on the given data $(\widetilde{\bf f}^+,\widetilde{\bf f}^-,{ g^+,g^-},{\boldsymbol \varphi }_{_{\Sigma }}, {\boldsymbol \psi}_{_{\Sigma }}, {\bs\varphi})\in \mathfrak Y_\bullet$.
This solution is unique by the uniqueness statement in Theorem \ref{T-p}.
\qed
\end{proof}

\subsection{\bf Dirichlet-transmission problem with more general non-homogeneous interface and Dirichlet conditions}\label{S5.5}
Let us now consider the non-homogeneous Dirichlet-transmission problem
\begin{equation}
\label{Dirichlet-var-Stokes-pm}
\left\{
\begin{array}{ll}
\boldsymbol{\mathcal L}({\bf u}^+,\pi ^+)=\widetilde{\bf f}^+|_{\Omega ^+},\quad {\rm{div}}\, {\bf u}^+={g^+} & \mbox{ in } \Omega ^+\,,\\
\boldsymbol{\mathcal L}({\bf u}^-,\pi ^-)=\widetilde{\bf f}^-|_{\Omega ^-},\quad {\rm{div}}\, {\bf u}^-={g^-} & \mbox{ in } \Omega ^-\,,\\
(\gamma _{_{\Omega ^+}}{\bf u}^+)|_{_{\Sigma }}-{(\gamma _{_{\Omega ^-}}{\bf u}^{-})|_{_{\Sigma }}}={\boldsymbol \varphi}_{_{\Sigma }} & \mbox{ on } \Sigma \,,\\
{\big({\bf t}_{\Omega ^+}({\bf u}^+,\pi ^+;\widetilde{\bf f}^+)\big)|_{_{\Sigma }}}\,  +\, {\big({\bf t}_{\Omega ^-}({\bf u}^-,\pi ^-;\widetilde{\bf f}^-)\big)|_{_{\Sigma }}}={\boldsymbol \psi}_{_{\Sigma }} &  \mbox{ on } \Sigma \,,\\
({\gamma _{_{\Omega ^+}}}{\bf u}^+)|_{\Gamma ^+}={\bs\varphi^+} &  \mbox{ on } \Gamma ^+\,,\\
({\gamma _{_{\Omega ^-}}}{\bf u}^-)|_{\Gamma ^-}={\bs\varphi^-} &  \mbox{ on } \Gamma ^-\,,
\end{array}\right.
\end{equation}
with more general data
$(\widetilde{\bf f}^+,\widetilde{\bf f}^-,{g^+,g^-},\boldsymbol \varphi_{_{\Sigma }}, {\boldsymbol \psi}_{_{\Sigma }}, {\bs\varphi^+,\bs\varphi^-})\in \mathfrak Y$, where
$\mathfrak Y$ consists of
\begin{multline*}
(\widetilde{\bf f}^+,\widetilde{\bf f}^-,g^+,g^-,\boldsymbol \varphi_{_{\Sigma }}, {\boldsymbol \psi}_{_{\Sigma }}, \bs\varphi^+,\bs\varphi^-)\in
\left({H}_{\Gamma ^+}^1(\Omega ^+)^n\right)'
\!\times\!\left({H}_{\Gamma ^-}^1(\Omega ^-)^n\right)'\\
\!\times{L^2(\Omega^+)
\!\times L^2(\Omega^-)}
\times \!{H^{\frac{1}{2}}}(\Sigma )^n\!
\times \!{{H}^{-\frac{1}{2}}(\Sigma )^n}
\!\times \!H^{\frac{1}{2}}(\Gamma^+)^n
\!\times \!H^{\frac{1}{2}}(\Gamma^-)^n,
\end{multline*}
such that $g^+,g^-,\bs\varphi_{_\Sigma},\bs\varphi^+, \bs\varphi^-$ satisfy the compatibility condition
\begin{align}\label{hg-pm}
\int _{\Omega^+}g^+ dx + \int _{\Omega^-}g^-dx
=\int_{\Gamma^+}\bs\varphi^+\cdot \boldsymbol\nu d\sigma
+\int_{\Gamma^-}\bs\varphi^-\cdot \boldsymbol\nu d\sigma
+\int_{\Sigma}{\boldsymbol \varphi }_{_{\Sigma }}\cdot\boldsymbol\nu _{_\Sigma }d\sigma,
\end{align}
and the condition
\begin{align}
\label{phi-pm-comp}
\bs\varphi_{_\Sigma}-r_{_\Sigma} \bs\Phi^+ +r_{_\Sigma} \bs\Phi^-
\in{H}_\bullet^{\frac{1}{2}}(\Sigma)^n
\end{align}
for some extensions $\bs\Phi^\pm\in H^{\frac{1}{2}}(\partial\Omega^\pm)^n$ such that
$\gr r_{_{\Gamma^\pm}}\bs\Phi^\pm=\bs\varphi^\pm$.
This space is endowed with the norm
\begin{multline}
\hspace{-1em}
\|(\widetilde{\bf f}^+,\widetilde{\bf f}^-,{g^+,g^-},{\boldsymbol \varphi}_{_{\Sigma }},{\boldsymbol \psi}_{_{\Sigma }}, \bs\varphi^+,\bs\varphi^-)\|_{\mathfrak Y}:=\\
\|\widetilde{\bf f}^+\|_{\left({H}_{\Gamma ^+}^1(\Omega ^+)^n\right)'}+
\|\widetilde{\bf f}^-\|_{\left({H}_{\Gamma ^-}^1(\Omega ^-)^n\right)'}
+\|g^+\|_{L^2(\Omega^+)}
+\|g^-\|_{L^2(\Omega^-)}\\
+\|\bs\varphi_{_{\Sigma }}\|_{H^{\frac{1}{2}}(\Sigma)^n}
+\|\bs\psi_{_{\Sigma }}\|_{H^{-\frac{1}{2}}(\Sigma )^n}
+\|\bs\varphi^+\|_{H^{\frac{1}{2}}(\Gamma^+)^n}
+\|\bs\varphi^-\|_{H^{\frac{1}{2}}(\Gamma^-)^n}\gR.
\label{mfYnorm}
\end{multline}
\begin{rem}\label{R3.3}
$(i)$ Condition \eqref{phi-pm-comp} is particularly satisfied if
$\bs\varphi_{_\Sigma}\in {H}_\bullet^{\frac{1}{2}}(\Sigma)^n$ and
$\bs\varphi^\pm=r_{_{\Gamma^\pm}}\bs\varphi$ for some $\bs\varphi\in H^{\frac{1}{2}}(\partial\Omega)^n$, as in the case considered in Section $\ref{S3.3.2}$.
Indeed, we can choose
$\bs\Phi^\pm=\gamma_{_{\Omega^\pm}}\gamma^{-1}_{_{\Omega}}\bs\varphi$ and obtain that
$\bs\Phi^\pm\in H^{\frac{1}{2}}(\partial\Omega^\pm)^n$ and
$$
r_{_{\Gamma^\pm}}\bs\Phi^\pm
=r_{_{\Gamma^\pm}}\gamma_{_{\Omega^\pm}}\gamma^{-1}_{_{\Omega}}\bs\varphi
=r_{_{\Gamma^\pm}}\gamma_{_{\Omega}}\gamma^{-1}_{_{\Omega}}\bs\varphi
=r_{_{\Gamma^\pm}}\bs\varphi
=\bs\varphi^\pm\,,
$$
which implies that $\bs\Phi^\pm$ are extensions of $\bs\varphi^\pm$.
Moreover, the property $\gamma^{-1}_{_{\Omega}}\bs\varphi\in H^{1}(\Omega)^n$ implies that
$$r_{_\Sigma} \bs\Phi^+ -r_{_\Sigma} \bs\Phi^-
=r_{_\Sigma}\gamma_{_{\Omega^+}}\gamma^{-1}_{_{\Omega}}\bs\varphi - r_{_\Sigma}\gamma_{_{\Omega^-}}\gamma^{-1}_{_{\Omega}}\bs\varphi
=\bf 0\,.$$

$(ii)$ If condition \eqref{phi-pm-comp} is satisfied for some functions $\bs\Phi^\pm\in H^{\frac{1}{2}}(\partial\Omega^\pm)^n$ such that $\bs\varphi^\pm=r_{_{\Gamma^\pm}}\bs\Phi^\pm$
then it is also satisfied for all functions $\bs\Phi^\pm_*\in H^{\frac{1}{2}}(\partial\Omega^\pm)^n$ such that $\bs\varphi^\pm=r_{_{\Gamma^\pm}}\bs\Phi^\pm_*$
because $\bs\Phi^\pm -\bs\Phi^\pm_* ={\bf 0}$ on $\Gamma^\pm$ and hence
$r_{_\Sigma} (\bs\Phi^\pm -\bs\Phi^\pm_*)\in{H}_\bullet^{\frac{1}{2}}(\Sigma)^n$.
\end{rem}

In order to analyze the non-homogeneous Dirichlet-transmission problem, we need the following generalized Bogoskii-type transmission result.
\begin{lemma}\label{L.BTpm}
Let Assumption $\ref{interface-Sigma}$ hold
and let
$$
(g^+,g^-,\bs\varphi_{_\Sigma},\bs\varphi^+,\bs\varphi^-)\in
L^2(\Omega^+)
\times L^2(\Omega^-)
\times H^{\frac{1}{2}}(\Sigma )^n
\times H^{\frac{1}{2}}(\Gamma^+)^n
\times H^{\frac{1}{2}}(\Gamma^-)^n
$$
satisfy conditions \eqref{hg-pm} and \eqref{phi-pm-comp}
Then
there exist ${\bf v}^\pm\in {H}^1(\Omega ^\pm)^n$
such that
\begin{align}
\label{BT1pm}
{\rm{div}}\, {\bf v}^+=g^+ & \mbox{ in } \Omega ^+,  \\
\label{BT2pm}
{\rm{div}}\, {\bf v}^-=g^-  & \mbox{ in } \Omega ^-,  \\
\label{BT3pm}
(\gamma _{_{\Omega ^+}}{\bf v}^+)|_{_{\Sigma }}-{(\gamma _{_{\Omega ^-}}{\bf v}^-)|_{_{\Sigma }}}
=\boldsymbol \varphi_{_{\Sigma }} &\mbox{ on } \Sigma \,,\\
\label{BT4pm}
({\gamma _{_{\Omega ^+}}}{\bf v}^+)|_{\Gamma ^+}=\bs\varphi^+ &  \mbox{ on } \Gamma ^+\,,\\
\label{BT5pm}
({\gamma _{_{\Omega ^-}}}{\bf v}^-)|_{\Gamma ^-}=\bs\varphi^- &  \mbox{ on } \Gamma ^-\,,
\end{align}
and, moreover,
\begin{multline}
\label{BT6pm}
\!\!\!\!\|{\bf v}^+\|_{H^1(\Omega ^+)^n}\!+\!\|{\bf v}^-\|_{H^1(\Omega ^-)^n}
\!\leq \!
C_\Sigma\Big(\|g^+\|_{L^2(\Omega^+)}\!+\|g^-\|_{L^2(\Omega^-)}\\
+\!\|{\boldsymbol \varphi }_{_{\Sigma }}\|_{H^{\frac{1}{2}}(\Sigma )^n}
\!+\!\|\bs\varphi^+\|_{H^{\frac{1}{2}}(\Gamma^+)^n}
\!+\!\|\bs\varphi^-\|_{H^{\frac{1}{2}}(\Gamma^-)^n}\Big)\,,
\end{multline}
with some constant $C_\Sigma=C_\Sigma(\Omega^+,\Omega^- ,n)>0$.
\end{lemma}
\begin{proof}
We will prove this lemma by modifying the proof of Lemma \ref{L.BT} appropriately.
Let
$\gamma^{-1}_{\Omega^\pm}: H^{\frac{1}{2}}(\partial \Omega^\pm)^n\to H^1(\Omega ^\pm )^n$
be some continuous right inverses to the corresponding trace operators.

Let $\bs\Phi^\pm\in H^{\frac{1}{2}}(\partial\Omega^\pm)^n$ denote some extensions of the functions $\bs\varphi^\pm$ that is,
$\gr r_{_{\Gamma^\pm}}\bs\Phi^\pm=\bs\varphi^\pm$. Let us introduce the functions
\begin{align}
\label{E3.49pm}
&\bs\Phi_{_\Sigma}:=\bs\varphi_{_{\Sigma }}
-r_{_\Sigma}\bs\Phi^+
+r_{_\Sigma}\bs\Phi^-,\\
&{\bf v}^+_1:=\gamma^{-1}_{\Omega^+}(\bs\Phi^+
+\frac{1}{2}\mathring E_{_{\Sigma\to\partial\Omega^+}}\bs\Phi_{_\Sigma}),\\
&{\bf v}^-_1:=\gamma^{-1}_{\Omega^-}(\bs\Phi^- -\frac{1}{2}\mathring E_{_{\Sigma\to\partial\Omega^-}}\bs\Phi_{_\Sigma}).
\end{align}
Due to condition \eqref{phi-pm-comp} and Remark \ref{R3.3}(ii), $\bs\Phi_{_\Sigma}\in {H}_\bullet^{\frac{1}{2}}(\Sigma)^n$ and hence
$\mathring E_{_{\Sigma\to\partial\Omega^\pm}}\bs\Phi_{_\Sigma}\in {\gR H^{\frac{1}{2}}}(\partial \Omega^\pm)^n$.
Then ${\bf v}^\pm_1\in H^1(\Omega^\pm)^n$ and satisfy transmission and boundary conditions \eqref{BT3pm}-\eqref{BT5pm}.
Let us now define
\begin{align}\label{gpm-pm}
g^\pm_1:={\rm{div}}\, {\bf v}_1^\pm\in L^2(\Omega^\pm)
\end{align}
Then by the divergence Theorem and condition \eqref{hg-pm} we obtain
\begin{align}
\label{hg1pm}
&\int _{\Omega^+}g_1^+ dx+\int _{\Omega^-}g_1^-dx
=\int _{\Omega^+}{\rm{div}}\, {\bf v}_1^+ dx+\int _{\Omega^-}{\rm{div}}\, {\bf v}_1^-dx\\
&=\int _{\partial\Omega^+}\gamma_{_{\Omega^+}}{\bf v}_1^+\cdot \boldsymbol\nu d\sigma
+\int _{\partial\Omega^-}\gamma_{_{\Omega^-}}{\bf v}_1^-\cdot \boldsymbol\nu d\sigma\nonumber\\
&=\int _{\Gamma^+}\gamma_{_{\Omega^+}}{\bf v}_1^+\cdot \boldsymbol\nu d\sigma
+\int _{\Gamma^-}\gamma_{_{\Omega^-}}{\bf v}_1^-\cdot \boldsymbol\nu d\sigma
+\int _{\Sigma}(\gamma_{_{\Omega^+}}{\bf v}_1^+-\gamma_{_{\Omega^-}}{\bf v}_1^-)\cdot \boldsymbol\nu _{_{\Sigma }}d\sigma\nonumber\\
&=\int _{\Gamma^+}\bs\varphi^+\cdot \boldsymbol\nu d\sigma
+\int _{\Gamma^-}\bs\varphi^-\cdot \boldsymbol\nu d\sigma
+\int _{\Sigma}{\boldsymbol \varphi }_{_{\Sigma }}\cdot\boldsymbol\nu _{_\Sigma }d\sigma
=\int _{\Omega^+}g^+ dx +\int _{\Omega_-}g^-dx.\nonumber
\end{align}
Let $g_2\in L^2(\Omega)$ be such that $g_2|_{\Omega_\pm}=g^\pm-g^\pm_1$.
Hence $g_2$ belongs to the space $L_0^2(\Omega)$ defined in \eqref{L2-0}.
Then by Proposition~\ref{LS-prop} there exists ${\bf v}_2\in \mathring H^1(\Omega)$ such that
\begin{align}
{\rm{div}}\, {\bf v}_2=g_2\ \mbox{ in }\, \Omega \,,
\quad
\label{BT7pm}
&\|{\bf v}_2\|_{H^1(\Omega )^n}\le C\|g_2\|_{L^2(\Omega )},
\end{align}
where $C=C(\Omega,n)$ is a positive constant.

Finally, choosing
$
{\bf v}^\pm:={\bf v}^\pm_1 + r_{_{\Omega^\pm}}{\bf v}_2
$
and using the inequality in \eqref{BT7pm} and the continuity of the operators involved in \eqref{E3.49pm}-\eqref{gpm-pm}, we get the assertion.
\qed
\end{proof}

\begin{theorem}
\label{T-pm}
Let Assumption $\ref{interface-Sigma}$ {and
conditions \eqref{Stokes-1}-\eqref{mu} hold.}
Then for all
\linebreak
$(\widetilde{\bf f}^+,\widetilde{\bf f}^-,{g^+,g^-},\boldsymbol \varphi_{_{\Sigma }}, {\boldsymbol \psi}_{_{\Sigma }}, {\bs\varphi^+,\bs\varphi^-})\in \mathfrak Y$
the Dirichlet-transmission problem \eqref{Dirichlet-var-Stokes-pm} has a unique solution $({\bf u}^+,\pi ^+,{\bf u}^-,\pi ^-)$ in the space
${\mathfrak X}_{\Omega ^+,\Omega ^-}$ defined in \eqref{Omega-pm},
and there exists a constant $C=C({\Omega^+,\Omega^-} ,C_{\mathbb A},n)>0$ such that
\begin{align*}
\|({\bf u}^+,\pi ^+,{\bf u}^-,\pi ^-)\|_{{\mathfrak X}_{\Omega ^+,\Omega ^-}}
\leq
C\|(\widetilde{\bf f}^+,\widetilde{\bf f}^-,{g^+,g^-},\boldsymbol \varphi_{_{\Sigma }}, {\boldsymbol \psi}_{_{\Sigma }}, {\bs\varphi^+,\bs\varphi^-})\|_{\mathfrak Y}.
\end{align*}
\end{theorem}
\begin{proof}
We use arguments similar to those in the proof of Theorem \ref{T-nn}.
Let ${\bf v}^\pm\in {H}^1(\Omega^\pm)^n$ be the functions given by Lemma~\ref{L.BTpm}.
For the velocity-pressure couples $({\bf v}^\pm,0)$, let
${\bf \check f}^\pm:=\check{\bs{\mathfrak L}}^\pm{\bf v}^\pm \in \widetilde{H}^{-1}(\Omega^\pm)^n$,
where operators $\check{\bs{\mathfrak L}}^\pm$ are defined in \eqref{checkL}.
Hence ${\bf \check f}^\pm\in \widetilde{H}^{-1}(\Omega ^\pm)^n$,
${\bf \check f}^\pm|_{\Omega^\pm}=\boldsymbol{\mathcal L}({\bf v}^\pm,0)$ in $\Omega^\pm$,
cf. \eqref{Stokes-new}, and
${\bf t}_{\Omega^\pm}({\bf v}^\pm,0;\check{\bf f}^\pm)={\bf 0}$ by Definition~\ref{conormal-derivative-var-Brinkman}.

Then for the new functions ${\bf w}^\pm :={\bf u}^\pm -{\bf v}^\pm$, the {fully} {\bn non-}homogeneous Dirichlet-transmission problem \eqref{Dirichlet-var-Stokes-pm} reduces to the following Dirichlet-transmission problem with homogeneous {Dirichlet} conditions on $\Gamma^\pm$ and homogeneous interface condition for the traces across $\Sigma $.
\begin{equation}
\label{Dirichlet-var-Stokes-0pm}
\!\!\!\!\!\left\{
\begin{array}{llll}
\boldsymbol{\mathcal L}({\bf w}^+,\pi ^+)=(\widetilde{\bf f}^+ -\check{\bf f}^+)|_{\Omega^+},\
{\rm{div}}\, {\bf w}^+=0 & \mbox{ in } \Omega^+,\\
\boldsymbol{\mathcal L}({\bf w}^-,\pi ^-)=(\widetilde{\bf f}^- -\check{\bf f}^-)|_{\Omega^-},\
{\rm{div}}\, {\bf w}^-=0 & \mbox{ in } \Omega^-,\\
(\gamma _{_{\Omega^+}}{\bf w}^+)|_{_{\Sigma }} =(\gamma _{_{\Omega^-}}{\bf w}^{-})|_{_{\Sigma }}
& \mbox{ on } \Sigma, \\
\big({\bf t}_{\Omega^+}({\bf w}^+,\pi ^+;\widetilde{\bf f}^+ -\check{\bf f}^+{)}\big)|_{_{\Sigma }}\,
+\, \big({\bf t}_{\Omega^-}({\bf w}^-,\pi ^-;\widetilde{\bf f}^- -\check{\bf f}^-)\big)|_{_{\Sigma }} = {\boldsymbol \psi}_{_{\Sigma }}
&\mbox{ on } \Sigma, \\
\!({\gamma _{_{\Omega^+}}}{\bf w}^+)|_{\Gamma ^+}={\bf 0}&  \mbox{ on } \Gamma ^+,\\
\!({\gamma _{_{\Omega^-}}}{\bf w}^-)|_{\Gamma ^-}={\bf 0}&  \mbox{ on } \Gamma ^-.
\end{array}\right.
\end{equation}

Theorem \ref{T-p} implies that the Dirichlet-transmission problem \eqref{Dirichlet-var-Stokes-0pm} has a unique solution $({\bf w}^+,\pi ^+,{\bf w}^-,\pi ^-)$ in the space ${\mathfrak X}_{\Omega^+,\Omega^-}$
and depends continuously on the given data of this problem.
Finally, the well-posedness of problem \eqref{Dirichlet-var-Stokes-0pm}
implies that the functions
$({\bf u}^\pm={\bf w}^\pm + {\bf v}^\pm,\ \pi ^\pm)$
determine a solution of the full {\bn non-}homogeneous Dirichlet-transmission problem \eqref{Dirichlet-var-Stokes-pm} in the space ${\mathfrak X}_{\Omega^+,\Omega^-}$,
and {the solution} depends continuously on the given data
$(\widetilde{\bf f}^+,\widetilde{\bf f}^-,{g^+,g^-},\boldsymbol \varphi_{_{\Sigma }}, {\boldsymbol \psi}_{_{\Sigma }}, {\bs\varphi^+,\bs\varphi^-})\in\mathfrak Y$.
This solution is unique by the uniqueness statement in Theorem \ref{T-p}.
\qed
\end{proof}

\section{\bf Dirichlet problem for the incompressible anisotropic Navier-Stokes system with general data in a bounded Lipschitz domain.}
In this section, we consider the existence of a weak solution of a fully {\bn non-}homogeneous Dirichlet problem for the anisotropic Navier-Stokes system {in the incompressible case} with general data in $L^2$-based Sobolev spaces in a bounded Lipschitz domain in ${\mathbb R}^n$, $n=2,3$.

We use the well-posedness result established in Theorem \ref{lemma-a47-1-Stokes} for the Dirichlet problem for the Stokes system and the following variant of the {\it Leray-Schauder fixed point theorem} (see \cite[Theorem 11.3]{Gilbarg-Trudinger}).
\begin{theorem}
\label{L-S-fixed}
Let ${\mathcal X}$ denote a Banach space and $T:{\mathcal X}\to {\mathcal X}$ be a continuous
and compact operator. If there exists a constant $M_0>0$ such that {$\|x\|_{\mathcal X}\leq M_0$} for every pair $(x,{\lambda })\in {\mathcal X}\times [0,1]$ satisfying $x={\lambda}Tx$, then the operator $T$ has a fixed point $x_0$ $($with $\|x_0\|_X\leq M_0$$)$.
\end{theorem}
Recall that $\Omega \subset \mathbb R^n$ is a bounded Lipschitz domain
and denote
\begin{align}
\label{h-n}
&H_{\boldsymbol\nu }^{\frac{1}{2}}(\partial \Omega )^n:=\left\{\boldsymbol\varphi \in H^{\frac{1}{2}}(\partial \Omega )^n: \langle \boldsymbol\varphi,\boldsymbol \nu \rangle _{\partial \Omega }=0\right\}.
\end{align}

Next we restrict our analysis to the case $n\in \{2,3\}$, which will allow to use some compact embedding results.
Consider the following Dirichlet problem,
\begin{equation}
\label{int-NS-D}
\left\{
\begin{array}{ll}
\boldsymbol{\mathcal L}({\bf u},\pi)=-\bs{\mathfrak F}+({\bf u}\cdot \nabla ){\bf u}\,, \ \ {\rm{div}} \, {\bf u}=0 & \mbox{ in } \Omega,\\
{\gamma _{_{\Omega}}}{\bf u}=\bs\varphi &  \mbox{ on } \partial\Omega.
\end{array}\right.
\end{equation}
for the couple of unknowns $({\bf u},\pi )\in {H}^1(\Omega )^n\times L^2(\Omega )/{\mathbb R}$ and the given data
$(\bs{\mathfrak F},\bs\varphi)\in {H}^{-1}(\Omega )^n
\times H_{\boldsymbol\nu }^{\frac{1}{2}}(\partial \Omega )^n$.

The main tool for our next arguments is the following assertion (see, e.g., \cite[(1.4)]{Korobkov}, \cite{Leray}).
\begin{lemma}\label{LH-lem}
Let
$\bs\varphi\in H_{\boldsymbol\nu }^{\frac{1}{2}}(\partial \Omega )^n$.
Then for any $\varepsilon >0$  there exists
${\bf v}_\varepsilon={\bf v}_\varepsilon(\bs\varphi;\Omega)\in H_{{\rm{div}}}^{1}(\Omega )^n$ such that
\begin{align}
\label{extension-h-1}
\gamma_{\Omega }{\bf v}_\varepsilon=\bs\varphi& \mbox{ on } \ \partial \Omega \,
\end{align}
and the following Leray-Hopf inequality holds
\begin{align}
\label{extension-h}
\big|\big\langle ({\bf v} \cdot \nabla ){\bf v}_\varepsilon,{\bf v}\big\rangle _{\Omega }\big|
\leq \varepsilon \|\nabla {\bf v}\|_{L^2(\Omega )^{n\times n}}^2,\ \forall \ {\bf v} \in \mathring{H}_{{\rm{div}}}^{1}(\Omega )^n\,.
\end{align}
\end{lemma}
Next we show the following existence result (see also
\cite[Proposition 1.1]{Seregin} in the isotropic incompressible case \eqref{isotropic} with $\mu=1$).
\begin{theorem}
\label{Dir-NS-var}
Let $\Omega \subset {\mathbb R}^n$, $n\in \{2,3\}$, be a bounded Lipschitz domain. Let conditions \eqref{Stokes-1}-\eqref{mu} hold.
Then for all given data $(\bs{\mathfrak F},\bs\varphi)\in {H}^{-1}(\Omega )^n
\times H_{\boldsymbol\nu }^{\frac{1}{2}}(\partial \Omega )^n$,
the Dirichlet problem for the anisotropic Navier-Stokes system \eqref{int-NS-D}
has a solution $({\bf u},\pi)\in {H}^1(\Omega )^n\times L^2(\Omega)/\R$.
\end{theorem}

\begin{proof}
We reduce the analysis of the nonlinear problem \eqref{int-NS-D}
to the analysis of a nonlinear operator in the Hilbert space $\mathring{H}_{{\rm{div}}}^1(\Omega )^n$ and show that this operator has a fixed-point due to the Leray-Schauder Theorem (cf. \cite{Leray}, see also \cite{Korobkov}).

To this end, we represent a solution of problem \eqref{int-NS-D} in the form
\begin{align}
\label{form-solution}
{\bf u}={\mathbf u_0}+{\bf v}_\varepsilon\,,
\end{align}
where ${\bf v}_\varepsilon\in {H}_{{\rm{div}}}^1(\Omega )^n$ satisfies relations \eqref{extension-h-1} and \eqref{extension-h} with an $\varepsilon$ that will be specified later, while ${\mathbf u_0}\in \mathring{H}_{{\rm{div}}}^1(\Omega )^n$. 

Then the Dirichlet problem \eqref{int-NS-D} reduces to the nonlinear equation
\begin{align}
\label{transmission-linear-NS-a-equiv-1}
\bs{\mathcal L}({\bf u}_0,\pi)=
\mathbf F_\varepsilon\mathbf u_0
\end{align}
for the couple of unknowns $({\bf u}_0,\pi )\in \mathring{H}_{{\rm{div}}}^1(\Omega )^n\times L^2(\Omega )/{\mathbb R}$,
where
\begin{align}
\label{Fuphi}
\mathbf F_\varepsilon\mathbf w:&=
-\bs{\mathfrak F}-\boldsymbol{\mathfrak L}{\bf v}_\varepsilon
+(({\mathbf w}+{\bf v}_\varepsilon)\cdot \nabla )({\mathbf w}+{\bf v}_\varepsilon),\ \ \forall \, {\bf w}\in \mathring{H}_{{\rm{div}}}^1(\Omega )^n
\end{align}
(cf. notations in \eqref{Stokes-0}, \eqref{L-oper} and \eqref{Stokes}).

For a fixed ${\bf v}_\varepsilon\in {H}_{{\rm{div}}}^1(\Omega )^n$, formula \eqref{Fuphi} defines a nonlinear mapping
${\bf w}\mapsto \mathbf F_\varepsilon\mathbf w$,
and  the nonlinear operator
$\mathbf F_\varepsilon$ acts from $\mathring{H}_{{\rm{div}}}^{1}(\Omega )^n$ to ${H}^{-1}(\Omega )^n$
due to the inclusion $\boldsymbol{\mathfrak L}{\bf v}_\varepsilon\in {H}^{-1}(\Omega )^n$ (provided by \eqref{A}) and estimate \eqref{P-01}.

By Theorem \ref{lemma-a47-1-Stokes}, the linear operator
\begin{align}\label{Lisom}
\bs{\mathcal L}:\mathring{H}_{{\rm{div}}}^1(\Omega )^n\times L^2(\Omega )/{\mathbb R}\to H^{-1}(\Omega)^n
\end{align}
is an isomorphism.
Its inverse operator can be split into two operator components,
$$\bs{\mathcal L}^{-1}=(\bs{\mathcal U},\mathcal P),$$
where
$\bs{\mathcal U}: H^{-1}(\Omega)^n\to \mathring{H}_{{\rm{div}}}^1(\Omega )^n$
and $\mathcal P:H^{-1}(\Omega)^n\to L^2(\Omega )/{\mathbb R}$ are linear continuous operators such that
$\bs{\mathcal L}(\bs{\mathcal U}\mathbf F,\mathcal P\mathbf F)=\mathbf F$ for any $\mathbf F\in H^{-1}(\Omega)^n$.
By applying the operator $\bs{\mathcal L}^{-1}$ in equation \eqref{transmission-linear-NS-a-equiv-1} we obtain the equivalent nonlinear system
\begin{align}
\label{transmission-linear-NS-a-equiv-2}
{\mathbf u_0}
&={\mathbf U}_{\varepsilon}{\mathbf u_0},\\
\label{transmission-linear-NS-a-equiv-2P}
\pi&= P_\varepsilon{\mathbf u_0},
\end{align}
where ${\mathbf U}_{\varepsilon}:\mathring{H}_{{\rm{div}}}^{1}(\Omega )^n\to \mathring{H}_{{\rm{div}}}^{1}(\Omega )^n$
and
$P_{\varepsilon}:\mathring{H}_{{\rm{div}}}^{1}(\Omega )^n\to L^2(\Omega )/{\mathbb R}$
are the nonlinear operators defined as
\begin{align}
\label{N}
&{\mathbf U}_{\varepsilon}{\mathbf w}:=\bs{\mathcal U}\,\mathbf F_\varepsilon\mathbf w,\\
\label{P}
&P_{\varepsilon}{\mathbf w}:=\mathcal P\,\mathbf F_\varepsilon\mathbf w
\end{align}
(cf. also \cite{Leray}
for $\mu =1$ in the isotropic incompressible case \eqref{isotropic}).

Since $\pi$ is not involved in \eqref{transmission-linear-NS-a-equiv-2}, we will first prove the existence of a solution ${\mathbf u_0}\in \mathring{H}_{\rm{div}}^1(\Omega )^n$ to this equation and then use \eqref{transmission-linear-NS-a-equiv-2P} as a representation formula.
This formula provides the existence of the pressure field $\pi\in L^2(\Omega )/{\mathbb R}$.

In order to show the existence of a fixed point of the operator ${\mathbf U}_\varepsilon$ and, thus, the existence of a solution of equation \eqref{transmission-linear-NS-a-equiv-2},
we employ Theorem \ref{L-S-fixed}.

We show first that ${\mathbf U}_{\varepsilon}$ is continuous.
Let ${\bf w},{\bf w}'\in \mathring{H}_{\rm{div}}^1(\Omega )^n$.
Then by \eqref{Fuphi} and  \eqref{P-01} there exists a constant $c_1>0$ such that
\begin{align}
\big\|\mathbf F_\varepsilon{\bf w}&-\mathbf F_\varepsilon{\bf w}'\big\|_{{H}^{-1}(\Omega)^n}
\le\left\|({\bf w}\cdot \nabla ){\bf w}-({\bf w}'\cdot \nabla ){\bf w}'\right\|_{{H}^{-1}(\Omega)^n}\nonumber\\
&\hspace{6em}+\left\|({\bf v}_\varepsilon\cdot \nabla )({\bf w}-{\bf w}')
+(({\bf w}-{\bf w}')\cdot \nabla ){\bf v}_\varepsilon\right\|_{{H}^{-1}(\Omega)^n}\nonumber\\
\le&\left\|(({\bf w}-{\bf w}')\cdot \nabla ){\bf w} + ({\bf w}'\cdot \nabla )({\bf w}-{\bf w}')\right\|_{{H}^{-1}(\Omega)^n}\nonumber\\
&\hspace{6em}
+2{\bn c_1^2}\left\|{\bf w}-{\bf w}'\right\|_{{H}^{1}(\Omega)^n} \|{\bf v}_\varepsilon\|_{{H}^{1}(\Omega)^n}\nonumber\\
\le& {\bn c_1^2}\left\|{\bf w}-{\bf w}'\right\|_{{H}^{1}(\Omega)^n}\left( \|{\bf w}\|_{{H}^{1}(\Omega)^n}
+ \|{\bf w}'\|_{{H}^{1}(\Omega)^n}
+2 \|{\bf v}_\varepsilon\|_{{H}^{1}(\Omega)^n}\right).
\label{P-3}
\end{align}
This estimate shows that the operator
$\mathbf F_\varepsilon:\mathring{H}_{{\rm{div}}}^{1}(\Omega )^n\to {H}^{-1}(\Omega )^n$
is continuous.
Consequently, the operator
${\mathbf U}_{\varepsilon}=\bs{\mathcal U}\,\mathbf F_\varepsilon:\mathring{H}_{{\rm{div}}}^1(\Omega )^n\to \mathring H_{{\rm{div}}}^1(\Omega )^n$ is also continuous, as asserted.

Next we show that the operator ${\mathbf U}_{\varepsilon}$ is compact. To this end, we assume that $\{{\bf w}_k\}_{k\in {\mathbb N}}$ is a bounded sequence in the space $\mathring{H}_{{\rm{div}}}^1(\Omega )^n$ endowed with the norm, coinciding with semi-norm \eqref{seminorm-Omega}, and prove that the sequence $\{\mathbf F_\varepsilon{\bf w}_k\}_{k\in {\mathbb N}}$ contains a convergent subsequence in ${H}^{-1}(\Omega )^n$.

Let $M>0$ be such that $\|{\bf w}_k\|_{\mathring{H}_{{\rm{div}}}^1(\Omega )^n}\leq M$ for all $k\in {\mathbb N}$.
Since the embedding of the space $\mathring{H}_{{\rm{div}}}^1(\Omega )^n$ into the space $L^4(\Omega )^n$ is compact (see, e.g., \cite[Theorem 6.3]{Adams2003}), there exists a subsequence of $\{{\bf w}_k\}_{k\in {\mathbb N}}$, labeled as the sequence for the sake of brevity, which converges in $L^4(\Omega )^n$, and, hence, is a Cauchy sequence in $L^4(\Omega )^n$.
From \eqref{Fuphi}, \eqref{P-01} and \eqref{P-0512b},  we obtain
\begin{align}
\label{compactness}
\big\|\mathbf F_\varepsilon{\bf w}_k&-\mathbf F_\varepsilon{\bf w}_\ell\big\|_{{H}^{-1}(\Omega)^n}
\!\le\!\left\|(({\bf w}_k-{\bf w}_\ell)\cdot \nabla ){\bf w}_k
\!+\!({\bf w}_\ell\cdot \nabla )({\bf w}_k-{\bf w}_\ell)\right\|_{{H}^{-1}(\Omega)^n}
\quad\nonumber\\
&\hspace{8em}+\left\|({\bf v}_\varepsilon\cdot \nabla )({\bf w}_k-{\bf w}_\ell)
+(({\bf w}_k-{\bf w}_\ell)\cdot \nabla ){\bf v}_\varepsilon\right\|_{{H}^{-1}(\Omega)^n}\nonumber\\
&\hspace{-2em}\leq {c_1}\left(\|{\bf w}_k\|_{H^1(\Omega)^n} + \|{\bf w}_\ell\|_{\bn H^1(\Omega)^n}
+ 2\|{\bf v}_\varepsilon\|_{H^1(\Omega)^n}
\right)
\left\|{\bf w}_k-{\bf w}_\ell\right\|_{L^4(\Omega)^n}
\nonumber\\
&\hspace{-2em} {\leq 2c_1}\left(M+\|{\bf v}_\varepsilon\|_{{H}^{1}(\Omega)^n}\right)
\left\|{\bf w}_k-{\bf w}_\ell\right\|_{L^4(\Omega)^n}.
\end{align}
This inequality, combined with the property that $\{{\bf w}_k\}_{k\in {\mathbb N}}$ is a Cauchy sequence in the space $L^4(\Omega )^n$, implies that $\{\mathbf F_\varepsilon{\bf w}_k\}_{k\in \mathbb N}$ is a Cauchy sequence in the space ${H}^{-1}(\Omega )^n$.
Therefore,
$\mathbf F_\varepsilon:\mathring{H}_{\rm{div}}^1(\Omega )^n\to {H}^{-1}(\Omega )^n$
is a compact operator.
Hence, the operator
${\mathbf U}_{\varepsilon}=\bs{\mathcal U}\,\mathbf F_\varepsilon:
\mathring{H}_{\rm{div}}^1(\Omega )^n\to \mathring H_{\rm{div}}^1(\Omega )^n$  is also compact, as asserted.

Next, we show that there exists a constant $M_0>0$ such that if ${\bf w}\in \mathring{H}_{\rm{div}}^1(\Omega )^n$ satisfies the equation
\begin{align}
\label{Aset}
&{\bf w}=\lambda{\mathbf U}_{\varepsilon}{\bf w}
\end{align}
for some $\lambda \in [0,1]$, then $\|{\bf w}\|_{{H}^1(\Omega )^n}\le M_0$.
Let us also introduce the function
\begin{align}\label{Aset-q}
q:=\lambda P_{\varepsilon}{\mathbf w}.
\end{align}
By applying the operator $\boldsymbol{\mathcal L}$ to equations \eqref{Aset}-\eqref{Aset-q} and by using relations \eqref{N} and \eqref{P}, we deduce that whenever the pair
$({\mathbf w},\lambda)\in \mathring{H}_{\rm{div}}^1(\Omega )^n\times\R$ satisfies equation \eqref{Aset}, then the equation
\begin{align}\label{transmission-linear-NS-a-equiv-2new}
&\boldsymbol{\mathcal L}({\bf w},q)=\lambda \mathbf F_\varepsilon{\bf w}, 
\end{align}
is also satisfied. (Recall the isomorphism property of operator \eqref{Lisom}.)
Then the first Green identity \eqref{Green-H0} implies the equation
\begin{align}
&\left\langle a_{ij}^{\alpha \beta }E_{j\beta }({\bf w}),E_{i\alpha }({\bf w})\right\rangle _{\Omega }
=-\langle\lambda \mathbf F_\varepsilon{\bf w},{\bf w}\rangle_{\Omega}\,,
\end{align}
which, in view of relation \eqref{Fuphi}, takes the form
\begin{align}
\label{transmission-linear-NS-a-equiv-5new-e}
\big\langle a_{ij}^{\alpha \beta }&E_{j\beta }({\bf w}),E_{i\alpha }({\bf w})\big\rangle _{\Omega }
=\lambda \langle \boldsymbol{\mathfrak F},{\bf w}\rangle _{\Omega }
-\lambda \big\langle a_{ij}^{\alpha \beta }E_{j\beta }({\bf v}_\varepsilon),E_{i\alpha }({\bf w})\big\rangle _{\Omega }\nonumber\\
-&\lambda \big\langle(({\bf w}+{\bf v}_\varepsilon)\cdot \nabla ){\bf w},{\bf w}\big\rangle _{\Omega }
-\lambda \big\langle({\bf v}_\varepsilon\cdot \nabla ){\bf v}_\varepsilon,{\bf w}\big\rangle _{\Omega }
-\lambda \big\langle({\bf w}\cdot \nabla ){\bf v}_\varepsilon,{\bf w}\big\rangle _{\Omega }\,.
\end{align}
Relation \eqref{P-5a} implies that
$\big\langle(({\bf w}+{\bf v}_\varepsilon)\cdot \nabla ){\bf w},{\bf w}\big\rangle _{\Omega }=0$.
Then by using the Korn first inequality \eqref{Korn3-R3}, the ellipticity condition \eqref{mu}, equation \eqref{transmission-linear-NS-a-equiv-5new-e},
the H\"older inequality, relation \eqref{C7}, and the Leray-Hopf inequality \eqref{extension-h}, we obtain for $\lambda\ge 0$ that
\begin{align}
\label{weak-D-1}
&\frac{1}{2}C_{\mathbb A}^{-1}\|\nabla {\bf w}\|_{L^2(\Omega )^{n\times n}}^2\leq
C_{\mathbb A}^{-1}\|{\mathbb E}({\bf w})\|_{L^2(\Omega )^{n\times n}}^2
\leq \left\langle a_{ij}^{\alpha \beta }E_{j\beta }({\bf w}),E_{i\alpha }({\bf w})\right\rangle _{\Omega }\nonumber\\
&\hspace{5em}\leq \lambda |\!|\!|\bs{\mathfrak F}|\!|\!|_{{H}^{-1}(\Omega )^n} \|\nabla {\bf w}\|_{L^2(\Omega )^{n\times n}}
+\lambda \|{\mathbb A}\|\, \|\nabla {\bf w}\|_{L^2(\Omega )^{n\times n}}
\|\nabla {\bf v}_\varepsilon\|_{L^2(\Omega )^{n\times n}}\nonumber\\
&\hspace{9em}+\lambda \|{\bf v}_\varepsilon\|_{L^4(\Omega )^n}^2 \|\nabla {\bf w}\|_{L^2(\Omega )^{n\times n}}
+ {\lambda }{\Bl \varepsilon }\|\nabla {\bf w}\|_{L^2(\Omega )^{n\times n}}^2,
\end{align}
where the norm $|\!|\!|\cdot|\!|\!|_{{H}^{-1}(\Omega )^n}$ is defined in \eqref{norm-3e} and
$\|{\mathbb A}\|$ is the norm of the viscosity tensor coefficient given by \eqref{A}.
Hence, for $\lambda\in[0,1]$,
\begin{align*}
\Big({\frac{1}{2}C_{\mathbb A}^{-1}-{\Bl \varepsilon }}\Big)\|\nabla {\bf w}\|_{L^2(\Omega )^{n\times n}}
\leq |\!|\!|\bs{\mathfrak F}|\!|\!|_{{H}^{-1}(\Omega )^n}+
\|{\mathbb A}\|\, \|\nabla {\bf v}_\varepsilon\|_{L^2(\Omega )^{n\times n}}+ \|{\bf v}_\varepsilon\|_{L^4(\Omega )^n}^2 \,.
\end{align*}
Choosing ${\Bl \varepsilon }< \frac{1}{2}C_{\mathbb A}^{-1}$ in the {\Bl Leray-Hopf's inequality \eqref{extension-h}}, we obtain the estimate
\begin{align}
\label{weak-D}
\hspace{-1em}\|\nabla {\bf w}\|_{L^2(\Omega )^{n\times n}}\!\leq \!\frac{2}{C_{\mathbb A}^{-1}\!-\!2{\Bl \varepsilon }}
\left(|\!|\!|\bs{\mathfrak F}|\!|\!|_{{H}^{-1}(\Omega )^n}\!+\!
\|{\mathbb A}\|\, \|\nabla {\bf v}_\varepsilon\|_{L^2(\Omega )^{n\times n}}\!+\!\|{\bf v}_\varepsilon\|_{L^4(\Omega )^n}^2\right),
\end{align}
that is, $\|{\bf w}\|_{{H}^1(\Omega )^n}\le M_0$, where $M_0$ is given by the right hand side of \eqref{weak-D} multiplied by the equivalence constant between the norm and semi-norm in $\mathring{H}^1(\Omega )^n$.

Therefore, the operator $\mathbf U_\varepsilon :\mathring{H}_{{\rm{div}}}^1(\Omega )^n\to \mathring H_{{\rm{div}}}^1(\Omega )^n$ satisfies the hypothesis of Theorem \ref{L-S-fixed} (with $\mathcal X=\mathring{H}_{\rm{div}}^1(\Omega )^n$), and hence it has a fixed point ${\mathbf u_0}\!\in \!\mathring{H}_{\rm{div}}^1(\Omega )^n$, that is, ${\mathbf u_0}=\mathbf U_\varepsilon {\mathbf u_0}$. Then with $\pi\in L^2(\Omega )/{\mathbb R}$ as in \eqref{transmission-linear-NS-a-equiv-2P}, we obtain that the couple $(\mathbf u_0, \pi)\in\mathring H_{{\rm{div}}}^1(\Omega )^n\times L^2(\Omega )/{\mathbb R}$ satisfies the nonlinear equation \eqref{transmission-linear-NS-a-equiv-1}.
Consequently, 
the couple $(\mathbf u, \pi)=({\bf v}_\varepsilon+{\mathbf u_0},\pi )\in H^1(\Omega )^n\times L^2(\Omega )/{\mathbb R}$ is a solution of the nonlinear Dirichlet problem \eqref{int-NS-D}. (Recall that ${\bf v}_\varepsilon$ is an extension to $H_{{\rm{div}}}^{1}(\Omega )^n$ of the function ${\bs\varphi}\in H_{\boldsymbol\nu }^{\frac{1}{2}}(\partial \Omega )^n$, and, thus, it satisfies the Dirichlet condition \eqref{extension-h-1}.)
\qed
\end{proof}

\section{\bf Dirichlet-transmission problem for the anisotropic Navier-Stokes system in a bounded Lipschitz domain with a transversal Lipschitz interface}

In this section we show the  existence of weak solutions of Dirichlet-transmission problems for the anisotropic Navier-Stokes system with data in $L^2$-based Sobolev spaces in a bounded Lipschitz domain in ${\mathbb R}^n$, $n=2,3$, satisfying Assumption \ref{interface-Sigma}.
First, we analyze a Dirichlet-transmission problem for the incompressible Navier-Stokes system with general PDE right hand sides and a jump of conormal derivatives on the transversal Lipschitz interface. We reduce this nonlinear problem to a Dirichlet problem for the Navier-Stokes system whose analysis is based on the  Leray-Hopf inequality and the Leray-Schauder fixed point theorem.
Then, we study a Dirichlet-transmission problem for the anisotropic Navier-Stokes system in a compressible framework with non-homogeneous Dirichlet condition and trace and conormal derivative jumps across the internal Lipschitz interface. We use a Bogovskii-type result established in Lemma \ref{L.BT}, some useful estimates and the Leray-Schauder fixed point theorem to show the existence of a weak solution to this nonlinear problem.
In the case of all small data, the uniqueness of the weak solution is also established.

\subsection{\bf Dirichlet-transmission problem in a bounded Lipschitz domain with conormal derivative jump on a transversal Lipschitz interface}

Let us consider the following {\it {Dirichlet-transmission problem} for the incompressible anisotropic Navier-Stokes system with a prescribed conormal derivative jump but without velocity jump on the interface},
\begin{equation}
\label{int-NS-DT-0}
\left\{
\begin{array}{ll}
{\boldsymbol{\mathcal L}({\bf u}^+,\pi ^+)}=\widetilde{\bf f}^+ |_{\Omega ^{+}}+{\bl ({\bf u}^+\cdot \nabla ){\bf u}^+}\,, \ \ {\rm{div}} \, {\bf u}^+=0 & \mbox{ in } \Omega ^{+},\\
{\boldsymbol{\mathcal L}({\bf u}^-,\pi ^-)}=\widetilde{\bf f}^- |_{\Omega ^{-}}
+{\bl ({\bf u}^-\cdot \nabla ){\bf u}^-}\,, \ \ {\rm{div}} \, {\bf u}^-=0  & \mbox{ in } \Omega ^-,\\
{(\gamma _{_{\Omega ^+}}{\bf u}^+)|_{_{\Sigma }}=(\gamma _{_{\Omega ^-}}{\bf u}^{-})|_{_{\Sigma }}} &  \mbox{ on } \Sigma \,,\\
\big({\bf t}_{{\Omega ^+}}\big({\bf u}^+,\pi ^+;\widetilde{\bf f}^+
+\mathring{E}_{\Omega ^+\to \Omega ^-}({\bf u}^+\cdot \nabla ){\bf u}^+\big)\big)\big|_{\Sigma }\\
\hspace{3em}+ \, \big({\bf t}_{{\Omega ^-}}\big({\bf u}^-,\pi ^-;\widetilde{\bf f}^-
+\mathring{E}_{\Omega ^-\to \Omega ^+}({\bf u}^-\cdot \nabla ){\bf u}^-\big)\big)\big|_{\Sigma }
={{\boldsymbol \psi }_{_\Sigma }} &  \mbox{ on } \Sigma \,,\\
({\gamma _{_{\Omega ^+}}}{\bf u}^+)|_{\Gamma ^+}=\bs\varphi |_{\Gamma ^+}&  \mbox{ on } \Gamma ^+\,,\\
({\gamma _{_{\Omega ^-}}}{\bf u}^-)|_{\Gamma ^-}=\bs\varphi |_{\Gamma ^-} &  \mbox{ on } \Gamma ^-\,
\end{array}\right.
\end{equation}
with the unknown $({\bf u}^+,\pi ^+,{\bf u}^-,\pi ^-)\in {\mathfrak X}_{\Omega ^+,\Omega ^-}$ and the given data
$({\widetilde{\bf f}^+,\widetilde{\bf f}^-},{{\boldsymbol \psi }_{_\Sigma }},{\bs\varphi})
\in {\left({H}_{\Gamma ^+}^1(\Omega ^+)^n\right)'
\!\times\!\left({H}_{\Gamma ^-}^1(\Omega ^-)^n\right)'}\!
\times \!H^{-\frac{1}{2}}(\Sigma )^n\times H_{\boldsymbol\nu }^{\frac{1}{2}}(\partial \Omega )^n$.
Recall that ${\mathfrak X}_{\Omega ^+,\Omega ^-}$ is the space defined in \eqref{Omega-pm}, and note that $\mathring E_{\Omega^\pm\to \Omega}$ in \eqref{int-NS-DT-0} is the operator of extension by zero from $\Omega ^\pm $ to $\Omega$.

\subsection*{{\bf Existence of a weak solution}}

Let $({\bf u}^+,\pi ^+,{\bf u}^-,\pi ^-)\in {\mathfrak X}_{\Omega ^+,\Omega ^-}$.
Assume that ${\bf u}^+$ and ${\bf u}^-$ satisfy the homogeneous interface condition
${(\gamma _{_{\Omega ^+}}{\bf u}^+)|_{_{\Sigma }}-(\gamma _{_{\Omega ^-}}{\bf u}^{-})|_{_{\Sigma }}={\bf 0}}$ on $\Sigma $.
Then by Lemma \ref{extention}, there exists a unique pair
$({\bf u},\pi)\in {{H}^1(\Omega )^n}\times\!L^2(\Omega )/{\mathbb R}$ such that
\begin{align}
\label{u-new}
{\bf u}|_{\Omega ^+}={\bf u}^+,\quad
{\bf u}|_{\Omega ^-}={\bf u}^-,\quad
\pi|_{\Omega ^+} =\pi^+,\quad
\pi|_{\Omega ^-} =\pi^-.
\end{align}
Let also $\boldsymbol{\mathfrak F}\in {\Bl {H}^{-1}(\Omega )^n}=\big(\mathring{H}^{1}(\Omega )^n\big)'\subset
\big(\mathring{H}_{{\rm{div}}}^{1}(\Omega )^n\big)'$ be such that
\begin{align}
\label{F-a-new}
\langle \boldsymbol{\mathfrak F},{\bf v}\rangle _{\Omega }:=-\big\langle \tilde{\bf f}^++\tilde{\bf f}^-,{\bf v}\big\rangle _{\Omega }+\langle \bs\psi_{_{\Sigma }},\gamma _{_{\Sigma }}{\bf v}\rangle _{\Sigma }\,, \ \ \forall \, {\bf v}\in \mathring{H}^{1}(\Omega )^n\,,
\end{align}
that is,
$\bs{\mathfrak F}=-(\tilde{\bf f}^++\tilde{\bf f}^-)+\gamma _{_\Sigma }^*{\bs\psi_{_{\Sigma }}}$.
Here
$\gamma _{_{\Sigma }}^*:{H}^{-\frac{1}{2}}(\Sigma )^n\!\to \!{H}^{-1}(\Omega )^n$
denotes the adjoint of the trace operator
$\gamma _{_{\Sigma }}:{\mathring{H}^{1}(\Omega )^n}\to \widetilde{H}^{\frac{1}{2}}(\Sigma )^n$
defined by \eqref{gamma-Sigma}, and the support of $\gamma _{_{\Sigma }}^*{\bs\psi_{_{\Sigma }}}$ is a subset of
$\overline\Sigma $.

An argument similar to that for problem \eqref{Dirichlet-var-Stokes-a} implies the following result.
\begin{lemma}\label{equiv-nonlin}
The nonlinear Dirichlet-transmission problem \eqref{int-NS-DT-0}
is equivalent, in the sense of relations \eqref{u-new}, to the nonlinear Dirichlet problem \eqref{int-NS-D} with
$\bs{\mathfrak F}=-(\tilde{\bf f}^++\tilde{\bf f}^-)+\gamma _{_\Sigma }^*{\bs\psi_{_{\Sigma }}}$.
\end{lemma}
\begin{proof}
Assume that $\left({\bf u}^+,\pi ^+,{\bf u}^-,\pi ^-\right)\in {\mathfrak X}_{\Omega ^+,\Omega ^-}$ satisfies the nonlinear Dirichlet-transmission problem \eqref{int-NS-DT-0}.
Let $({\bf u},\pi )\in \mathring{H}^1(\Omega )^n\times L^2(\Omega )/{\mathbb R}$
be the pair defined by relations \eqref{u-new} (cf. Lemma \ref{extention}). Then the Green identity \eqref{Green-particular} and relation \eqref{F-a-new} for $\boldsymbol{\mathfrak F}$ give the following weak equation:
\begin{align}
\label{transmission-linear-NS-a}
\left\langle a_{ij}^{\alpha \beta }E_{j\beta }({\bf u}),E_{i\alpha }({\bf w})\right\rangle _{\Omega }&\!+\!\left\langle({\bf u}\cdot \nabla ){\bf u},{\bf w}\right\rangle _{\Omega }
-\langle {\rm{div}}\, {\bf w},\pi \rangle _{\Omega }\nonumber\\
&={-\big\langle \tilde{\bf f}^++\tilde{\bf f}^-,{\bf w}\big\rangle _{\Omega }+\big\langle \boldsymbol\psi_{_\Sigma},{\gamma _{_{\Sigma }}}{\bf w} \big\rangle _{\Sigma }}\,, \forall \, {\bf w} \in {\mathring{H}^{1}(\Omega )^n}\,,
\end{align}
which implies the distributional form of the first Navier-Stokes equation in \eqref{int-NS-D}.

Note that the variational form of the nonlinear Dirichlet problem \eqref{int-NS-D} is given by equation \eqref{transmission-linear-NS-a} complemented by equations
\begin{align*}
\begin{array}{lll}
\langle {\rm{div}}\, {\bf u},q\rangle _{\Omega }=0, \quad& \forall \, q\in L^2(\Omega )/{\mathbb R}\,, \\
{\gamma _{_{\Omega }}}{\bf u}={\bs\varphi}&  \mbox{ on } \partial \Omega \,.
\end{array}
\end{align*}

Conversely, assume that $({\bf u},\pi )\!\in \!{H}^1(\Omega )^n\!\times \! L^2(\Omega )/{\mathbb R}$ satisfies the Dirichlet problem \eqref{int-NS-D} and let $({\bf u}^\pm ,\pi ^\pm )=({\bf u}|_{\Omega ^\pm },\pi |_{\Omega ^\pm })$.
Then by the Green identity \eqref{Green-H0} and relation \eqref{F-a-new}, the first equation in \eqref{int-NS-D} can be written in the equivalent variational form
\begin{align}
\label{conormal-derivative-particular-ae}
&\big\langle a_{ij}^{\alpha \beta }E_{j\beta }({\bf u}^+),E_{i\alpha }({\bf w}^+)\big\rangle _{\Omega ^+}-\big\langle \pi ^+,{\rm{div}}\, {\bf w}^+\big\rangle _{\Omega ^+}
-\big\langle ({\bf u}^+\cdot \nabla ){\bf u}^+,{\bf w}^+\big\rangle _{\Omega ^+}\nonumber\\
&+\big\langle a_{ij}^{\alpha \beta }E_{j\beta }({\bf u}^-),E_{i\alpha }({\bf w}^-)\big\rangle _{\Omega ^-}-\big\langle \pi ^-,{\rm{div}}\, {\bf w}^-\big\rangle _{\Omega ^-}
-\big\langle ({\bf u}^-\cdot \nabla ){\bf u}^-,{\bf w}^-\big\rangle _{\Omega ^-}\nonumber\\
&-\big\langle \tilde{\bf f}^+,{\bf w}\big\rangle _{\Omega ^+}-\big\langle \tilde{\bf f}^-,{\bf w}\big\rangle _{\Omega ^-}+\langle {\bs\psi_{_{\Sigma }}},\gamma _{_{\Sigma }}{\bf w}\rangle _{\Sigma }=0\,, \ \forall \, {\bf w}\in \mathring{H}^1(\Omega )^n\,.
\end{align}
Since the spaces ${\mathcal D}(\Omega ^\pm )^n$ are subspaces of $\mathring{H}^{1}(\Omega)^n$,
the (distributional form of the) anisotropic Navier-Stokes equation in \eqref{int-NS-DT-0}, in each of the domains $\Omega ^+$ and $\Omega ^-$, follows from equation \eqref{conormal-derivative-particular-ae} written for all
${\bf w}\in {\mathcal D}(\Omega ^+)^n$ and ${\bf w}\in {\mathcal D}(\Omega ^-)^n$, respectively.
The second equation in \eqref{int-NS-D} implies the equation ${\rm{div}}\, {\bf u}^\pm =0$ in $\Omega ^\pm $.
Thus, $\left({\bf u}^+,\pi ^+,{\bf u}^-,\pi ^-\right)$ satisfies the anisotropic Navier-Stokes system in $\Omega ^+\cup \Omega ^-$,
the Dirichlet boundary condition $(\gamma _{_{\Omega ^\pm }}{\bf u}^\pm )|_{_{\Gamma ^\pm }}=\boldsymbol \varphi |_{\Gamma ^\pm }$ on $\Gamma ^\pm $, and the interface condition $(\gamma _{_{\Omega ^+}}{\bf u}^+)|_{_{\Sigma }}={(\gamma _{_{\Omega ^-}}{\bf u}^{-})|_{_{\Sigma }}}$ on $\Sigma $.
Then substituting \eqref{conormal-derivative-particular-ae} into the Green identity \eqref{Green-particular}, we obtain the equation
\begin{align}
\label{Sigma-abe}
&\big\langle \big({\bf t}_{{\Omega ^+}}\big({\bf u}^+,\pi ^+;\tilde{\bf f}^+ +\mathring{E}_{\Omega ^+\to \Omega ^-}({\bf u}^+\cdot \nabla ){\bf u}^+\big)\big)\big|_{\Sigma }\\
&\hspace{1em}+\, \big({\bf t}_{{\Omega ^-}}\big({\bf u}^-,\pi ^-;\tilde{\bf f}^- +\mathring{E}_{\Omega ^-\to \Omega ^+}({\bf u}^-\cdot \nabla ){\bf u}^-\big)\big)\big|_{\Sigma },
(\gamma _{_{\Omega }}{\bf w})|_{_{\Sigma }}\big\rangle _{_{\Sigma }}
=\big\langle {\bs\psi_{_{\Sigma }}},{\gamma _{_{\Sigma }}{\bf w}}\big\rangle _\Sigma \,.\nonumber
\end{align}
In view of Lemma \ref{gamma-Sigma-surj}, formula \eqref{Sigma-abe}
can be written in the equivalent form
\begin{align*}
&\big\langle \big({\bf t}_{{\Omega ^+}}\big({\bf u}^+,\pi ^+;\tilde{\bf f}^+ +\mathring{E}_{\Omega ^+\!\to \!\Omega ^-}({\bf u}^+\cdot \nabla ){\bf u}^+\big)\big|_{\Sigma }\nonumber\\
&+{\bf t}_{{\Omega ^-}}\big({\bf u}^-,\pi ^-;\tilde{\bf f}^- +\mathring{E}_{\Omega ^-\!\to \!\Omega ^+}({\bf u}^-\cdot \nabla ){\bf u}^-\big)\big)\big|_{\Sigma },
\bs\phi\big\rangle _{_{\Sigma }}
\!=\!\big\langle {\bs\psi_{_{\Sigma }}},{\bs\phi}\big\rangle _\Sigma\,,\
\forall \, {{\bs\phi}\!\in \! \widetilde{H}^{\frac{1}{2}}(\Sigma )^n}.
\end{align*}
Therefore,
$\big({\bf t}_{\Omega ^+}({\bf u}^+,\pi^+;\tilde{\bf f}^+)
{+}{{\bf t}_{\Omega ^-}}({\bf u}^-,\pi^-;\tilde{\bf f}^-)\big)\big|_{\Sigma }
={\bs\psi_{_{\Sigma }}}$ on $\Sigma $.

Consequently, problems \eqref{int-NS-DT-0} and \eqref{int-NS-D} are equivalent, as asserted. \qed
\end{proof}

\begin{theorem}
\label{int-D-NS-var}
Let $\Omega \!\subset \!{\mathbb R}^n$, $n\!\in \!\{2,3\}$, be a bounded Lipschitz domain satisfying Assumption $\ref{interface-Sigma}$.
Let conditions \eqref{Stokes-1}-\eqref{mu} hold.
Then for all given data $(\tilde{\bf f}^+,\tilde{\bf f}^-,{\boldsymbol \psi }_{_\Sigma },{\bs\varphi})$ in the space
$\left({H}_{\Gamma ^+}^1(\Omega ^+)^n\right)'\times \left({H}_{\Gamma ^-}^1(\Omega ^-)^n\right)'\times H^{-\frac{1}{2}}(\Sigma )^n\times H^{\frac{1}{2}}(\partial \Omega )^n$,
the Dirichlet-transmission problem \eqref{int-NS-DT-0} for the anisotropic Navier-Stokes system has a weak solution
$({\bf u}^+,\pi ^+,{\bf u}^-,\pi ^-)\in {\mathfrak X}_{\Omega ^+,\Omega ^-}$ defined by relations \eqref{u-new} in terms of the solution $({\bf u},\pi )$ of the nonlinear Dirichlet problem \eqref{int-NS-D} with
$\bs{\mathfrak F}=-(\tilde{\bf f}^++\tilde{\bf f}^-)+\gamma _{_\Sigma }^*{\bs\psi_{_{\Sigma }}}$.
\end{theorem}
\begin{proof}
The pair $({\bf u},\pi )\in {H}_{\rm{div}}^1(\Omega )^n\times L^2(\Omega )/\mathbb R$ is a solution of the variational problem \eqref{transmission-linear-NS-a}, and then, in view of Lemma \ref{equiv-nonlin}, the functions $({\bf u}^+,\pi ^+,{\bf u}^-,\pi^-)\!\in \!{\mathfrak X}_{\Omega ^+,\Omega ^-}$ defined by \eqref{u-new} satisfy the nonlinear problem \eqref{int-NS-DT-0} in the distribution sense. \qed
\end{proof}

\subsection*{\bf Uniqueness result for the Dirichlet-transmission problem \eqref{int-NS-DT-0}}
Next we show that an additional constraint to the given data of the nonlinear Dirichlet-transmission problem \eqref{int-NS-DT-0} leads to the uniqueness of the weak solution of this problem.

Recall that ${\gamma _{_{\Sigma }}^*}:H^{-\frac{1}{2}}(\Sigma )^n\to {H}^{-1}(\Omega )^n$ is the adjoint of the trace operator ${\gamma _{_{\Sigma }}}:\mathring{H}^{1}(\Omega )^n\to \widetilde{H}^{\frac{1}{2}}(\Sigma )^n$ {defined by \eqref{gamma-Sigma}}, and that $C_{\mathbb A}$ is the ellipticity constant in \eqref{mu}.
On the other hand, in view of Lemma \ref{B-D}, there exists an extension ${\bf v}_{\bs\varphi }$ of $\bs\varphi \in H_{\bs \nu}^{\frac{1}{2}}(\partial \Omega )^n$ to $H_{{\rm{div}}}^1(\Omega )$, that is, $\gamma _{_\Omega }{\bf v}_{\bs\varphi}=\bs\varphi $ on $\partial \Omega $, and
\begin{align}
\label{v-0-1}
\|\nabla {\bf v}_{\bs\varphi}\|_{L^2(\Omega )^{n\times n}}
{\leq\|{\bf v}_{\bs\varphi}\|_{H^1(\Omega )^{n}}}
\leq  C\|\bs\varphi\|_{H^{\frac{1}{2}}(\partial \Omega )^n}
\end{align}
with some constant $C=C(\Omega,n)>0$.

Then we prove the following uniqueness result (see also \cite[Lemma 3.1]{Seregin}
in the isotropic case \eqref{isotropic} with $\mu =1$ and homogeneous Dirichlet condition, and \cite[Theorem 4.2]{K-M-W-2} for a nonlinear transmission problem in a pseudostress approach).
\begin{theorem}
\label{well-posed-N-S-Stokes-small}
Let $n=2,3$ and $\Omega \subset {\mathbb R}^n$ be a bounded Lipschitz domain satisfying Assumption $\ref{interface-Sigma}$.
Let conditions \eqref{Stokes-1}-\eqref{mu} are satisfied.
Let $(\tilde{\bf f}^+,\tilde{\bf f}^-,{\bs\varphi},{\boldsymbol \psi }_{_\Sigma })$ be given in the space
$\left({H}_{\Gamma ^+}^1(\Omega ^+)^n\right)'
\times  \left({H}_{\Gamma ^-}^1(\Omega ^-)^n\right)'
\times H_{\boldsymbol\nu }^{\frac{1}{2}}(\partial \Omega )^n
\times H^{-\frac{1}{2}}(\Sigma )^n$.
Let
\begin{align}
\label{uniqueness-a}
&c_0^2|\!|\!|\bs{\mathfrak F}|\!|\!|_{{H}^{-1}(\Omega )^n}
+(c_0^2C\|{\mathbb A}\|+C_{\mathbb A}^{-1}Cc_0c_1)\|\bs\varphi \|_{H^{\frac{1}{2}}(\partial \Omega )^n}
<\frac{1}{4}C_{\mathbb A}^{-2},
\end{align}
where $C_{\mathbb A}$, $C$, $c_0$, and $c_1$ are the constants in \eqref{mu},  \eqref{v-0-1}, \eqref{L4}, and \eqref{SET}, respectively,
while $\boldsymbol{\mathfrak F}=-(\tilde{\bf f}^+ +\tilde{\bf f}^-)+{\gamma _{_{\Sigma }}^*(\boldsymbol\psi_{_\Sigma})}$.
Then the nonlinear problem \eqref{int-NS-DT-0} has {only one solution $({\bf u}^+,\pi ^+,{\bf u}^-,\pi ^-)\in {\mathfrak X}_{\Omega ^+,\Omega ^-}$.}
\end{theorem}
\begin{proof}
{The solution existence is implied by Theorem \ref{int-D-NS-var}.
To prove that it is unique under the theorem conditions, let us assume}
that ${\bf u}^{(1)},{\bf u}^{(2)}\in {\rd {H}_{\rm{div}}^1(\Omega )^n}$ are the velocities in two
solutions of the nonlinear problem \eqref{int-NS-DT-0} {in the sense of \eqref{u-new}}.
Let us write them in the form
\begin{align}
{\bf u}^{(i)}={\bf v}_{\bf \varphi}+{\bf u}_0^{(i)},\ \ i=1,2,
\end{align}
where ${\bf v}_{\bs\varphi }\in H_{{\rm div}}^1(\Omega )^n$ satisfies the relation $\gamma _{_\Omega }{\bf v}_{\bs\varphi}=\bs\varphi $ on $\partial \Omega $ and estimate \eqref{v-0-1}, while {\rd ${\bf u}_0^{(1)},{\bf u}_0^{(2)}\in {\rd \mathring{H}_{\rm{div}}^1(\Omega )^n}$} satisfy equations \eqref{transmission-linear-NS-a-equiv-1}-\eqref{Fuphi} with ${\bf v}_{\bs\varphi }$ instead of ${\bf v}_{\varepsilon }$.

Let us denote $\overline{\bf u}:={\bf u}^{(1)}-{\bf u}^{(2)}={\bf u}_0^{(1)}-{\bf u}_0^{(2)}$, $\overline\pi:=\pi^{(1)}-\pi^{(2)}$, where $\pi ^{(i)}$ is the pressure term corresponding to ${\bf u}^{(i)}$, $i=1,2$.
Using the first equations in the two upper lines of \eqref{int-NS-DT-0},
we obtain
\begin{align}
\label{transmission-linear-NS-a-equiv-u}
\bs{\mathcal L}(\overline{\bf u},\overline\pi)=&
({\bf u}^{(1)}\cdot \nabla ){\bf u}^{(1)}-({\bf u}^{(2)}\cdot \nabla ){\bf u}^{(2)}.
\end{align}
This implies that
\begin{align}
\label{NS-var-eq-int-unique-0}
\big\langle a_{ij}^{\alpha \beta }E_{j\beta }(\overline{\bf u}&),
E_{i\alpha }(\overline{\bf u})\big\rangle _{\Omega }\!
=\!-\big\langle\big(\overline{\bf u}\cdot \nabla \big){\bf u}^{(1)},\overline{\bf u}\big\rangle _{\Omega }
-\big\langle({\bf u}^{(2)}\cdot \nabla )\overline{\bf u},\overline{\bf u}\big\rangle _{\Omega }.
\end{align}
Moreover, identity \eqref{P-5a}
and the assumption that $\overline{\bf u}\in \mathring{H}_{\rm{div}}^1({\Omega })^n$ show that the last term in the right-hand side of \eqref{NS-var-eq-int-unique-0} equals zero.
Therefore, equation \eqref{NS-var-eq-int-unique-0} reduces to
\begin{align}
\label{NS-var-eq-int-uniqe-1}
\big\langle a_{ij}^{\alpha \beta }E_{j\beta }(\overline{\bf u}&),E_{i\alpha }(\overline{\bf u})\big\rangle _{\Omega }
=-\big\langle\big(\overline{\bf u}\cdot \nabla \big){\bf u}^{(1)},\overline{\bf u}\big\rangle _{\Omega }.
\end{align}
On the other hand, estimate \eqref{a-1-v2-S} implies that
\begin{align}
\label{P-11-unique}
\!\!\!\!\!\|\nabla \overline{\bf u}\|_{L^2(\Omega )^{n\times n}}^2\,\leq
\!2C_{\mathbb A}\big\langle a_{ij}^{\alpha \beta }E_{j\beta }(\overline{\bf u}),E_{i\alpha }(\overline{\bf u})\big\rangle _{\Omega }\,,
\end{align}
and by the H\"{o}lder inequality and inequalities \eqref{L4} and {\rd \eqref{v-0-1}} together with \eqref{antisym} and \eqref{SET}, we obtain
\begin{align}
\label{NS-var-eq-int-unique-2a}
\big|\big\langle\big(&\overline{\bf u}\cdot \nabla \big){\bf u}^{(1)},\overline{\bf u}\big\rangle _{\Omega }\big|
=\big|\big\langle\big(\overline{\bf u}\cdot \nabla \big)\overline{\bf u},{\bf u}^{(1)}\big\rangle_{\Omega }\big|
\leq \|\overline{\bf u}\|_{L^4(\Omega )^n}\|\nabla \overline{\bf u}\|_{L^2(\Omega )^{n\times n}}
\|{\bf u}^{(1)}\|_{L^4(\Omega )^n}\nonumber\\
&\leq \|\overline{\bf u}\|_{L^4(\Omega )^n}\|\nabla \overline{\bf u}\|_{L^2(\Omega )^{n\times n}}
(\|{\bf u}_0^{(1)}\|_{L^4(\Omega )^n}+\|{\bf v}_{\bs\varphi}\|_{L^4(\Omega )^{n}})\nonumber\\
&\leq c_0\|\nabla \overline{\bf u}\|_{L^2(\Omega )^{n\times n}}^2\left(c_0\|\nabla {\bf u}_0^{(1)}\|_{L^2(\Omega )^{n\times n}}
 +Cc_1\|\bs\varphi\|_{H^\frac{1}{2}(\partial \Omega )^n}\right).
\end{align}
Hence
\begin{align}
\label{P-11-unique-a}
\hspace{-1.15em}\|\nabla \overline{\bf u}\|_{L^2(\Omega )^{n\times n}}^2\!\leq
\!2C_{\mathbb A}{c_0}\|\nabla \overline{\bf u}\|_{L^2(\Omega )^{n\times n}}^2
\left({c_0}\|\nabla {\bf u}_0^{(1)}\|_{L^2(\Omega )^{n\times n}}\!+\!Cc_1\|\bs\varphi\|_{H^\frac{1}{2}(\partial \Omega )^n}\right).
\end{align}

Moreover, an estimate similar to \eqref{weak-D-1} with $\lambda=1$, combined with estimates \eqref{L4}, \eqref{SET}, \eqref{antisym} and \eqref{v-0-1} together with relation \eqref{transmission-linear-NS-a-equiv-5new-e} imply that
\begin{align*}
&\frac{1}{2}C_{\mathbb A}^{-1}\|\nabla {\bf u}_0^{(1)}\|_{L^2(\Omega )^{n\times n}}^2\leq
C_{\mathbb A}^{-1}\|{\mathbb E}({\bf u}_0^{(1)})\|_{L^2(\Omega )^{n\times n}}^2
\leq \left\langle a_{ij}^{\alpha \beta }E_{j\beta }({\bf u}_0^{(1)}),E_{i\alpha }({\bf u}_0^{(1)})\right\rangle _{\Omega }\nonumber\\
&\hspace{2em}\leq |\!|\!|\bs{\mathfrak F}|\!|\!|_{{H}^{-1}(\Omega )^n} \|\nabla {\bf u}_0^{(1)}\|_{L^2(\Omega )^{n\times n}}
+\|{\mathbb A}\|\, \|\nabla {\bf u}_0^{(1)}\|_{L^2(\Omega )^{n\times n}}
\|\nabla {\bf v}_{\bs\varphi}\|_{L^2(\Omega )^{n\times n}}\nonumber\\
&\hspace{2em}+\|{\bf v}_{\bs\varphi}\|_{L^4(\Omega )^n}^2 \|\nabla {\bf u}_0^{(1)}\|_{L^2(\Omega )^{n\times n}}
+ \|\nabla {\bf u}_0^{(1)}\|_{L^2(\Omega )^{n\times n}}\|{\bf u}_0^{(1)}\|_{L^4(\Omega )^{n}}
\|{\bf v}_{\bs\varphi}\|_{L^4(\Omega )^n}
\nonumber\\
&\hspace{2em}\leq |\!|\!|\bs{\mathfrak F}|\!|\!|_{{H}^{-1}(\Omega )^n} \|\nabla {\bf u}_0^{(1)}\|_{L^2(\Omega )^{n\times n}}+C\|{\mathbb A}\|\, \|\nabla {\bf u}_0^{(1)}\|_{L^2(\Omega )^{n\times n}}\|\bs\varphi \|_{H^{\frac{1}{2}}(\partial \Omega )^n}\\
&\hspace{3em}
+{C^2c_1^2}\|\bs\varphi \|_{H^{\frac{1}{2}}(\partial \Omega )^n}^2 \|\nabla {\bf u}_0^{(1)}\|_{L^2(\Omega )^{n\times n}}
+ Cc_1c_0\|\bs\varphi \|_{H^{\frac{1}{2}}(\partial \Omega )^n}\|\nabla {\bf u}_0^{(1)}\|_{L^2(\Omega )^{n\times n}}^2\,.\nonumber
\end{align*}

Thus, we obtain the estimate
\begin{align}
\label{weak-D-1ua}
&\left(\frac{1}{2}C_{\mathbb A}^{-1}-Cc_1c_0\|\bs\varphi \|_{H^{\frac{1}{2}}(\partial \Omega )^n}\right)\|\nabla {\bf u}_0^{(1)}\|_{L^2(\Omega )^{n\times n}}\nonumber\\
&\hspace{5em}\leq
|\!|\!|\bs{\mathfrak F}|\!|\!|_{{H}^{-1}(\Omega )^n}+C\|{\mathbb A}\|\, \|\bs\varphi \|_{H^{\frac{1}{2}}(\partial \Omega )^n}
+C^2c_1^2\|\bs\varphi \|_{H^{\frac{1}{2}}(\partial \Omega )^n}^2\,.
\end{align}

From \eqref{uniqueness-a} we have,
\begin{align*}
&C_{\mathbb A}^{-1}Cc_0c_1\|\bs\varphi \|_{H^{\frac{1}{2}}(\partial \Omega )^n}
<\frac{1}{4}C_{\mathbb A}^{-2}-c_0^2|\!|\!|\bs{\mathfrak F}|\!|\!|_{{H}^{-1}(\Omega )^n}
-c_0^2C\|{\mathbb A}\|\|\bs\varphi \|_{H^{\frac{1}{2}}(\partial \Omega )^n}
<\frac{1}{2}C_{\mathbb A}^{-2}
\end{align*}
This implies that the term
$\left(\frac{1}{2}C_{\mathbb A}^{-1}-Cc_1c_0\|\bs\varphi \|_{H^{\frac{1}{2}}(\partial \Omega )^n}\right)$
is positive.
Dividing \eqref{weak-D-1ua}  by this term and substituting the obtained inequality to  \eqref{P-11-unique-a}, we obtain
\begin{align*}
\|\nabla &\overline{\bf u}\|_{L^2(\Omega )^{n\times n}}^2
\leq 2C_{\mathbb A}c_0\|\nabla \overline{\bf u}\|_{L^2(\Omega )^{n\times n}}^2\
\\
&\times\Big[c_0\,\frac{|\!|\!|\bs{\mathfrak F}|\!|\!|_{{H}^{-1}(\Omega )^n}
+C\|{\mathbb A}\|\, \|\bs\varphi \|_{H^{\frac{1}{2}}(\partial \Omega )^n}
+{C^2c_1^2}\|\bs\varphi \|_{H^{\frac{1}{2}}(\partial \Omega )^n}^2}
{\frac{1}{2}C_{\mathbb A}^{-1}-Cc_1c_0\|\bs\varphi \|_{H^{\frac{1}{2}}(\partial \Omega )^n}}
+Cc_1\|\bs\varphi \|_{H^{\frac{1}{2}}(\partial \Omega )^n}\Big].
\end{align*}
This finally reduces to
\begin{align}
&\frac{1}{4}C_{\mathbb A}^{-2}\|\nabla \overline{\bf u}\|_{L^2(\Omega )^{n\times n}}^2
\leq \nonumber\\
&\hspace{0.5em}\left[c_0^2|\!|\!|\bs{\mathfrak F}|\!|\!|_{{H}^{-1}(\Omega )^n}
+c_0^2C\|{\mathbb A}\|\, \|\bs\varphi \|_{H^{\frac{1}{2}}(\partial \Omega )^n}
+C_{\mathbb A}^{-1}Cc_0c_1\|\bs\varphi \|_{H^{\frac{1}{2}}(\partial \Omega )^n}\right]
\|\nabla \overline{\bf u}\|_{L^2(\Omega )^{n\times n}}^2\,,\nonumber
\end{align}
In view of assumption \eqref{uniqueness-a}, this is possible only if $\nabla \overline{\bf u}=0$, and since
$\overline{\bf u}\in \mathring{H}^1(\Omega )^n$, we obtain that
${\bf u}^{(1)}={\bf u}^{(2)}$ in $\Omega $.
{\Bl Moreover, equation \eqref{transmission-linear-NS-a-equiv-u} reduces to the equation $\nabla\overline\pi=0$,
that is, $\pi^{(1)}=\pi^{(2)}$ in $L^2(\Omega )/{\mathbb R}$.}
\qed
\end{proof}

In the special case of zero Dirichlet datum on $\partial \Omega $, we obtain the following result.
\begin{corollary}
\label{well-posed-N-S-Stokes-small-1}
Let $n=2,3$ and $\Omega \subset {\mathbb R}^n$ be a bounded Lipschitz domain satisfying Assumption $\ref{interface-Sigma}$. Let $\mathbb A$ satisfy conditions \eqref{Stokes-1}-\eqref{mu}.
Let $\big({\widetilde{\bf f}},\boldsymbol\psi_{_\Sigma}\big)\in {H}^{-1}(\Omega )^n\times H^{-\frac{1}{2}}(\Sigma )^n$ and $\boldsymbol{\mathfrak F}\in {H}^{-1}(\Omega )^n$ be given by $\boldsymbol{\mathfrak F}=-(\tilde{\bf f}^++\tilde{\bf f}^-)+{\gamma _{_{\Sigma }}^*(\boldsymbol\psi_{_\Sigma})}$. If 
\begin{align}
\label{uniqueness-2}
c_0^2|\!|\!|\bs{\mathfrak F}|\!|\!|_{{H}^{-1}(\Omega )^n}<\frac{1}{4}C_{\mathbb A}^{-2},
\end{align}
where $c_0$ is the constant given in 
\eqref{L4}, then the nonlinear problem \eqref{int-NS-DT-0} with $\bs\varphi ={\bf 0}$ has a unique weak solution ${\bf u}\in \mathring{H}_{\rm{div}}^1(\Omega )^n$.
\end{corollary}

\subsection{\bf Dirichlet-transmission problem for the compressible anisotropic Navier-Stokes system with trace and conormal derivative jumps on a transversal Lipschitz interface in a bounded Lipschitz domain}

Let us consider the Dirichlet-transmission problem
\begin{equation}
\label{int-NS-DT-1}
\left\{
\begin{array}{ll}
\boldsymbol{\mathcal L}({\bf u}^+,\pi ^+)=\widetilde{\bf f}^+|_{\Omega ^{+}}+{\bl ({\bf u}^+\cdot \nabla ){\bf u}^+}\,, \ \ {\rm{div}} \, {\bf u}^+={\Bl g^+} & \mbox{ in } \Omega ^{+},\\
{\boldsymbol{\mathcal L}({\bf u}^-,\pi ^-)}=\widetilde{\bf f}^-|_{\Omega ^{-}}
+{\bl ({\bf u}^-\cdot \nabla ){\bf u}^-}\,, \ \ {\rm{div}} \, {\bf u}^-={\Bl g^-}  & \mbox{ in } \Omega ^-,\\
(\gamma _{_{\Omega ^+}}{\bf u}^+)|_{_{\Sigma }}-{(\gamma _{_{\Omega ^-}}{\bf u}^{-})|_{_{\Sigma }}}
=\bs\varphi_{_{\Sigma }}&  \mbox{ on } \Sigma \,,\\
\big({\bf t}_{{\Omega ^+}}\big({\bf u}^+,\pi ^+;\widetilde{\bf f}^+ +\mathring{E}_{\Omega ^+\to \Omega ^-}({\bf u}^+\cdot \nabla ){\bf u}^+\big)\big)\big|_{\Sigma }\\
\hspace{3em}{+\, \big({\bf t}_{{\Omega ^-}}\big({\bf u}^-,\pi ^-;\widetilde{\bf f}^- +\mathring{E}_{\Omega ^-\to \Omega ^+}({\bf u}^-\cdot \nabla ){\bf u}^-\big)\big)\big|_{\Sigma }}={{\boldsymbol \psi }_{_\Sigma }} &  \mbox{ on } \Sigma \,,\\
({\gamma _{_{\Omega ^+}}}{\bf u}^+)|_{\Gamma ^+}=\bs\varphi |_{\Gamma ^+}&  \mbox{ on } \Gamma ^+\,,\\
({\gamma _{_{\Omega ^-}}}{\bf u}^-)|_{\Gamma ^-}=\bs\varphi |_{\Gamma ^-}&  \mbox{ on } \Gamma ^-\,
\end{array}\right.
\end{equation}
with the unknowns $({\bf u}^+,\pi ^+,{\bf u}^-,\pi ^-)\in {\mathfrak X}_{\Omega ^+,\Omega ^-}$ and the given data\newline
$(\widetilde{\bf f}^+,\widetilde{\bf f}^-,{\Bl g^+},{\Bl g^-},\bs\varphi_{_{\Sigma }},{{\boldsymbol \psi }_{_\Sigma }},{\bs\varphi})
\in \mathfrak Y_\bullet$.
Recall that ${\mathfrak X}_{\Omega ^+,\Omega ^-}$ is the space defined in \eqref{Omega-pm}, and $\mathfrak Y_\bullet$ is the space defined in the beginning of Section~\ref{S3.3.2}.

Note that by Lemma \ref{L.BT} there exist some functions ${\bf v}_{\bs\varphi}^\pm\in {H}^1(\Omega ^\pm )^n$ such that
\begin{equation}
\label{int-S-DT-1}
\left\{
\begin{array}{ll}
{\rm{div}} \, {\bf v}_{\bs\varphi}^+=g^+ & \mbox{ in } \Omega^{+},\\
{\rm{div}} \, {\bf v}_{\bs\varphi}^-=g^- & \mbox{ in } \Omega^{-},\\
\big(\gamma _{_{\Omega^+}}{\bn\bf v}_{\bs\varphi}^+\big)|_{_{\Sigma }}-{\big(\gamma _{_{\Omega^-}}{\bf v}_{\bs\varphi}^-\big)|_{_{\Sigma }}}
=\bs\varphi_{_{\Sigma }}&  \mbox{ on } \Sigma \,,\\
\big({\gamma _{_{\Omega^+}}}{\bf v}_{\bs\varphi}^+\big)|_{\Gamma^+}=\bs\varphi |_{\Gamma^+}&  \mbox{ on } \Gamma^+\,,\\
\big({\gamma _{_{\Omega^-}}}{\bf v}_{\bs\varphi}^-\big)|_{\Gamma^-}=\bs\varphi |_{\Gamma^-}&  \mbox{ on } \Gamma^-\,,
\end{array}\right.
\end{equation}
and some constant $C_\Sigma=C_\Sigma(\Omega ^+,\Omega ^-,n)>0$ such that
\begin{align}
\label{estimate-a1}
\|{\bf v}_{\bs\varphi}^\pm\|_{H^1(\Omega ^\pm)^n}\
\leq C_\Sigma
\gr\|(g^+,g^-,{\boldsymbol \varphi}_{_{\Sigma }},\boldsymbol \varphi)\|_{\mathcal M_\bullet}\,,
\end{align}
where
$\|(g^+,g^-,{\boldsymbol \varphi}_{_{\Sigma }},{\boldsymbol \varphi})\|_{\mathcal M_\bullet}$ is defined by \eqref{M-norm}.

On the other hand,
by \eqref{SET}
there exists constants $c_1^\pm>0$ depending only on $\Omega ^\pm $ and $n$, such that
\begin{align}
\label{L4-a}
\|{\bf v}\|_{L^4(\Omega^\pm )^n}
\leq c_1^\pm\|{\bf v}\|_{H^1(\Omega ^\pm )^n}
\leq  c_1^*\|{\bf v}\|_{H^1(\Omega ^\pm )^n}\,,
\quad \forall \ {\bf v}\in H^1(\Omega ^\pm )^n\,,
\end{align}
where $c_1^*=\max(c_1^+,c_1^-)$.
In addition,
inequality \eqref{L4} holds on $\Omega$.

Let us prove the existence of a solution to problem \eqref{int-NS-DT-1} by employing arguments similar to those in the proof of Theorem \ref{Dir-NS-var}.
\begin{theorem}
\label{Dir-NS-varT}
Let Assumption $\ref{interface-Sigma}$ and conditions \eqref{Stokes-1}-\eqref{mu} hold with $n\in \{2,3\}$.
Let $(\widetilde{\bf f}^+,\widetilde{\bf f}^-,{\Bl g^+},{\Bl g^-},\bs\varphi_{_{\Sigma }},{{\boldsymbol \psi }_{_\Sigma }},{\bs\varphi}) \in \mathfrak Y_\bullet$.
\begin{itemize}
\item[$(i)$]
If
\begin{align}
\label{assumption-data}
&\|(g^+,g^-,{\boldsymbol \varphi}_{_{\Sigma }},{\boldsymbol \varphi})\|_{\mathcal M_\bullet}
<\frac{1}{4}C_{\mathbb A}^{-1}c_0^{-1}(c_1^*+c_0)^{-1}C^{-1}_{\Sigma}\,,
\end{align}
where $C_{\mathbb A}$,  $C_{\Sigma}$, $c_1^*$ and $c_0$, are the constants in \eqref{mu}, \eqref{estimate-a1}, \eqref{L4-a} and \eqref{L4}, respectively,
then the Dirichlet-transmission problem \eqref{int-NS-DT-1} for the Navier-Stokes system has a solution $({\bf u}^+,\pi^+,{\bf u}^-,\pi^-)\in{\mathfrak X}_{\Omega^+,\Omega^-}$.
\item[$(ii)$]
If
\begin{multline}
\label{assumption-data-1}
c^2_0|\!|\!|\bs{\mathfrak F}|\!|\!|_{H^{-1}(\Omega )^n}
+\left(c^2_0C_{\Sigma}\|{\mathbb A}\|
+2C_{\mathbb A}^{-1}C_{\Sigma}c_0(c^*_1+c_0)\right)
\|(g^+,g^-,{\boldsymbol \varphi}_{_{\Sigma }},{\boldsymbol \varphi})\|_{\mathcal M_\bullet}\\
<\frac{1}{4}C_{\mathbb A}^{-2}\,,
\end{multline}
where $\bs{\mathfrak F}:=(\widetilde{\bf f}^+ + \widetilde{\bf f}^-) -\gamma _{_{\Sigma }}^*{{\boldsymbol \psi }_{_\Sigma }}$, then the nonlinear Dirichlet-transmission problem \eqref{int-NS-DT-1} has a unique solution $({\bf u}^+,\pi^+,{\bf u}^-,\pi^-)\in{\mathfrak X}_{\Omega^+,\Omega^-}$.
\end{itemize}
\end{theorem}

\begin{proof}
(i) Let us represent the unknowns ${\bf u}^\pm$ of problem \eqref{int-NS-DT-1} in the form
\begin{align}
\label{form-solutionT}
{\bf u}^\pm=\mathbf u_0^\pm+{\bf v}_{\bs\varphi}^\pm\,,
\end{align}
where $({\bf v}_{\bs\varphi}^+,{\bf v}_{\bs\varphi}^-)\in H^1(\Omega ^+)^n\times H^1(\Omega ^-)^n$ satisfy relations \eqref{int-S-DT-1} and estimate \eqref{estimate-a1}.
Then the nonlinear Dirichlet-transmission problem \eqref{int-NS-DT-1} reduces to the problem
\begin{equation}
\label{transmission-linear-NS-a-equiv-2T}
\!\!\!\!\!\left\{
\begin{array}{llll}
\boldsymbol{\mathcal L}(\mathbf u_0^+,\pi^+)=(\widetilde{\mathbf F}^+_{\bs\varphi}\mathbf u_0^+)|_{\Omega^+},\
{\rm{div}}\, \mathbf u_0^+=0
& \mbox{ in } \Omega^+,\\
\boldsymbol{\mathcal L}(\mathbf u_0^-,\pi_0^-)=(\widetilde{\mathbf F}^-_{\bs\varphi}\mathbf u_0^-)|_{\Omega^-},\
{\rm{div}}\, \mathbf u_0^-=0 & \mbox{ in } \Omega^-,\\
(\gamma _{_{\Omega^+}}\mathbf u_0^+)|_{_{\Sigma }} =(\gamma _{_{\Omega^-}}\mathbf u_0^{-})|_{_{\Sigma }} & \mbox{ on } \Sigma ,\\
{\rd \big({\bf t}_{\Omega^+}(\mathbf u_0^+,\pi^+;\widetilde{\mathbf F}^+_{\bs\varphi}\mathbf u_0^+)\big)|_{_{\Sigma }}\,
+\, \big({\bf t}_{\Omega^-}(\mathbf u_0^-,\pi^-;\widetilde{\mathbf F}^-_{\bs\varphi}\mathbf u_0^-)\big)|_{_{\Sigma }}={{\boldsymbol \psi }_{_\Sigma }}} &  \mbox{ on } \Sigma ,\\
\!({\gamma _{_{\Omega^+}}}\mathbf u_0^+)|_{\Gamma^+}={\bf 0}&  \mbox{ on } \Gamma^+,\\
\!({\gamma _{_{\Omega^-}}}\mathbf u_0^-)|_{\Gamma^-}={\bf 0}&  \mbox{ on } \Gamma^-\,,
\end{array}\right.
\end{equation}
for the unknowns {\rd $({\bf u}_0^+,\pi^+,{\bf u}_0^-,\pi^-)\in {\mathfrak X}_{\Omega^+,\Omega^-}$}, where
\begin{align}
\label{FuphiTa}
\widetilde{\mathbf F}^\pm_{\bs\varphi}\mathbf w^\pm:&=\widetilde{\bf f}^\pm {\Bl -{\Bl \check{\bs{\mathfrak L}}^\pm{\bf v}_{\bs\varphi}^\pm}}
+\mathring E_{\Omega^\pm}\big[((\mathbf w^\pm+{\bf v}_{\bs\varphi}^\pm)\cdot \nabla )(\mathbf w^\pm+{\bf v}_{\bs\varphi}^\pm)\big]\,.
\end{align}
{Recall that $\check{\bs{\mathfrak L}}^\pm: {H}^1(\Omega^\pm)^n\to \widetilde{H}^{-1}(\Omega^\pm)^n$ are the operators defined in \eqref{checkL}}
and $\big(\check{\bs{\mathfrak L}}^\pm{\bf v}^\pm\big)|_{\Omega^\pm}=\boldsymbol{\mathcal L}({\bf v}^\pm,0)$ in $\Omega^\pm$.
For fixed ${\bf v}_{\bs\varphi}^\pm\!\in \!{H}^1(\Omega )^n$, formula \eqref{FuphiTa} defines nonlinear operators
${\bf w}^\pm\!\mapsto \!\widetilde{\mathbf F}^\pm_{\bs\varphi}\mathbf w^\pm$
from the space ${H}^{1}(\Omega^\pm)^n$ {\Bl to the space $\left({H}_{\Gamma^\pm}^1(\Omega^\pm)^n\right)'$}
due to estimate \eqref{P-01til} and the inclusion $\check{\bs{\mathfrak L}}^\pm{\bf v}_{\bs\varphi}^\pm\in \widetilde{H}^{-1}(\Omega ^\pm )^n\hookrightarrow \left({H}_{\Gamma^\pm}^1(\Omega^\pm)^n\right)'$. 

In addition, the inclusion $({\bf u}_0^+,{\rd \pi^+},{\bf u}_0^-,{\rd \pi^-})\in {\mathfrak X}_{\Omega^+,\Omega^-}$, the homogeneous interface condition for traces ${(\gamma _{_{\Omega^+}}{\bf u}_0^+)|_{_{\Sigma }}}-{(\gamma _{_{\Omega^-}}{\bf u}_0^{-})|_{_{\Sigma }}}={\bf 0}$ in \eqref{transmission-linear-NS-a-equiv-2T}, and
Lemma \ref{extention} imply that there exists a unique pair $({\bf u}_0,\pi)\in H^1(\Omega )^n\times L^2(\Omega)/{\mathbb R}$ such that
\begin{align}
\label{uNS}
{\bf u}_0|_{\Omega^+}={\bf u}_0^+,\quad
{\bf u}_0|_{\Omega^-}={\bf u}_0^-,\quad
{\rd \pi }|_{\Omega^+} ={\rd \pi ^+},\quad
{\rd \pi }|_{\Omega^-} ={\rd \pi^-}\,.
\end{align}
Since ${\bf u}_0^+$ and ${\bf u}_0^-$ also satisfy the homogeneous Dirichlet condition in \eqref{transmission-linear-NS-a-equiv-2T} and are divergence-free,
we have that ${\bf u}_0\in \mathring{H}_{{\rm{div}}}^1(\Omega )^n$.
Hence, $({\bf u}_0,{\rd \pi })\in\mathring{H}_{{\rm{div}}}^1(\Omega )^n\times L^2(\Omega )/\mathbb R$.

For any ${\bf w}\in \mathring{H}_{{\rm{div}}}^1(\Omega )^n$ let ${\mathbf F}_{\bs\varphi}\mathbf w$
be defined by
\begin{align}
\label{FuphiT}
{\mathbf F}_{\bs\varphi}\mathbf w
&=\widetilde{\mathbf F}^+_{\bs\varphi} r_{_{\Omega^+}}\mathbf w
+ \widetilde{\mathbf F}^-_{\bs\varphi} r_{_{\Omega^-}}\mathbf w {\rd -\gamma _{_{\Sigma }}^*{\boldsymbol \psi }_{_\Sigma }}\nonumber\\
&
={\bs{\mathfrak F}} {-\gr \big(\check{\bs{\mathfrak L}}^+{\bf v}_{\bs\varphi}^+
+\check{\bs{\mathfrak L}}^-{\bf v}_{\bs\varphi}^-\big)}
+\mathring E_{\Omega^+}\big[({\bf v}_{\bs\varphi}^+\cdot \nabla ){\bf v}_{\bs\varphi}^+\big]
+\mathring E_{\Omega^-}\big[({\bf v}_{\bs\varphi}^-\cdot \nabla ){\bf v}_{\bs\varphi}^-\big]+(\mathbf w\cdot \nabla )\mathbf w\nonumber\\
&\quad
+((\mathring E_{\Omega^+}{\bf v}_{\bs\varphi}^++\mathring E_{\Omega^-}{\bf v}_{\bs\varphi}^-)\cdot \nabla )\mathbf w
+\mathbf w\cdot (\mathring E_{\Omega^+}\nabla {\bf v}_{\bs\varphi}^+ + \mathring E_{\Omega^-}\nabla {\bf v}_{\bs\varphi}^-)\,,
\end{align}
where
$\bs{\mathfrak F}{\gr:=}(\widetilde{\bf f}^+ + \widetilde{\bf f}^-) -\gamma _{_{\Sigma }}^*{\boldsymbol \psi }_{_\Sigma}$, ${\gamma _{_{\Sigma }}^*}:H^{-\frac{1}{2}}(\Sigma )^n\to {H}^{-1}(\Omega )^n$ is the adjoint of the trace operator ${\gamma _{_{\Sigma }}}:\mathring{H}^{1}(\Omega )^n\to \widetilde{H}^{\frac{1}{2}}(\Sigma )^n$ defined by \eqref{gamma-Sigma},
and hence $\bs{\mathfrak F}\in H^{-1}(\Omega )^n$
due to Lemma \ref{identification}.
For fixed ${\bf v}_{\bs\varphi}^\pm\in {H}^1(\Omega^\pm)^n$, formula \eqref{FuphiT} defines a nonlinear operator ${\bf w}\mapsto \mathbf F_{\bs\varphi}\mathbf w$
from $\mathring{H}_{{\rm{div}}}^{1}(\Omega )^n$ to ${H}^{-1}(\Omega )^n$
due to Lemma \ref{identification} and estimates \eqref{P-01til} and \eqref{P-01}.

Now, arguing as in the proof of Lemma \ref{equiv-nonlin}
(cf. also Theorem \ref{T-p}),
we obtain that the nonlinear Dirichlet-transmission problem \eqref{transmission-linear-NS-a-equiv-2T}
with the unknowns $\left({\bf u}_0^+,\pi^+,{\bf u}_0^-,\pi^-\right)\in {\mathfrak X}_{\Omega^+,\Omega^-}$ is equivalent,
in the sense of relations \eqref{uNS},
to the nonlinear equation
\begin{align}
\label{transmission-linear-NS-a-equiv-1T}
\bs{\mathcal L}({\bf u}_0,{\rd \pi })=
\mathbf F_{\bs\varphi}\mathbf u_0\quad \mbox{in } \Omega
\end{align}
for the {\rd unknowns} $({\bf u}_0,{\rd \pi })\in \mathring{H}_{{\rm{div}}}^1(\Omega )^n\times L^2(\Omega )/{\mathbb R}$,
with $\mathbf F_{\bs\varphi}$ given by \eqref{FuphiT}.

The following arguments are similar to {\rd those in} the proof of Theorem \ref{Dir-NS-var}.
By Theorem \ref{lemma-a47-1-Stokes}, the linear operator
\begin{align}\label{LisomT}
\bs{\mathcal L}:\mathring{H}_{{\rm{div}}}^1(\Omega )^n\times L^2(\Omega )/{\mathbb R}\to H^{-1}(\Omega)^n
\end{align}
is an isomorphism.
Its inverse operator can be split into two operator components,
$$\bs{\mathcal L}^{-1}=(\bs{\mathcal U},\mathcal P),$$
where
$\bs{\mathcal U}: H^{-1}(\Omega)^n\to \mathring{H}_{{\rm{div}}}^1(\Omega )^n$
and $\mathcal P:H^{-1}(\Omega)^n\to L^2(\Omega )/{\mathbb R}$ are linear continuous operators such that
$\bs{\mathcal L}(\bs{\mathcal U}\mathbf F,\mathcal P\mathbf F)=\mathbf F$ for any $\mathbf F\in H^{-1}(\Omega)^n$.
Applying the operator $\bs{\mathcal L}^{-1}$ to equation \eqref{transmission-linear-NS-a-equiv-1T} we obtain the equivalent nonlinear system
\begin{align}
\label{transmission-linear-NS-a-equiv-2Tu}
{\mathbf u_0}
&=\mathbf U{\mathbf u_0},\\
\label{transmission-linear-NS-a-equiv-2PTpi}
{\rd \pi }&=P{\mathbf u_0},
\end{align}
where $\mathbf U:\mathring{H}_{{\rm{div}}}^{1}(\Omega )^n\to \mathring{H}_{{\rm{div}}}^{1}(\Omega )^n$
and
$P:\mathring{H}_{{\rm{div}}}^{1}(\Omega )^n\to L^2(\Omega )/{\mathbb R}$
are the nonlinear operators defined as
\begin{align}
\label{NT}
&\mathbf U{\mathbf w}:=
\bs{\mathcal U}\,\mathbf F_{\bs\varphi}\mathbf w,\\
\label{PT}
&P{\mathbf w}:=
\mathcal P\,\mathbf F_{\bs\varphi}\mathbf w.
\end{align}

Since ${\rd \pi }$ is not involved in \eqref{transmission-linear-NS-a-equiv-2Tu}, we will first prove the existence of a solution ${\mathbf u_0}\in \mathring{H}_{\rm{div}}^1(\Omega )^n$ to this equation and then use \eqref{transmission-linear-NS-a-equiv-2PTpi} as a representation formula. This provides the existence of {\rd a pressure field} ${\rd \pi }\in L^2(\Omega )/{\mathbb R}$.

In order to show the existence of a fixed point of the operator $\mathbf U$ and, thus, the existence of  a weak solution of nonlinear problem \eqref{transmission-linear-NS-a-equiv-2T}, we employ Theorem \ref{L-S-fixed}.

Let us show first that $\mathbf U$ is continuous.
Let ${\bf w},{\bf w}'\in \mathring{H}_{\rm{div}}^1(\Omega )^n$.
Then by \eqref{FuphiT},  \eqref{P-01},  \eqref{P-01'3} and \eqref{P-01'3-1} we obtain that
\begin{align*}
&\big\|{\mathbf F}_{\bs\varphi}{\bf w}-{\mathbf F}_{\bs\varphi}{\bf w}'\big\|_{{H}^{-1}(\Omega)^n}
\le\left\|({\bf w}\cdot \nabla ){\bf w}-({\bf w}'\cdot \nabla ){\bf w}'\right\|_{{H}^{-1}(\Omega)^n}
\nonumber\\
&+\big\|((\mathring E_{\Omega^+}{\bf v}_{\bs\varphi}^++\mathring E_{\Omega^-}{\bf v}_{\bs\varphi}^-)\cdot \nabla )({\bf w}-{\bf w}')
+({\bf w}-{\bf w}')\cdot (\mathring E_{\Omega^+}\nabla {\bf v}_{\bs\varphi}^+
+ \mathring E_{\Omega^-}\nabla {\bf v}_{\bs\varphi}^-) \big\|_{{H}^{-1}(\Omega)^n}
\nonumber\\
&\le\left\|(({\bf w}-{\bf w}')\cdot \nabla ){\bf w} + ({\bf w}'\cdot \nabla )({\bf w}-{\bf w}')\right\|_{{H}^{-1}(\Omega)^n}
\nonumber\\
&\hspace{6em}
+2(c_1^*)^2
\left\|{\bf w}-{\bf w}'\right\|_{{H}^{1}(\Omega)^n}(\|{\bf v}_{\bs\varphi}^+\|_{{H}^{1}(\Omega^+)^n}
+\|{\bf v}_{\bs\varphi}^-\|_{{H}^{1}(\Omega^-)^n})
\nonumber\\
&\le
\left\|{\bf w}-{\bf w}'\right\|_{{H}^{1}(\Omega)^n}\big(c_1^2\|{\bf w}\|_{{H}^{1}(\Omega)^n}
+c_1^2\|{\bf w}'\|_{{H}^{1}(\Omega)^n}
\nonumber\\
&\hspace{6em}+2(c_1^*)^2
\|{\bf v}_{\bs\varphi}^+\|_{{H}^{1}(\Omega^+)^n}
+2(c_1^*)^2
|{\bf v}_{\bs\varphi}^-\|_{{H}^{1}(\Omega^-)^n}\big).
\end{align*}
This estimate shows that the operator
${\mathbf F}_{\bs\varphi}:\mathring{H}_{{\rm{div}}}^{1}(\Omega )^n\to {H}^{-1}(\Omega )^n$
is continuous.
Consequently, the operator $\mathbf U=\bs{\mathcal U}\,{\mathbf F}_{\bs\varphi}:\mathring{H}_{{\rm{div}}}^1(\Omega )^n\to \mathring H_{{\rm{div}}}^1(\Omega )^n$ is also continuous, as asserted.

Next we show that the operator $\mathbf U$ is compact. To this end, we assume that $\{{\bf w}_k\}_{k\in {\mathbb N}}$ is a bounded sequence in the space $\mathring{H}_{{\rm{div}}}^1(\Omega )^n$
and prove that the sequence $\{{\mathbf F}_{\bs\varphi}{\bf w}_k\}_{k\in {\mathbb N}}$ contains a convergent subsequence in ${H}^{-1}(\Omega )^n$.

Let $M>0$ be such that $\|{\bf w}_k\|_{{H}^1(\Omega )^n}\leq M$ for all $k\in {\mathbb N}$.
By \eqref{FuphiT},
\begin{align*}
\big\|{\mathbf F}_{\bs\varphi}{\bf w}_k-{\mathbf F}_{\bs\varphi}{\bf w}_\ell\big\|_{{H}^{-1}(\Omega)^n}
&\le\left\|(({\bf w}_k-{\bf w}_\ell)\cdot \nabla ){\bf w}_k
+ ({\bf w}_\ell\cdot \nabla )({\bf w}_k-{\bf w}_\ell)\right\|_{{H}^{-1}(\Omega)^n}
\quad\nonumber\\
&+\big\|((\mathring E_{\Omega^+}{\bf v}_{\bs\varphi}^+ +\mathring E_{\Omega^-}{\bf v}_{\bs\varphi}^-)\cdot \nabla )
({\bf w}_k-{\bf w}_\ell)\big\|_{{H}^{-1}(\Omega)^n}\\
&+\big\|({\bf w}_k-{\bf w}_\ell)\cdot (\mathring E_{\Omega^+}\nabla {\bf v}_{\bs\varphi}^+
+ \mathring E_{\Omega^-}\nabla {\bf v}_{\bs\varphi}^-) \big\|_{{H}^{-1}(\Omega)^n}.
\end{align*}
Employing estimates \eqref{P-01} and \eqref{P-0512b} to the first norm in the right hand side, \eqref{P-01-gamma'} for the second norm, and estimate \eqref{P-01'3} for the third norm, we obtain
\begin{align}
\label{compactnessT}
&\big\|{\mathbf F}_{\bs\varphi}{\bf w}_k-{\mathbf F}_{\bs\varphi}{\bf w}_\ell\big\|_{{H}^{-1}(\Omega)^n}
\leq \Big( c_1\|{\bf w}_k\|_{{H}^{1}(\Omega)^n} +c_1 \|{\bf w}_\ell\|_{{H}^{1}(\Omega)^n}
\nonumber\\
&+ 3c_1^*\|{\bf v}_{\bs\varphi}^+\|_{{H}^{1}(\Omega^+)^n}
+ 3c_1^* \|{\bf v}_{\bs\varphi}^-\|_{{H}^{1}(\Omega^-)^n}\Big)
\left\|{\bf w}_k-{\bf w}_\ell\right\|_{L^4(\Omega)^n}
\nonumber\\
&+\Big((c^+_2)^2\|\gamma_{\Omega^+}\|^2\|{\bf v}_{\bs\varphi}^+\|_{H^{1}(\Omega^+)^n}
+(c^-_2)^2\|\gamma_{\Omega^-}\|^2\|{\bf v}_{\bs\varphi}^-\|_{H^{1}(\Omega^-)^n}\Big)
\|\gamma_{_\Sigma}({\bf w}_k-{\bf w}_\ell)\|_{L^{3}(\Sigma)^{n}}\nonumber\\
&\leq \Big(2c_1M+ 3c_1^*\|{\bf v}_{\bs\varphi}^+\|_{{H}^{1}(\Omega^+)^n}
+ 3c_1^*\|{\bf v}_{\bs\varphi}^-\|_{{H}^{1}(\Omega^-)^n}\Big)\left\|{\bf w}_k-{\bf w}_\ell\right\|_{L^4(\Omega)^n}
\nonumber\\
&+\Big((c^+_2)^2\|\gamma_{\Omega^+}\|^2\|{\bf v}_{\bs\varphi}^+\|_{H^{1}(\Omega^+)^n}
+(c^-_2)^2\|\gamma_{\Omega^-}\|^2\|{\bf v}_{\bs\varphi}^-\|_{H^{1}(\Omega^-)^n}\Big)
\|\gamma_{_\Sigma}({\bf w}_k-{\bf w}_\ell)\|_{L^{3}(\Sigma)^{n}},
\end{align}
where we denoted
$\gamma_{_\Sigma}({\bf w}_k-{\bf w}_\ell):=r_{_\Sigma}\gamma_{\Omega^+}({\bf w}_k-{\bf w}_\ell)
=r_{_\Sigma}\gamma_{\Omega^-}({\bf w}_k-{\bf w}_\ell)$
and took into account that
$r_{_{\Gamma^+}}\gamma_{\Omega^+}({\bf w}_k-{\bf w}_\ell)=\bf 0$
and
$r_{_{\Gamma^-}}\gamma_{\Omega^-}({\bf w}_k-{\bf w}_\ell)=\bf 0$.

Since {\gr $\|{\bf w}_k\|_{{H}^1(\Omega )^n}\leq M$ and} the embedding of the space ${H}^1(\Omega )^n$ into the space $L^4(\Omega )^n$ is compact (see, e.g., \cite[Theorem 6.3]{Adams2003}), there exists a subsequence of $\{{\bf w}_k\}_{k\in {\mathbb N}}$, labelled as the sequence,
which converges in $L^4(\Omega )^n$, and, hence, is a Cauchy sequence in  $L^4(\Omega )^n$.

Since $\|{\bf w}_k\|_{{H}^1(\Omega )^n}\leq M$, the sequence $\{\gamma_{_\Sigma}{\bf w}_k\}_{k\in {\mathbb N}}$ is bounded in ${H}^{1/2}(\Sigma)^n$.
Further, for $n=2,3$ the space ${H}^{1/2}(\Sigma)^n$ is compactly embedded in $L^3({\gr\Sigma})^n$.
For bounded Lipschitz domains in $\R^{n-1}$, this follows by the Rellich-Kondrachev compactness theorems, e.g., from the embeddings in \cite[Section 2.2.4, Corollary 2(i)]{Runst-Sickel} for $\R^{n-1}$ and can be extended to {\rd $(n-1)$-dimensional bounded Lipschitz manifolds} by standard arguments (cf. also a more general statement in \cite[Proposition 3.8]{MM2013Spr}).
Then again, there exists a subsequence of $\{{\bf w}_k\}_{k\in {\mathbb N}}$, labelled as the sequence, such that $\{\gamma_{_\Sigma}{\bf w}_k\}_{k\in {\mathbb N}}$ converges in $L^3(\Sigma )^n$, and, hence, is a Cauchy sequence in  $L^3(\Sigma)^n$.

Inequality \eqref{compactnessT} combined with the these Cauchy properties implies that
${\mathbf F}_{\bs\varphi}:\mathring{H}_{\rm{div}}^1(\Omega )^n\to {H}^{-1}(\Omega )^n$
is a compact operator.
Hence, the operator
$\mathbf U=\bs{\mathcal U}\,{\mathbf F}_{\bs\varphi}:
\mathring{H}_{\rm{div}}^1(\Omega )^n\to \mathring H_{\rm{div}}^1(\Omega )^n$  is also compact, as asserted.

It remains to show that there exists a constant $M_0>0$ such that if ${\bf w}\in \mathring{H}_{\rm{div}}^1(\Omega )^n$ satisfies the equation
\begin{align}
\label{AsetT}
&{\bf w}=\lambda\mathbf U{\bf w}
\end{align}
for some $\lambda \in [0,1]$, then $\|{\bf w}\|_{{H}^1(\Omega )^n}\le M_0$.
Let us introduce the function
\begin{align}\label{Aset-qT}
q:=\lambda P{\mathbf w}.
\end{align}
By applying the operator $\boldsymbol{\mathcal L}$ to equations \eqref{AsetT}-\eqref{Aset-qT}, and by using relations \eqref{NT} and \eqref{PT}, we deduce that whenever the pair
$({\mathbf w},\lambda)\in \mathring{H}_{\rm{div}}^1(\Omega )^n\times\R$ satisfies {\gr equation \eqref{AsetT}}, the equation
$
\boldsymbol{\mathcal L}({\bf w},q)=\lambda {\mathbf F}_{\bs\varphi}{\bf w}
$
is also satisfied. (Recall the isomorphism property of operator \eqref{LisomT}.)
Then the first Green identity \eqref{Green-H0} implies
\begin{align}
a_{{\mathbb A};\Omega}({\bf w},{\bf v}):=\left\langle a_{ij}^{\alpha \beta }E_{j\beta }({\bf w}),E_{i\alpha }({\bf v})\right\rangle _{\Omega}
=-\langle\lambda \mathbf F_{\bs\varphi}{\bf w},{\bf v}\rangle_{\Omega}
\quad \forall\ \mathbf v\in \mathring{H}_{\rm{div}}^1(\Omega )^n,
\end{align}
which, in view of relation \eqref{FuphiT}, takes the form
\begin{align}
\label{transmission-linear-NS-a-equiv-5new-0}
a_{{\mathbb A};\Omega}({\bf w},{\bf v})=
&\!-\!\lambda \langle \bs{\mathfrak F},{\bf v}\rangle _{\Omega }\!
-\!\lambda {\Bl \Big(\big\langle a_{ij}^{\alpha \beta}E_{j\beta }({\gr\bf v}_{\bs\varphi}^+),E_{i\alpha }({\bf v})\big\rangle_{\Omega ^+}\!
+\!\big\langle a_{ij}^{\alpha \beta }E_{j\beta }({\gr\bf v}_{\bs\varphi}^-),E_{i\alpha }({\bf v})\big\rangle _{\Omega ^-}\Big)}\nonumber\\
&-\lambda \Big\langle \mathring E_{\Omega^+}\big[({\bf v}_{\bs\varphi}^+\cdot \nabla ){\bf v}_{\bs\varphi}^+\big]
+\mathring E_{\Omega^-}\big[({\bf v}_{\bs\varphi}^-\cdot \nabla ){\bf v}_{\bs\varphi}^-\big]
\nonumber\\
&
+((\mathring E_{\Omega^+}{\bf v}_{\bs\varphi}^++\mathring E_{\Omega^-}{\bf v}_{\bs\varphi}^-)\cdot \nabla )\mathbf w
\nonumber\\
&
\hspace{-2em}+\mathbf w\cdot (\mathring E_{\Omega^+}\nabla {\bf v}_{\bs\varphi}^+ + \mathring E_{\Omega^-}\nabla {\bf v}_{\bs\varphi}^-)
+(\mathbf w\cdot \nabla )\mathbf w,
{\bf v}\Big\rangle _{\Omega}\,, \ \forall\, \mathbf v\in \mathring{H}_{\rm{div}}^1(\Omega )^n\,.
\end{align}
Moreover, formula \eqref{P-5a} and the {\gr inclusion} ${\bf w}\in \mathring{H}_{\rm{div}}^1(\Omega )^n$ imply the relation
$
\Big\langle (\mathbf w\cdot \nabla )\mathbf w,{\bf w}\Big\rangle _{\Omega}=0\,.
$
Then by \eqref{transmission-linear-NS-a-equiv-5new-0} 
we obtain the formula
\begin{align}
\label{transmission-linear-NS-a-equiv-5new}
\hspace{-1.018em}a_{{\mathbb A};\Omega}({\bf w},{\bf w})=
&-\lambda \langle {{\bs{\mathfrak F}}},{\gr\bf w}\rangle _{\Omega }
-\lambda {\Bl \Big(\big\langle a_{ij}^{\alpha \beta }E_{j\beta }({\bf v}_{\bs\varphi}^+),E_{i\alpha }({\bf w})\big\rangle _{\Omega ^+}\!+\!\big\langle a_{ij}^{\alpha \beta }E_{j\beta }({\gr\bf v}_{\bs\varphi}^-),E_{i\alpha }({\bf w})\big\rangle _{\Omega ^-}\Big)}\nonumber\\
&{\rd -\lambda \Big\langle \mathring E_{\Omega^+}\big[({\bf v}_{\bs\varphi}^+\cdot \nabla ){\bf v}_{\bs\varphi}^+\big]+\mathring E_{\Omega^-}\big[({\bf v}_{\bs\varphi}^-\cdot \nabla ){\bf v}_{\bs\varphi}^-\big],{\bf w}\Big\rangle _{\Omega }}\nonumber\\
&{\rd -\lambda \Big\langle \left(\left(\mathring E_{\Omega^+}{\bf v}_{\bs\varphi}^++\mathring E_{\Omega^-}{\bf v}_{\bs\varphi}^-\right)\cdot \nabla \right)\mathbf w ,{\bf w}\Big\rangle _{\Omega}}\nonumber\\
&-\lambda \Big\langle \mathbf w\cdot \left(\mathring E_{\Omega^+}\nabla {\bf v}_{\bs\varphi}^+
+ \mathring E_{\Omega^-}\nabla {\bf v}_{\bs\varphi}^-\right),
{\bf w}\Big\rangle _{\Omega}\,,
\end{align}
Arguments similar to those for estimate \eqref{weak-D-1} combined with formula \eqref{transmission-linear-NS-a-equiv-5new}, the {\gr inclusion} $\lambda \in [0,1]$ and inequalities \eqref{L4},  \eqref{L4-a}, and \eqref{estimate-a1} imply that
\begin{align}
\frac{1}{2}&C_{\mathbb A}^{-1}\|\nabla {\bf w}\|_{L^2(\Omega )^{n\times n}}^2\leq
C_{\mathbb A}^{-1}\|{\mathbb E}({\bf w})\|_{L^2(\Omega )^{n\times n}}^2
\leq a_{{\mathbb A};\Omega}({\bf w},{\bf w})
\nonumber\\
&\leq
|\!|\!|\bs{\mathfrak F}|\!|\!|_{H^{-1}(\Omega )^n}\|\nabla {\bf w}\|_{L^2(\Omega )^{n\times n}}
\nonumber
\\
&\hspace{1em}+\|{\mathbb A}\|\Big(\|\nabla {\bf v}_{\bs\varphi}^+\|_{L^2(\Omega ^+)^{n\times n}}
+\|\nabla {\bf v}_{\bs\varphi}^-\|_{L^2(\Omega ^-)^{n\times n}}\Big)
\|\nabla {\bf w}\|_{L^2(\Omega )^{n\times n}}\nonumber\\
&\hspace{1em}+ \Big(\|{\bf v}_{\bs\varphi}^+\|_{L^4(\Omega ^+)^n}
\|\nabla {\bf v}_{\bs\varphi}^+\|_{L^2(\Omega ^+)^{n\times n}}\!+\!
\|{\bf v}_{\bs\varphi}^-\|_{L^4(\Omega ^-)^n}\|\nabla {\bf v}_{\bs\varphi}^-\|_{L^2(\Omega ^-)^{n\times n}}\Big)\|{\bf w}\|_{L^4(\Omega )^n}
\nonumber
\\
&\hspace{1em} +\Big(\|{\bf v}_{\bs\varphi}^+\|_{L^4(\Omega ^+)^n}+\|{\bf v}_{\bs\varphi}^-\|_{L^4(\Omega ^+)^n}\Big)\|\nabla {\bf w}\|_{L^2(\Omega )^{n\times n}}\|{\bf w}\|_{L^4(\Omega )^n}\nonumber\\
&\hspace{1em}+\Big(\|\nabla {\bf v}_{\bs\varphi}^+\|_{L^2(\Omega ^+)^{n\times n}}
+\|\nabla {\bf v}_{\bs\varphi}^-\|_{L^2(\Omega ^-)^{n\times n}}\Big)\|{\bf w}\|_{L^4(\Omega )^n}^2\nonumber\\
&
\leq
|\!|\!|\bs{\mathfrak F}|\!|\!|_{H^{-1}(\Omega )^n}\|\nabla {\bf w}\|_{L^2(\Omega )^{n\times n}}
\nonumber
\\
&\hspace{1em}+\|{\mathbb A}\|\Big(\|{\bf v}_{\bs\varphi}^+\|_{H^1(\Omega ^+)^n}
+\|{\bf v}_{\bs\varphi}^-\|_{H^1(\Omega ^-)^n}\Big)
\|\nabla {\bf w}\|_{L^2(\Omega )^{n\times n}}\nonumber\\
&\hspace{1em}+ c_0{\gr c_1^*}\Big(\|{\bf v}_{\bs\varphi}^+\|^2_{H^1(\Omega ^+)^n}+
\|{\bf v}_{\bs\varphi}^-\|^2_{H^1(\Omega ^-)^n}\Big)\|\nabla {\bf w}\|_{L^2(\Omega )^{n\times n}}\nonumber\\
&\hspace{1em} +c_0{\gr c_1^*}\Big(\|{\bf v}_{\bs\varphi}^+\|_{H^1(\Omega ^+)^n}
+\|{\bf v}_{\bs\varphi}^-\|_{H^1(\Omega ^+)^n}\Big)\|\nabla {\bf w}\|^2_{L^2(\Omega )^{n\times n}}\nonumber\\
&\hspace{1em}+c_0^2\Big(\|{\bf v}_{\bs\varphi}^+\|_{H^1(\Omega ^+)^n}
+\|{\bf v}_{\bs\varphi}^-\|_{H^1(\Omega ^+)^n}\Big)\|\nabla{\bf w}\|_{L^2(\Omega )^{n\times n}}^2\nonumber\\
&
\leq
|\!|\!|\bs{\mathfrak F}|\!|\!|_{H^{-1}(\Omega )^n}\|\nabla {\bf w}\|_{L^2(\Omega )^{n\times n}}
+\|{\mathbb A}\|\,C_{\gr\Sigma}\,\|(g^+,g^-,{\boldsymbol \varphi}_{_{\Sigma }},{\boldsymbol \varphi})\|_{\mathcal M_\bullet}
\|\nabla {\bf w}\|_{L^2(\Omega )^{n\times n}}
\nonumber\\
&\hspace{1em}
+2c_0 c_1^*C^2_{\Sigma}\|(g^+,g^-,{\boldsymbol \varphi}_{_{\Sigma }},{\boldsymbol \varphi})\|_{\mathcal M_\bullet}^2
\|\nabla {\bf w}\|_{L^2(\Omega )^{n\times n}}
\nonumber\\
&\hspace{1em}
+2{(c_0 c_1^*+c_0^2)}C_{\Sigma}\|(g^+,g^-,{\boldsymbol \varphi}_{_{\Sigma }},{\boldsymbol \varphi})\|_{\mathcal M_\bullet}
\|\nabla {\bf w}\|_{L^2(\Omega )^{n\times n}}^2
\label{E7.44}
\end{align}
Therefore, we obtain the estimate
\begin{align}
\label{weak-D-1T}
\hspace{-1em}&\left(\frac{1}{2}C_{\mathbb A}^{-1}
-2c_0(c_1^*+c_0)C_{\Sigma}\|(g^+,g^-,{\boldsymbol \varphi}_{_{\Sigma }},{\boldsymbol \varphi})\|_{\mathcal M_\bullet}\right)\|\nabla {\bf w}\|_{L^2(\Omega )^{n\times n}}
\\
\hspace{-1em}&
\le
|\!|\!|\bs{\mathfrak F}|\!|\!|_{H^{-1}(\Omega )^n}
+ C_{\Sigma}\|{\mathbb A}\|  \|(g^+,g^-,{\boldsymbol \varphi}_{_{\Sigma }},{\boldsymbol \varphi})\|_{\mathcal M_\bullet}
\!+\!2c_0c_1^*C^2_{\gr\Sigma}\|(g^+,g^-,{\boldsymbol \varphi}_{_{\Sigma }},{\boldsymbol \varphi})\|_{\mathcal M_\bullet}^2.\nonumber
\end{align}
In view of assumption \eqref{assumption-data}, estimate \eqref{weak-D-1T} can be written in the form
\begin{align}
\label{weak-D-2}
&\hspace{-0.5em}\|\nabla {\bf w}\|_{L^2(\Omega )^{n\times n}}\leq
\nonumber\\
&\frac{|\!|\!|\bs{\mathfrak F}|\!|\!|_{H^{-1}(\Omega )^n}
+C_{\Sigma}\|{\mathbb A}\|  \|(g^+,g^-,{\boldsymbol \varphi}_{_{\Sigma }},{\boldsymbol \varphi})\|_{\mathcal M_\bullet}
+2c_0c_1^*C^2_{\Sigma}\|(g^+,g^-,{\boldsymbol \varphi}_{_{\Sigma }},{\boldsymbol \varphi})\|_{\mathcal M_\bullet}^2}
{\frac{1}{2}C_{\mathbb A}^{-1}
-2c_0(c_1^*+c_0)C_{\Sigma}\|(g^+,g^-,{\boldsymbol \varphi}_{_{\Sigma }},{\boldsymbol \varphi})\|_{\mathcal M_\bullet}}\,,
\end{align}
that is, $\|{\bf w}\|_{{H}^1(\Omega )^n}\le M_0$, where $M_0$ is given by the right hand side of \eqref{weak-D-2}
multiplied by the equivalence constant $\mathring C$ from \eqref{norm-ineq}.

Therefore, the operator ${\rd \mathbf U}:\mathring{H}_{{\rm{div}}}^1(\Omega )^n\to \mathring H_{{\rm{div}}}^1(\Omega )^n$ given by \eqref{NT} satisfies the hypothesis of Theorem \ref{L-S-fixed} (for $\mathcal X\!=\!\mathring{H}_{\rm{div}}^1(\Omega )^n$), and hence {\rd it has a fixed point ${\mathbf u_0}\!\in \!\mathring{H}_{\rm{div}}^1(\Omega )^n$, that is,} ${\mathbf u_0}={\rd \mathbf U}{\mathbf u_0}$. Then with $\pi \!\in \!L^2(\Omega )/{\mathbb R}$ as in \eqref{transmission-linear-NS-a-equiv-2PTpi}, the couple $(\mathbf u_0, \pi )\!\in \!\mathring H_{{\rm{div}}}^1(\Omega )^n\times L^2(\Omega )/{\mathbb R}$ satisfies the nonlinear equation \eqref{transmission-linear-NS-a-equiv-1T}.
Consequently, the couples $(\mathbf u_0^+ + {\bf v}_{\bs\varphi}^+, \pi ^+ ,\mathbf u_0^- + {\bf v}_{\bs\varphi}^-, \pi ^- )\in{\mathfrak X}_{\Omega^+,\Omega^-}$
provide a solution of the nonlinear Dirichlet-transmission \eqref{int-NS-DT-1} in the sense of relations \eqref{form-solutionT}. (Revall the equivalence between the nonlinear Dirichlet-transmission problem \eqref{transmission-linear-NS-a-equiv-2T} and the nonlinear equation \eqref{transmission-linear-NS-a-equiv-1T}.)

(ii) Let us assume that condition \eqref{assumption-data-1} holds. Then it is immediate that condition \eqref{assumption-data} holds as well, and, thus, the nonlinear Dirichlet-transmission problem \eqref{int-NS-DT-1} has at least one solution in the space ${\mathfrak X}_{\Omega^+,\Omega^-}$.

Now, assume that the nonlinear problem \eqref{int-NS-DT-1} has two solutions,\\
$({\bf u}^{(1)+},\pi^{(1)+}, {\bf u}^{(1)-},\pi^{(1)-})$ and
$({\bf u}^{(2)+},\pi^{(2)+},{\bf u}^{(2)-},\pi^{(2)-})$ in the space ${\mathfrak X}_{\Omega^+,\Omega^-}$.
Let us represent the velocities in $\Omega^\pm$ in the form
\begin{align}
\label{form-solution-new}
{\bf u}^{(i)\pm}={\bf v}_{\bf \bs\varphi}^\pm+{\bf u}_0^{(i)\pm},\ \ i=1,2,
\end{align}
where $({\bf v}_{\bs\varphi}^+,{\bf v}_{\bs\varphi}^-)\in H^1(\Omega ^+)^n\times H^1(\Omega ^-)^n$ satisfy relations \eqref{int-S-DT-1} and estimate \eqref{estimate-a1}, while
\begin{align}
\label{uNSi}
{\bf u}^{(i)}_0|_{\Omega^+}={\bf u}_0^{(i)+},\quad
{\bf u}^{(i)}_0|_{\Omega^-}={\bf u}_0^{(i)-},\quad
{\rd \pi }^{(i)}|_{\Omega^+} = \pi^{(i)+},\quad
{\rd \pi }^{(i)}|_{\Omega^-} = \pi^{(i)-}\,.
\end{align}
corresponds to the pairs $({\bf u}^{(i)}_0,\pi^{(i)})\in \mathring H^1(\Omega )^n\times L^2(\Omega)/{\mathbb R}$.

Let us also introduce the notations $\overline{\bf u}_0:={\bf u}_0^{(1)}-{\bf u}_0^{(2)}, \quad
\overline\pi:=\pi^{(1)}-\pi^{(2)}$,
$$
\overline{\bf u}^\pm:={\bf u}^{(1)\pm}-{\bf u}^{(2)\pm}=\overline{\bf u}_0|_{\Omega^\pm}, \quad
\overline\pi^\pm:=\pi^{(1)\pm}-\pi^{(2)\pm}=\overline\pi|_{\Omega^\pm}.
$$
Using \eqref{transmission-linear-NS-a-equiv-1T} and \eqref{FuphiT}, we obtain,
\begin{align}
\label{transmission-linear-NS-a-equiv-upm}
\bs{\mathcal L}(\overline{\bf u}_0,\overline\pi)&=
\mathring E_{\Omega^+}[({\bf u}^{(1)+}\cdot \nabla ){\bf u}^{(1)+}-({\bf u}^{(2)+}\cdot \nabla ){\bf u}^{(2)+}]\nonumber\\
&\quad +\mathring E_{\Omega^-}[({\bf u}^{(1)-}\cdot \nabla ){\bf u}^{(1)-}-({\bf u}^{(2)-}\cdot \nabla ){\bf u}^{(2)-}],
\end{align}
that is,
\begin{align*}
\bs{\mathcal L}(\overline{\bf u}_0,\overline\pi)&=
\overline{\bf u}_0\cdot (\nabla {\bf u}^{(1)}_0+\mathring E_{\Omega^+}\nabla {\bf v}_{\bf \bs\varphi}^+
+ \mathring E_{\Omega^-}\nabla {\bf v}_{\bf \bs\varphi}^-)\nonumber\\
&\quad
+ (({\bf u}^{(2)}_0 + \mathring E_{\Omega^+}{\bf v}_{\bf \bs\varphi}^+
+\mathring E_{\Omega^-}{\bf v}_{\bf \bs\varphi}^-)\cdot \nabla )\overline{\bf u}_0\  \mbox{ in } \Omega.
\end{align*}
This implies that
\begin{multline}
\label{NS-var-eq-int-unique-0pm}
\hspace{-1.2em}\big\langle a_{ij}^{\alpha \beta }E_{j\beta }(\overline{\bf u}_0),
E_{i\alpha }(\overline{\bf u}_0)\big\rangle _{\Omega}\!=\!
-\big\langle\big(\overline{\bf u}_0\cdot \nabla \big){\bf u}^{(1)}_0,\overline{\bf u}_0\big\rangle_{\Omega}
-\big\langle\big(\overline{\bf u}_0\cdot (\mathring E_{\Omega^+}\nabla {\bf v}_{\bf \bs\varphi}^+
+ \mathring E_{\Omega^-}\nabla {\bf v}_{\bf \bs\varphi}^-),\overline{\bf u}_0\big\rangle_{\Omega}\\
-\big\langle\big( ( \mathring E_{\Omega^+}{\bf v}_{\bf \bs\varphi}^+
+\mathring E_{\Omega^-}{\bf v}_{\bf \bs\varphi}^-)\cdot \nabla \big)\overline{\bf u}_0,\overline{\bf u}_0\big\rangle_{\Omega}
-\big\langle({\bf u}^{(2)}_0\cdot \nabla )\overline{\bf u}_0,\overline{\bf u}_0\big\rangle_{\Omega}.
\end{multline}
Moreover, identity \eqref{P-5a}
and the inclusion $\overline{\bf u}_0\in \mathring{H}_{\rm{div}}^1({\Omega })^n$ show that the last term in the right-hand side of \eqref{NS-var-eq-int-unique-0pm} equals zero.
Therefore,
inequality \eqref{a-1-v2-S} similar to \eqref{NS-var-eq-int-unique-2a} implies that
\begin{multline}
\label{P-11-unique0}
\frac{1}{2}C_{\mathbb A}^{-1}\|\nabla \overline{\bf u}_0\|_{L^2(\Omega )^{n\times n}}^2\,\leq
\big\langle a_{ij}^{\alpha \beta }E_{j\beta }(\overline{\bf u}_0),
E_{i\alpha }(\overline{\bf u}_0)\big\rangle _{\Omega }\\
\shoveleft{\le\big|\big\langle\big(\overline{\bf u}_0\cdot \nabla \big){\bf u}^{(1)}_0,\overline{\bf u}_0\big\rangle_{\Omega}\big|
+\big|\big\langle\big(\overline{\bf u}_0\cdot (\mathring E_{\Omega^+}\nabla {\bf v}_{\bf \bs\varphi}^+
+ \mathring E_{\Omega^-}\nabla {\bf v}_{\bf \bs\varphi}^-),\overline{\bf u}_0\big\rangle_{\Omega}\big|}\\
\shoveright{+\big|\big\langle\big( ( \mathring E_{\Omega^+}{\bf v}_{\bf \bs\varphi}^+
+\mathring E_{\Omega^-}{\bf v}_{\bf \bs\varphi}^-)\cdot \nabla \big)\overline{\bf u}_0,\overline{\bf u}_0\big\rangle_{\Omega}\big|}\\
\shoveleft{\le  c_0^2\|\nabla {\bf u}_0^{(1)}\|_{L^2(\Omega )^{n\times n}}}\|\nabla \overline{\bf u}_0\|_{L^2(\Omega )^{n\times n}}^2\\
+ 2c_0(c_1^*+c_0)C_\Sigma\|(g^+,g^-,{\boldsymbol \varphi}_{_{\Sigma }},{\boldsymbol \varphi})\|_{\mathcal M_\bullet}
\|\nabla \overline{\bf u}_0\|_{L^2(\Omega )^{n\times n}}^2.
\end{multline}
Employing in \eqref{P-11-unique0} estimate \eqref{weak-D-2} for ${\bf w}={\bf u}_0^{(1)}$,  after some simplifications we obtain
\begin{multline}
\left(\frac{1}{2}C_{\mathbb A}^{-1}
-2c_0(c_1^*+c_0)C_\Sigma\|(g^+,g^-,{\boldsymbol \varphi}_{_{\Sigma }},{\boldsymbol \varphi})\|_{\mathcal M_\bullet}\right)^2
\|\nabla \overline{\bf u}_0\|^2_{L^2(\Omega )^{n\times n}}
\\
\shoveleft{\le c_0^2\Big({|\!|\!|\bs{\mathfrak F}|\!|\!|}_{H^{-1}(\Omega )^n}
+C_{\Sigma}\|{\mathbb A}\|
\|(g^+,g^-,{\boldsymbol \varphi}_{_{\Sigma }},{\boldsymbol \varphi})\|_{\mathcal M_\bullet}}
\\
+ 2c_0c_1^*C^2_{\Sigma}{\|(g^+,g^-,{\boldsymbol \varphi}_{_{\Sigma }},{\boldsymbol \varphi})\|_{\mathcal M_\bullet}^2}\Big)
\|\nabla \overline{\bf u}_0\|^2_{L^2(\Omega )^{n\times n}}\,,
\end{multline}
which implies
\begin{multline*}
\frac{1}{4}C_{\mathbb A}^{-2}\|\nabla \overline{\bf u}_0\|_{L^2(\Omega )^{n\times n}}^2
\leq
\Big(c_0^2|\!|\!|\bs{\mathfrak F}|\!|\!|_{{H}^{-1}(\Omega )^n}
+c_0^2C_\Sigma\|{\mathbb A}\|\, \|(g^+,g^-,{\boldsymbol \varphi}_{_{\Sigma }},{\boldsymbol \varphi})\|_{\mathcal M_\bullet}
\\
+2C_{\mathbb A}^{-1}c_0(c_1^*+c_0)C_\Sigma
\|(g^+,g^-,{\boldsymbol \varphi}_{_{\Sigma }},{\boldsymbol \varphi})\|_{\mathcal M_\bullet}\\
\shoveright{-c_0^2\big(4(c_1^*+c_0)^2-2c_0c_1^*\big)C^2_{\Sigma}
\|(g^+,g^-,{\boldsymbol \varphi}_{_{\Sigma }},{\boldsymbol \varphi})\|_{\mathcal M_\bullet}^2
\Big)
\|\nabla \overline{\bf u}_0\|_{L^2(\Omega )^{n\times n}}^2}
\\
\shoveleft{\le\Big(c_0^2|\!|\!|\bs{\mathfrak F}|\!|\!|_{{H}^{-1}(\Omega )^n}
+c_0^2C_\Sigma\|{\mathbb A}\|\, \|(g^+,g^-,{\boldsymbol \varphi}_{_{\Sigma }},{\boldsymbol \varphi})\|_{\mathcal M_\bullet}}
\\
+2C_{\mathbb A}^{-1}c_0(c_1^*+c_0)C_\Sigma
\|(g^+,g^-,{\boldsymbol \varphi}_{_{\Sigma }},{\boldsymbol \varphi})\|_{\mathcal M_\bullet}
\Big)
\|\nabla \overline{\bf u}_0\|_{L^2(\Omega )^{n\times n}}^2\,,\nonumber
\end{multline*}
In view of condition \eqref{assumption-data-1}, this is possible only if
$$\|\nabla \overline{\bf u}_0\|_{L^2(\Omega )^{n\times n}}^2=\|\nabla ({\bf u}_0^{(1)}-{\bf u}_0^{(2)})\|_{L^2(\Omega )^{n\times n}}=0.$$
Hence, ${\bf u}_0^{(1)}={\bf u}_0^{(2)}$ in $\Omega $ and  relations \eqref{form-solution-new} imply that
${{\bf u}^{(1)\pm}}={{\bf u}^{(2)\pm}}$ in $\Omega ^\pm $.
Finally,
equation \eqref{transmission-linear-NS-a-equiv-upm} leads to $\nabla\overline\pi=0$, i.e., $\pi ^{(1)}=\pi ^{(2)}$ in $L^2(\Omega )/{\mathbb R}$.
\qed
\end{proof}

\appendix
{
{\section*{APPENDIX}
}

\section{\bf Trace and generalized conormal derivative for the anisotropic Stokes system in a Lipschitz domain}\label{A-GCN}

Let $\Omega \subset \mathbb R^n$ be a bounded Lipschitz domain in ${\mathbb R}^n$.
The following trace theorem holds (see \cite{Co}, \cite[Lemma 2.6]{Mikh}, \cite[Theorem 2.5.2]{M-W}).
\begin{theorem}\label{trace-operator1}
Then there exists a linear, bounded trace operator $\gamma_{_{\Omega }}:H^1(\Omega )\to H^{\frac{1}{2}}(\partial\Omega )$ such that $\gamma_{_{\Omega }}f=f_{|\partial\Omega }$ for any $f\in C^{\infty }(\overline{\Omega })$.
The trace operator $\gamma _{_{\Omega }}$ is surjective and has a $($non-unique$)$ linear and bounded right inverse operator $\gamma ^{-1}_{_\Omega }:H^{\frac{1}{2}}(\partial\Omega )\to H^1(\Omega ).$ The trace operator
$\gamma :{H}^1({\mathbb R}^n)\to H^{\frac{1}{2}}(\partial \Omega )$ can also be considered and it is linear and bounded with a bounded linear right inverse $\gamma ^{-1}:H^{\frac{1}{2}}(\partial \Omega )\to {H}^1({\mathbb R}^n)$\footnote{The trace operators defined on Sobolev spaces of vector fields on $\Omega $, or ${\mathbb R}^n$ are also denoted by $\gamma_{_\Omega }$ and $\gamma $, respectively.}.
\end{theorem}

Let $\boldsymbol{\mathfrak L}$ be the divergence form second-order elliptic differential operator given by \eqref{Stokes-0}. The coefficients $A^{\alpha \beta }$ of the anisotropic tensor ${\mathbb A}=\left(A^{\alpha \beta }\right)_{1\leq \alpha ,\beta \leq n}$ are $n\times n$ matrix-valued functions in $L^\infty({\mathbb R}^n)^{n\times n}$, with bounded measurable, real-valued entries $a_{ij}^{\alpha \beta }$, satisfying the symmetry and ellipticity conditions \eqref{Stokes-sym} and \eqref{mu}. Let $\boldsymbol{\mathcal L}$ denote the Stokes operator defined in \eqref{Stokes-new}
and let us consider the anisotropic Stokes system \eqref{Stokes}, i.e.,
\begin{equation}
\label{Stokes-new-1}
\boldsymbol{\mathcal L}({\bf u},\pi )
={\bf f}
,\ {\rm{div}}\ {\bf u}=g \mbox{ in } \Omega
\end{equation}
for $({\bf u},\pi )\in{{H}^1(\Omega )^n}\times  L^2(\Omega )$,
with $\mathbf f \in {H}^{-1}(\Omega )^n$ and $g\in L^2(\Omega )$.
System \eqref{Stokes-new-1} is understood in the sense of distribution, i.e.,
\begin{align}
\label{Stokes-dist}
\left\{\begin{array}{ll}
a_{{\mathbb A};\Omega }({\bf u},{\bf v})+b_{\Omega }({\bf v},\pi )=-\langle \mathbf f,{\bf v}\rangle _{\Omega }\,,
& \forall \, {\bf v}\in \mathcal D(\Omega )^n,\\
b_{\Omega }({\bf u},q)=-\langle g,q\rangle _{\Omega }\,, & \forall \, q\in \mathcal D(\Omega ),
\end{array}
\right.
\end{align}
where
\begin{align}
\label{a-v-A}
a_{{\mathbb A};\Omega }({\bf u},{\bf v})&:
=\left\langle a_{ij}^{\alpha \beta }E_{j\beta }({\bf u}),E_{i\alpha }({\bf v})\right\rangle _{\Omega }
=\left\langle A^{\alpha \beta }\partial _\beta ({\bf u}),\partial _\alpha ({\bf v})\right\rangle _{\Omega },\\
\label{b-v-A}
b_{\Omega }({\bf v},q)&:=-\langle {\rm{div}}\, {\bf v},q\rangle _{\Omega },
\end{align}
The space $\mathcal D(\Omega )^n$ is dense in $ \mathring{H}^1(\Omega )^n$ and $ \mathcal D(\Omega )$ is dense in $L^2(\Omega)$, while the bilinear forms $a_{{\mathbb A};\Omega }:{H}^1(\Omega )^n\times \mathring{H}^{1}(\Omega )^n\to {\mathbb R}$ and $b_{\Omega }: {H}^1(\Omega )^n\times L^2(\Omega )\to {\mathbb R}$,  defined by \eqref{a-v-A} and \eqref{b-v-A}, are bounded.
Hence \eqref{Stokes-dist} implies the following weak formulation of the Stokes system,
\begin{align}
\label{Stokes-weak}
\left\{\begin{array}{ll}
a_{{\mathbb A};\Omega }({\bf u},{\bf v})+b_{\Omega }({\bf v},\pi )=-\langle \mathbf f,{\bf v}\rangle _{\Omega }\,,
& \forall \, {\bf v}\in \mathring{H}^1(\Omega )^n,\\
b_{\Omega }({\bf u},q)=-\langle g,q\rangle _{\Omega }\,, & \forall \, q\in L^2(\Omega ).
\end{array}
\right.
\end{align}
The first equation in \eqref{Stokes-weak} can be considered as the Green identity, as follows.
\begin{lemma}
\label{GI-H0}
Let $\Omega $ be a bounded Lipschitz domain in ${\mathbb R}^n$, $n\geq 2$,
and let conditions \eqref{Stokes-1}, \eqref{Stokes-sym} be satisfied.
Then
the following first Green identity holds
\begin{align}
\label{Green-H0}
\big\langle a_{ij}^{\alpha \beta }E_{j\beta }({\bf u}),E_{i\alpha }({\bf w})\big\rangle _{\Omega }
-\langle {\pi },{\rm{div}}\, {\bf w} \rangle _{\Omega }+\langle {\bf f},{\bf w}\rangle _{{\Omega }}&=0
\end{align}
for all ${\bf w}\in \mathring{H}^{1}({\Omega })^n$,  ${\bf u}\in {H}^{1}({\Omega })^n$ and
$\mathbf f \in {H}^{-1}(\Omega )^n$ such that $\boldsymbol{\mathcal L}({\bf u},\pi )={\bf f}$  in $\Omega$.
\end{lemma}

Let $\boldsymbol\nu =(\nu _1,\ldots ,\nu _n)^\top$ denote the outward unit normal to $\Omega $, which is defined a.e. on the boundary {\bl $\partial {\Omega }$} of $\Omega $.
In the special case, when $({\bf u},\pi)\!\in \!C^1(\overline {\Omega })^n\times C^0(\overline{\Omega })$, the {\em classical}
conormal derivative (i.e., the {\it boundary traction field}) for the anisotropic Stokes system \eqref{Stokes-new-1}
where ${\bf f}\in L^2(\Omega )^n$ and $g\in L^2(\Omega )$, is defined by the formula
\begin{align}
\label{2.37-}
{{\bf t}^{{\rm{c}}}_{\Omega }}({\bf u},\pi ):
=-({\gamma_{_{\Omega }}}\pi)\boldsymbol\nu
+\gamma _{_{\Omega  }}\big(a_{\cdot j}^{\alpha \beta }E_{j\beta }({\bf u})\big)\nu _\alpha
=-({\gamma_{_{\Omega }}}\pi)\boldsymbol\nu
+{\gamma _{_{\Omega }}(A^{\alpha \beta }\partial _\beta {\bf u})\nu _\alpha }
\end{align}
where $E_{j\beta }({\bf u}):=\frac{1}{2}\left(\partial _j(u_\beta )+\partial _\beta (u_j)\right)$ and the symmetry conditions \eqref{Stokes-sym} show that\\
$
\gamma _{_{\Omega  }}\big(a_{ij}^{\alpha \beta }\partial _\beta (u_j)\big)\nu _\alpha
=\gamma _{_{\Omega  }}\big(a_{ij}^{\alpha \beta }E_{j\beta }({\bf u})\big)\nu _\alpha \,.
$
Then the following {\it first Green identity} holds
\begin{multline}
\label{special-v}
{\left\langle {\bf t}^{{\rm{c}}}_{\Omega }({\bf u},\pi ),
{\bf v}\right\rangle _{_{\!\partial\Omega }}}
=\left\langle a_{ij}^{\alpha \beta}E_{j\beta}({\bf u}),E_{i\alpha}({\bf v} )\right\rangle _{\Omega }
-\langle \pi,{\rm{div}}\, {\bf v}\rangle _{\Omega }
+\left\langle \bs{\mathcal L}({\bf u},\pi ),{\bf v}\right\rangle _{{\Omega }},\,
\forall \, {\bf v} \in {\mathcal D}({\mathbb R}^n)^n\,.
\end{multline}

As in \cite[Definition 2.2]{KMW-DCDS2021} and \cite[Definition 1]{KMW-LP}, formula \eqref{special-v} suggests the following definition
(cf. also \cite[Lemma 3.2]{Co}, 
\cite[Definition 3.1, Theorem 3.2]{Mikh}, \cite[Definition 2.4]{K-M-W-2}, \cite[Theorem 10.4.1]{M-W} for other equations).
\begin{defn}
\label{conormal-derivative-var-Brinkman}
Let conditions \eqref{Stokes-1} {\Bn and \eqref{Stokes-sym}} be satisfied and
\begin{align*}
{\pmb{H}}^1(\Omega ,\boldsymbol{\mathcal L})
:=\Big\{({\bf u},\pi ,{\tilde{\bf f}})\in {H}^1({\Omega })^n\times L^2({\Omega })\times \widetilde{H}^{-1}({\Omega  })^n:
\boldsymbol{\mathcal L}({\bf u},\pi )=\tilde{\bf f}|_{\Omega } \mbox{ in } {\Omega }\Big\}.
\end{align*}

If $({\bf u},\pi ,{\tilde{\bf f}})\in {\pmb{H}}^1({\Omega },{\boldsymbol{\mathcal L}})$, then the generalized conormal derivative ${{\bf t}_{\Omega }}({\bf u} ,\pi ;{\tilde{\bf f}})\in H^{-\frac{1}{2}}(\partial\Omega )^n$ is
defined in the weak form as
\begin{multline}
\label{conormal-derivative-var-Brinkman-3}
\big\langle {{\bf t}_{\Omega }}({\bf u} ,\pi ;{\tilde{\bf f}}),{\bs\Phi}\big\rangle _{_{\!\partial\Omega }}:=
\big\langle a_{ij}^{\alpha \beta}E_{j\beta}({\bf u}),E_{i\alpha} (\gamma^{-1}_{_{\Omega }}{\bs\Phi})\big\rangle _{\Omega }
-\big\langle {\pi },{\rm{div}}(\gamma^{-1}_{_{\Omega }}{\bs\Phi})\big\rangle _{\Omega }
\\
+\big\langle {\tilde{\bf f}} ,\gamma^{-1}_{_{\Omega }}{\bs\Phi}\big\rangle _{{\Omega }}, \ \forall\, {\bs\Phi}\!\in\!H^{\frac{1}{2}}(\partial\Omega )^n\,,
\end{multline}
where $\gamma^{-1}_{_{\Omega }}:H^{\frac{1}{2}}(\partial\Omega )^n\to {H}^{1}({\Omega })^n$ is a bounded right inverse of the trace operator $\gamma _{_{\Omega }}:{H}^{1}({\Omega })^n\to H^{\frac{1}{2}}(\partial\Omega )^n$.
\end{defn}
Thus, for given $({\bf u} ,\pi ;{\tilde{\bf f}} )\in {\pmb{ H}}^1({\Omega  },{\boldsymbol{\mathcal L}})$, we have ${{\bf t}_{\Omega }}({\bf u} ,\pi ;{\tilde{\bf f}})\in H^{-\frac{1}{2}}(\partial\Omega )^n$.
In addition, \cite[Lemma 2.3]{KMW-DCDS2021}, \cite[Lemma 1]{KMW-LP} imply the following property (see also \cite{Co},
\cite[Definition 3.1, Theorem 3.2]{Mikh}, \cite{Mikh-3}, \cite[Theorem 10.4.1]{M-W}).
\begin{lemma}
\label{lem-add1}
Let $\Omega $ be a bounded Lipschitz domain in ${\mathbb R}^n$, $n\geq 2$,
and let conditions \eqref{Stokes-1}, \eqref{Stokes-sym} hold.
Then the generalized conormal derivative operator ${{\bf t}_{{\Omega }}}:\pmb{H}^1(\Omega ,\boldsymbol{\mathcal L})\to H^{-\frac{1}{2}}(\partial \Omega )^n$ is linear and bounded, and definition \eqref{conormal-derivative-var-Brinkman-3} does not depend on the particular choice of a right inverse $\gamma _{_{\Omega  }}^{-1}:H^{\frac{1}{2}}(\partial \Omega )^n\!\to \!{H}^{1}(\Omega )^n$ of the trace operator $\gamma _{_{\Omega }}:{H}^{1}(\Omega )^n\to H^{\frac{1}{2}}(\partial \Omega )^n$. 
In addition, for all ${\bf w}\in {H}^{1}({\Omega })^n$ and
$({\bf u},\pi ,{\tilde{\bf f}})\in {\pmb{H}}^1({\Omega },{\boldsymbol{\mathcal L}})$,
the following first Green identity holds,
\begin{align}
\label{Green-particular-p}
\big\langle{{\bf t}_{\Omega }}({\bf u},\pi ;{\tilde{\bf f}}),\gamma_{_{\Omega }}{\bf w}\big\rangle _{_{\partial\Omega }}
&=\big\langle a_{ij}^{\alpha \beta }E_{j\beta }({\bf u}),E_{i\alpha }({\bf w})\big\rangle _{\Omega }-\langle {\pi },{\rm{div}}\, {\bf w} \rangle _{\Omega }+\langle {\tilde{\bf f}},{\bf w}\rangle _{{\Omega }}.
\end{align}
\end{lemma}
We also adopt the simplified notation ${{\bf t}_{{\Omega }}}({\bf u},\pi )$ for ${\bf t}_{{\Omega}}({\bf u},\pi ;{\bf 0})$.

\section{Extension results for Sobolev spaces on Lipschitz domains with Lipschitz interfaces}

Let $\Omega \subset \mathbb R^n$, $n\geq 2,$ be a bounded Lipschitz domain satisfying Assumption \ref{interface-Sigma}.
Thus, $\Omega =\Omega ^+\cup \Sigma \Omega ^-$, where $\Sigma $ is the $(n-1)$-dimensional Lipschitz interface between the disjoint Lipschitz sub-domains $\Omega ^+$ and $\Omega ^-$, and $\Sigma $
meets transversally $\partial \Omega $.
The boundary $\partial \Omega ^\pm $ of $\Omega ^\pm $ is partitioned into {two relatively open subsets} $\Gamma ^\pm $ and $\Sigma $, and $\Gamma ^+$ and $\Gamma ^-$ are not empty. Let $\gamma _{_{\Omega ^\pm }}$ be the trace operator from ${H}^{1}(\Omega ^\pm )$ to $H^{\frac{1}{2}}(\partial \Omega ^\pm)$.

The proof of the following extension property is based on similar arguments to those for Theorem 5.13 in \cite{B-M-M-M} (see also Lemma C.1 in \cite{KMW-LP}). 
We omit the details for the shortness.
\begin{lemma}
\label{extention}
The following assertions hold.
\begin{itemize}
\item[$(i)$]
Let $u^+\in H^1(\Omega ^+)$ and $u^-\in {H}^1(\Omega ^-)$ be such that
$\gamma _{_{\Omega ^+}}u^+\!=\!\gamma _{_{\Omega ^-}}u^-$ on $\Sigma $.
Then there exists a unique function $u\in H^1(\Omega )$ such that
$
{u|_{\Omega^\pm}\!=\!u^\pm}.
$
Moreover, there exists $C\!=\!C(n,\Omega ^\pm )>0$ such that
$\|u\|_{{H}^1(\Omega )}\leq C\left(\|u^+\|_{H^1(\Omega ^+)}+\|u^-\|_{{H}^1(\Omega ^-)}\right).$
\item[$(ii)$]
If $u\in H^1(\Omega )$ then $\gamma _{_{\Omega ^+}}(u|_{\Omega ^+})=\gamma _{_{\Omega ^-}}(u|_{\Omega ^-})$ on $\Sigma $.
\end{itemize}
\end{lemma}

\begin{lemma}\label{L3}
Let $\Gamma_1$ and $\Gamma_2$ be two $(n-1)$-dimensional Lipschitz hyper-surfaces in $\R^n$, $n\ge 2$, that coincide on a relatively open $(n-1)$-dimensional subset $\Gamma_0$ having a Lipschitz boundary.
Assume that either one of the following assumptions holds:
\begin{enumerate}
\item[$(1)$]
$\Gamma_1$ and $\Gamma_2$ are the graphs of two Lipschitz functions;
\item[$(2)$]
$\Gamma_1$ and $\Gamma_2$ can be mapped by rigid rotations into two Lipschitz graphs;
\item[$(3)$]
$\Gamma_1$ and $\Gamma_2$ are two bounded Lipschitz surfaces in $\R^n$.
\end{enumerate}
Let $0\le s\le 1$ and functions $f_i\in L^2(\Gamma_i)$, $f_i=0$ on $\Gamma_i\setminus\overline\Gamma_{0}$,  $i=1,2$,  and $f_2=f_1$ on $\Gamma_0$.
Then $f_1\in H^s(\Gamma_1)$ if and only if $f_2\in H^s(\Gamma_2)$.
\end{lemma}
\begin{proof}
(1)
Let $\Gamma_1$ and $\Gamma_2$ be graphs of two Lipschitz functions $x_n=\zeta_1(x')$ and $x_n=\zeta_2(x')$, $x'\in\R^{n-1}$, and $\Gamma_0$ be the image of a domain $S_0\subset \R^{n-1}$, i.e., $x_n=\zeta_1(x')=\zeta_2(x')$ for $x'\in S_0\subset \R^{n-1}$.

By the definition of the Sobolev spaces on Lipschitz graphs (see, e.g., \cite[p. 98]{Lean}),
$f_1\in H^s(\Gamma_1)$, $0\le s\le 1$, means that $f_{\zeta_1}\in H^s(\R^{n-1})$, where $f_{\zeta_1}(x')=f_{1}(x',\zeta_1(x'))$, $x'\in\R^{n-1}$.
On the other hand, $f_{\zeta_2}(x')=f_{2}(x',\zeta_2(x'))=f_{1}(x',\zeta_1(x'))$ for $x'\in S_0$, and $f_{\zeta_2}(x')=f_{2}(x',\zeta_2(x'))=0=f_{1}(x',\zeta_1(x'))$ for $x'\in \R^{n-1}\setminus\overline S_0$.
Hence $f_{\zeta_2}(x')=f_{\zeta_1}(x')$ for almost any $x'\in \R^{n-1}$ which implies $f_1\in H^s(\Gamma_1)$ if and only if $f_2\in H^s(\Gamma_2)$.

(2) Let further $i=1,2$.
By the assumption of item (2), there exist constant invertible rotation matrices $\Phi_i\in\R^{n\times n}$ such that
$\Gamma^*_i=\{x=\Phi_i y,\ y\in\Gamma_i\}$, $i=1,2$, are  Lipschitz graphs, i.e.,
they are represented by two Lipschitz functions, $x_{1,n}=\zeta_1(x'_1)$ and $x_{2,n}=\zeta_2(x'_2)$, where $x_{1,n},x_{2,n}\in\R$ and $x'_1, x'_2\in\R^{n-1}$.
By the definition of Sobolev spaces on Lipschitz hyper-surfaces, the inclusion $f_i\in H^s(\Gamma_i)$ implies that $f_{\zeta_i}\in H^s(\R^{n-1})$, where $f_{\zeta_i}(x')=f_{i}(\Phi^{-1}_i(x'_i,\zeta_i(x'_i))$, $x'_i\in\R^{n-1}$.

Let $\Gamma^*_{0i}\subset \Gamma^*_{i}$ be {\Rd the} images of $\Gamma_0$ through the above mapping, i.e.,
$\Gamma^*_{0i}\ni x_i=\Phi_i y$, $y\in \Gamma_{0}$, and, on the other hand,
$\Gamma^*_{0i}\ni x_i=(x'_i,\zeta_i(x'_i))$, $x'_i\in S_{0i}$, where $S_{0i}$ are Lipschitz domains in $\R^{n-1}$.
Then
\begin{align*} 
\Phi^{-1}_2 x_2=\Phi^{-1}_2(x'_2,\zeta_2(x'_2))=y=\Phi^{-1}_1(x'_1,\zeta_1(x'_1))=\Phi^{-1}_1 x_1,\quad
\forall\, y\in\Gamma_0,\ x'_1\in S_{01},\ x'_2\in S_{02},
\end{align*}
and
\begin{multline}
\label{E7g}
f_{\zeta_2}(x'_2)=f_{2}(\Phi^{-1}_2(x'_2,\zeta_2(x'_2))=f_{2}(y)=f_{1}(y)=f_{1}(\Phi^{-1}_1(x'_1,\zeta_1(x'_1))
=f_{\zeta_1}(x'_1),\\
\forall\,y\in\Gamma_0,\ x'_1\in S_{01},\ x'_2\in S_{02}.
\end{multline}
By definition of the rigid rotation matrices and graph functions, we have,
$$(x'_1,\zeta_1(x_1'))=\Phi_1 y =(\Phi_1' y, \Phi_{1,n} y),\quad
(x'_2,\zeta_2(x_2'))=\Phi_2 y =(\Phi_2' y, \Phi_{2,n} y),\quad y\in \Gamma_0,$$
where the matrices $\Phi_i'\in\R^{(n-1)\times n}$ and $\Phi_{i,n}\in\R^{1\times n}$ are parts of the corresponding matrices $\Phi_i$.
Then
\begin{align}
\label{E9g}
\hspace{-1em}x'_1=\Phi_1' y=\Phi_1' \Phi^{-1}_2\left(x'_2,\zeta_2(x'_2)\right),\
x'_2=\Phi_2' y=\Phi_2' \Phi^{-1}_1\left(x'_1,\zeta_1(x'_1)\right),\ x'_1\in \bar S_{01},\ x'_2\in \bar S_{02}.
\end{align}
Since $\zeta_1$ and $\zeta_2$ are Lipschitz functions, relations \eqref{E9g} imply that $x'_1=x'_1(x'_2)$ and $x'_2=x'_2(x'_1)$ are mutually inverse bi-Lipschitz mappings for $x'_1\in \bar S_{01}$, $x'_2\in \bar S_{02}$.

Assume now that $f_1\in H^s(\Gamma_1)$ and $f_1=0$ on $\Gamma_1\setminus\overline\Gamma_{0}$.
Then $f_{\zeta_1}\in H^s(\R^{n-1})$ and $f_{\zeta_1}=0$ on $\R^{n-1}\setminus\overline S_{01}$.
Assume also that a function $f_2\in L^2(\Gamma_2)$ is such that $f_2=f_1$ on $\Gamma_0$ and $f_2=0$ on $\Gamma_2\setminus\overline\Gamma_{0}$.
From \eqref{E7g} we have for any $x'_2\in S_{02}$ that
\begin{align}
\label{E10g}
f_{\zeta_2}(x'_2)=f_{\zeta_1}(x'_1(x'_2)),
\end{align}
where the Lipschitz map $x'_1(x'_2)$ is defined for $x'_2\in S_{02}$ by \eqref{E9g}.
By the Kirszbraun theorem (cf., Lemma 1.29 and Theorem 1.31 in \cite{Schwartz1969}) the map $x'_1(x'_2)$ can be extended to all $x'_2\in \R^{n-1}$ with the same Lipschitz constant.
Hence, taking into account that $f_{\zeta_1}(x'_1)=0$ for $x'_1\in\R^{n-1}\setminus\overline S_{01}$ and
$f_{\zeta_2}(x'_2)=0$ for $x'_2\in\R^{n-1}\setminus\overline S_{02}$, we obtain that
 for such extension, relation \eqref{E10g} holds for almost any $x'_2\in \R^{n-1}$.
Then, cf., e.g., Theorem 3.23 in \cite{Lean}, the inclusion $f_{\zeta_1}\in H^s(\R^{n-1})$ implies that $f_{\zeta_2}\in H^s(\R^{n-1})$ and thus $f_2\in H^s(\Gamma_2)$.

(3) Assume now that $\Gamma_1$ and $\Gamma_2$ are bounded Lipschitz hyper-surfaces in $\R^n$ and arrange finite covers of both of them by open balls such that the intersections of each ball with the corresponding surface can be extended (possibly after some rigid rotations) to Lipschitz graphs.
Moreover we choose the covers in such a way that the balls covering the closure of $\Gamma_0$ coincide for both surfaces. Arranging the subordinate partition of unity (see, e.g., \cite[p. 98]{Lean}) and employing item (2) for the balls intersecting the boundary of $\Gamma_0$ yield the asserted result.
\qed
\end{proof}

Let us show that the space ${H}_\bullet^s(\cdot)$ can be characterized as the weighted space ${H}_{00}^{s}(\cdot)$, whose counterpart on smooth domains in $\R^n$ was given in \cite[Chapter 1, Theorem 11.7]{LiMa1}, see also Corollary 1.4.4.10 in \cite{Grisvard1985} for Lipschitz domains.
\begin{theorem}\label{Tbl}
Let $\Gamma$  be a $(n-1)$-dimensional Lipschitz graph or a bounded Lipschitz hyper-surface in $\R^n$, $n\ge 2$, and let $\Gamma_0$ be its relatively open $(n-1)$-dimensional subset with a $(n-2)$-dimensional Lipschitz boundary $\partial\Gamma_0$.
Let $0< s< 1$.

Let ${H}_{00}^{s}(\Gamma_0)$ denote the space of all functions $\phi\in\mathring{H}^{s}(\Gamma_0)$,  such that $\delta^{-s}\phi\in L^2(\Gamma_0)$, where $\delta(x)$ is the distance in $\R^n$ from $x$ to the boundary $\partial\Gamma_0$.

Then the space ${H}_{00}^{s}(\Gamma_0)$ coincides with the space ${H}_\bullet^{s}(\Gamma_0)$, i.e., with the space of all functions from ${H}^{s}(\Gamma_0)$ such that their extensions by zero to $\Gamma$ belong to ${H}^{s}(\Gamma)$.
\end{theorem}
\begin{proof}
Let first $\Gamma$  be graph of a Lipschitz function $x_n=\zeta(x')\in\R$, $x'\in\R^{n-1}$, and $\Gamma_0$ be the image of a domain $S_0\subset \R^{n-1}$, i.e., $x_n=\zeta(x')$ for $x'\in S_0\subset \R^{n-1}$.
{\Rd For all} $x,\tilde x\in\overline\Gamma_0$, we have
\begin{align}
|x'-\tilde x'|\le|x-\tilde x|&=\sqrt{|x'-\tilde x^{\prime}|^2+|\zeta(x')-\zeta(\tilde x^{\prime})|^2}
\le\sqrt{1+A^2}\,|x'-\tilde x^{\prime}|. \label{E13}
\end{align}
where $A$ is a finite Lipschitz constant of the function $\zeta$ on the domain $S_{0}$.
The distance to the boundary is defined as
\begin{align}\label{EA4a}
\delta(x)&=\inf_{\tilde x\in\partial \Gamma_{0}}|x-\tilde x|
=\inf_{\tilde x^{\prime}\in\partial S_{0}}\sqrt{|x'-\tilde x^{\prime}|^2+|\zeta(x')-\zeta(\tilde x^{\prime})|^2},
\end{align}
Denoting
$
\delta'(x')=\inf_{\tilde x^{\prime}\in\partial S_{0}}|x'-\tilde x^{\prime}|,
$
we obtain from \eqref{EA4a} and \eqref{E13} that
\begin{align}\label{EA5}
\delta'(x')\le \delta(x)&
\le\sqrt{1+A^2}\,\delta'(x').
\end{align}

By the definition of the Sobolev spaces on Lipschitz graphs (see, e.g., \cite[p. 98]{Lean}),
$\tilde f\in H^s(\Gamma)$, $0< s< 1$, means that $\tilde f_{\zeta}\in H^s(\R^{n-1})$, where
$f_{\zeta}(x')=f(x',\zeta(x'))$, $x'\in \R^{n-1}$
and $\phi\in \mathring H^s(\Gamma_0)$, $0< s< 1$, means that $\phi_{\zeta}\in \mathring H^s(S_0)$.

Let $\phi\in{H}_{00}^{s}(\Gamma_0)$.
Then $\phi \in\mathring{H}^{s}(\Gamma_0)$, $\delta^{-s}\phi\in L^2(\Gamma_0)$.
The surface measure formula
\begin{align}\label{EA6}
d\sigma(x)=\sqrt{1+|{\rm grad}\,\zeta(x')|^2}d x'
\end{align}
(see, e.g. \cite[Eq. (3.28)]{Lean}) together with \eqref{EA5} implies that
$\phi_\zeta \in\mathring{H}^{s}(S_0)$ and $(\delta')^{-s}\phi_\zeta\in L^2(S_0)$, where $\phi_{\zeta}(x')=\phi(x',\zeta(x'))$, $x'\in S_0$.
Then by Corollary 1.4.4.10 in \cite{Grisvard1985} we obtain that
the extension of $\phi_{\zeta}$ by zero from $S_0$ to $\R^{n-1}$ belongs to $H^s (\R^{n-1})$ and hence the extension of $\phi$ by zero from $\Gamma_0$ to $\Gamma$ belongs to $H^s (\Gamma)$.

Conversely, let $\phi \in\mathring{H}^{s}(\Gamma_0)$ be such that its extension by zero from $\Gamma_0$ to $\Gamma$ belongs to $H^s (\Gamma)$.
Then $\phi_{\zeta}(x')=\phi(x',\zeta(x'))$ belongs to $\mathring{H}^{s}(S_0)$ and its extension by zero from  to $S_0$ to $\R^{n-1}$ belongs to $H^s (\R^{n-1})$.
Hence by Corollary 1.4.4.10 in \cite{Grisvard1985} we obtain that $(\delta')^{-s}\phi_\zeta\in L^2(S_0)$, which by \eqref{EA5} and \eqref{EA6} implies that $\delta^{-s}\phi\in L^2(\Gamma_0)$ and thus $\phi\in{H}_{00}^{s}(\Gamma_0)$.

Let now $\Gamma$  be a $(n-1)$-dimensional bounded Lipschitz hyper-surface. Let us arrange a finite cover of the hyper-surface by open balls such that, as usual, the intersections of each ball with the hyper-surface can be extended (maybe after corresponding rigid rotations) to a Lipschitz graph. Arranging the subordinate partition of unity (see, e.g., \cite[p. 98]{Lean}) and employing the above arguments to each of these graphs we obtain the asserted result.
\qed
\end{proof}

In addition, the following extension result holds (see also \cite[p. 373]{Sayas-book}).
\begin{lemma}
\label{identification-E}
Let $n\geq 2$ and $\Omega \subset {\mathbb R}^n$ be a bounded Lipschitz domain satisfying Assumption $\ref{interface-Sigma}$.
{Then there exists an extension operator $E_{\Omega ^+\to \Omega }$ from the space ${H}_{\Gamma ^+}^1(\Omega ^+)^n$ to $\mathring {H}^1(\Omega )^n$.}
\end{lemma}
\begin{proof}
Let ${\bf u}^+\in {H}_{\Gamma ^+}^1(\Omega ^+)^n$. Let ${\bf u}_{_\Omega }\in {H}_{\Gamma ^+}^1(\Omega )^n$ be the function defined by
$
{\bf u}_{_\Omega }:=
\mathcal E_{\Omega ^+\to \Omega }{\bf u}^+\,,
$
where $\mathcal E_{\Omega ^+\to \Omega }:=r_{\Omega }\circ E_{\Omega ^+\to {\mathbb R}^n}$, and $E_{\Omega ^+\to {\mathbb R}^n}$ is the Rychkov extension operator from ${H}^1(\Omega ^+)^n$ to ${H}^1({\mathbb R}^n)^n$ (cf., e.g., \cite[Theorem 2.4.1]{M-W}). Thus, $r_{\Omega ^+}({\bf u}_{_\Omega })={\bf u}^+$, where {$r_{\Omega ^+}$
denotes} the restriction to $\Omega ^+$.

In addition, $(\gamma _{_{\Omega }}{\bf u}_{_\Omega })\!\in \!\widetilde H^{\frac{1}{2}}(\Gamma ^-)^n$ and $(\gamma _{_{\Omega }}{\bf u}_{_\Omega })|_{_{\Gamma ^-}}\!\in \! H^{\frac{1}{2}}(\Gamma ^-)^n$. Let $\mathring{E}_{\Gamma ^-\to \Sigma }\left(\left(\gamma _{_{\Omega }}{\bf u}_{_\Omega }\right)|_{\Gamma ^-}\right)$ be the extension of $(\gamma _{_{\Omega }}{\bf u}_{_\Omega })|_{\Gamma ^-}$ by zero on $\Sigma $. Then Lemma \ref{L3} (ii) (applied to the functions ${\bf f}_1=\gamma _{_{\Omega }}{\bf u}_{_{\Omega }}$ and ${\bf f}_2=\mathring{E}_{\Gamma ^-\to \Sigma }\left((\gamma _{_{\Omega }}{\bf u}_{_\Omega })|_{\Gamma ^-}\right)$, which are equal on $\Gamma ^-$ and vanish on $\Gamma ^+$ and $\Sigma $, respectively) implies that $\mathring{E}_{\Gamma ^-\to \Sigma }\left((\gamma _{_{\Omega }}{\bf u})|_{\Gamma ^-}\right)\in \widetilde{H}^{\frac{1}{2}}(\Gamma ^-)^n$. Moreover, by Theorem \ref{trace-operator1}, there exists ${\bf u}^-\in H^1_{_{\Sigma }}(\Omega ^-)^n$ such that
\begin{align}
\label{R-2}
{\bf u}^-=\gamma _{_{\Omega ^-}}^{-1}\big(\mathring{E}_{\Gamma ^-\to \Sigma }\big((\gamma _{_{\Omega }}{\bf u}_{_\Omega })|_{_{\Gamma ^-}}\big)\big)\,,
\end{align}
where $\gamma _{_{\Omega ^-}}^{-1}\!:\!H^{\frac{1}{2}}(\partial \Omega ^-)^n\!\to \!H^1(\Omega ^{-})^n$ is a right inverse of the trace map $\gamma _{_{\Omega ^-}}\!:\!H^1(\Omega ^{-})^n\!\to \! H^{\frac{1}{2}}(\partial \Omega ^-)^n$.
Thus, ${\bf u}^-\!\in \!H^1(\Omega ^-)^n$, $(\gamma _{_{\Omega ^-}}{\bf u}^-)|_{_{\Sigma }}\!=\!0$, $\left(\gamma _{_{\Omega ^-}}{\bf u}^-\right)\big|_{_{\Gamma ^-}}\!=\!(\gamma _{_{\Omega }}{\bf u}_{_\Omega })|_{_{\Gamma ^-}}$.
Let ${\bf u}_0$ be the extension by zero of ${\bf u}^-$ in $\Omega ^+$,
${\bf u}_0:=\mathring{E}_{\Omega ^-\to \Omega^+}{\bf u}^-\,.$
Therefore, ${\bf u}_0\in H^1(\Omega )^n$ and ${\bf u}_0=0$ in $\Omega ^+$. In addition, $(\gamma _{_{\Omega }}{\bf u}_0)|_{\Gamma ^-}=(\gamma _{_{\Omega }}{\bf u}_{_\Omega })|_{_{\Gamma ^-}}$.
Moreover,
\begin{align}
{\bf u}&:={\bf u}_{_\Omega }-{\bf u}_0=\mathcal E_{\Omega ^+\to \Omega }{\bf u}^+-\mathring{E}_{\Omega ^-\to \Omega^+}{\bf u}^-\nonumber\\
&=\mathcal E_{\Omega ^+\to \Omega }{\bf u}^+-\mathring{E}_{\Omega ^-\to \Omega^+}\left(\gamma _{_{\Omega ^-}}^{-1}\left(\mathring{E}_{\Gamma ^-\to \Sigma }\big((\gamma _{_{\Omega }}{\bf u}_{_\Omega })|_{_{\Gamma ^-}}\big)\right)\right)\,.
\end{align}
satisfies ${\bf u}\in H^1(\Omega )^n$ and, by construction, $(\gamma _{_{\Omega }}{\bf u})|_{_{\Gamma ^\pm }}=0$, i.e., $\gamma _{_{\Omega }}{\bf u}=0$ (a.e.) on $\Gamma $. Moreover, ${\bf u}|_{_{\Omega ^+}}={\bf u}^+$ in $\Omega ^+$. Consequently, ${\bf u}$ 
is an extension of ${\bf u}^+$ from ${H}_{\Gamma ^+}^1(\Omega ^+)^n$ to $\mathring H^1(\Omega )^n$.

Finally, we define the extension operator $E_{\Omega ^+\to \Omega }:{H}_{\Gamma ^+}^1(\Omega ^+)^n\to \mathring {H}^1(\Omega )^n$, such that $E_{\Omega ^+\to \Omega }({\bf u}^+):={\bf u}$, where ${\bf u}$ has been constructed above. Consequently,
\begin{align}
E_{\Omega ^+\to \Omega }{\bf u}^+:=\mathcal E_{\Omega ^+\to \Omega }{\bf u}^+-\mathring{E}_{\Omega ^-\to \Omega^+}\left(\gamma _{_{\Omega ^-}}^{-1}\left(\mathring{E}_{\Gamma ^-\to \Sigma }\left(\gamma _{_{\Omega }}\left(\mathcal E_{\Omega ^+\to \Omega }{\bf u}^+\right)|_{_{\Gamma ^-}}\right)\right)\right)\,.
\end{align}
{An alternative construction of such an extension map can be consulted in \cite[pp. 373, 374]{Sayas-book}.} \qed
\end{proof}

Let us introduce the space (cf., e.g., \cite[p. 76]{Lean}),
\begin{align}
H^{-1}_{_{\overline{\Gamma^+}}}(\R^n)^n
=\{\boldsymbol\Phi\in H^{-1}(\R^n)^n: {\rm supp\,} \boldsymbol\Phi\subseteq \overline{\Gamma^+}\}.
\end{align}
Note that,
$\Phi\in H^{-1}_{_{\overline{\Gamma^+}}}(\R^n)^n$ if and only if $\bs\Phi=\gamma_\Omega^*\bs\phi$, i.e.,
\begin{align}\label{EA.5a}
\langle {\bf v}, \bs\Phi\rangle_{\Omega}=\langle \gamma_\Omega{\bf v}, \bs\phi\rangle_{\partial\Omega}\quad
\forall\  {\bf v}\in {H}^1(\Omega)^n .
\end{align}
for some $\bs\phi\in\widetilde H^{-\frac{1}{2}}(\Gamma^+)^n$, which is uniquely defined by $\bs\Phi$, cf. \cite[Theorem 2.10(ii)]{Mikh}.

\begin{lemma}
\label{identification-Omega}
The dual $\left({H}_{\Gamma ^+}^1(\Omega)^n\right)'$ of the space ${H}_{\Gamma ^+ }^1(\Omega)^n$ can be identified with the space
$
\widetilde H^{-1}(\Omega)^n/H^{-1}_{_{\overline{\Gamma^+}}}(\R^n)^n.
$
\end{lemma}
\begin{proof}
First,
we remark that due to \eqref{EA.5a}, the space ${H}_{\Gamma ^+ }^1(\Omega ^+)^n$ defined as in \eqref{H-D-a}
can be also equivalently defined as
\begin{align*}
{H}^1_{\Gamma^+}(\Omega)^n
&=\left\{{\bf v}\in {H}^1(\Omega)^n :
\gamma_\Omega{\bf v}= 0\mbox{ on }\Gamma^+\right\}\\
&=\left\{{\bf v}\in {H}^1(\Omega)^n :
\langle \gamma_\Omega{\bf v}, \bs\phi\rangle_{\partial\Omega}= 0\,, \ \forall\ \bs\phi\in \widetilde H^{-\frac{1}{2}}(\Gamma^+)^n\right\}\\
&=\left\{{\bf v}\in {H}^1(\Omega)^n :
\langle {\bf v}, \bs\Phi\rangle_{\Omega}= 0\,, \ \forall\ \bs\Phi\in H^{-1}_{_{\overline{\Gamma^+}}}(\R^n)^n\right\}
:={H}^1(\Omega)^n\perp H^{-1}_{_{\overline{\Gamma^+}}}(\R^n)^n.
\end{align*}
Then a duality argument
(see, e.g., \cite[Sections 4.8, 4.9]{Rudin1991}) {\Rd yields} the required identification.
\qed
\end{proof}

\begin{lemma}
\label{identification}
The dual $\big({H}_{\Gamma ^+}^1(\Omega ^+)^n\big)'$ of the space ${H}_{\Gamma ^+ }^1(\Omega ^+)^n$ can be identified with the space
$\widetilde H^{-1}(\Omega^+)^n/H^{-1}_{_{\overline{\Gamma^+}}}(\R^n)^n$
and also with the space
\begin{align}
\label{equivalence-set}
\left\{\boldsymbol\varphi \in H^{-1}(\Omega )^n: \boldsymbol\varphi ={\bf 0} \mbox{ on } \Omega ^-\right\}.
\end{align}
\end{lemma}
\begin{proof}
The identification with
$\widetilde H^{-1}(\Omega^+)^n/H^{-1}_{_{\overline{\Gamma^+}}}(\R^n)^n$
follows from Lemma \ref{identification-Omega} applied to $\Omega^+$.

To prove the identification with the space in \eqref{equivalence-set}, assume first that $\boldsymbol\varphi$ belongs to the space defined in \eqref{equivalence-set}. Thus,
\begin{align}\label{E3.24}
\boldsymbol\varphi \in {H^{-1}(\Omega )^n=\big(\mathring{H}^1(\Omega )^n\big)'}\, \mbox{ and }\,
{\langle\boldsymbol\varphi,\boldsymbol\psi ^-\rangle_{\Omega}
=\langle\boldsymbol\varphi,\boldsymbol\psi ^-\rangle_{\Omega^-}= 0},
\ \ \forall \, \boldsymbol\psi ^-\in {\mathcal D}(\Omega ^-)^n\,,
\end{align}
i.e., {$\boldsymbol\varphi $ has the support in $\overline{\Omega ^+}$
($\boldsymbol\varphi ={\bf 0}$ on $\Omega ^-$)}. We have used the equivalent description of the space $\mathring{H}^1(\Omega )^n$ given in formula \eqref{duality-spaces} and the identification of $\mathring{H}^1(\Omega )^n$ and $\widetilde{H}^1(\Omega )^n$.

Note that $\boldsymbol\varphi :{\mathring{H}^1(\Omega )^n}\to {\mathbb R}$ is linear and bounded. Let {$E_{\Omega ^+\to \Omega }$ be an extension operator from ${H}_{\Gamma ^+}^1(\Omega ^+)^n$ to $\mathring{H}^1(\Omega )^n$, which exists in view of Lemma \ref{identification-E}}.
Therefore, {the mapping}
\begin{align*}
&\boldsymbol\phi :=\boldsymbol\varphi \circ E_{\Omega ^+\to \Omega }:{H}_{\Gamma ^+}^1(\Omega ^+)^n\to \mathbb R\,,
\quad
{\boldsymbol\phi (\boldsymbol\psi ^+):=\boldsymbol\varphi \left(E_{\Omega^+\to\Omega}{\boldsymbol\psi ^+}\right),\, \ \ \forall \, \boldsymbol\psi ^+\in {H}_{\Gamma ^+}^1(\Omega ^+)^n}
\end{align*}
is linear and bounded as well. Hence, we have that
\begin{align*}
{\boldsymbol\phi (\boldsymbol\psi ^+)\!=\!\boldsymbol\varphi \left(\boldsymbol\psi \right)},\quad {\forall \, \boldsymbol\psi ^+\!\in \!{H}_{\Gamma ^+}^1(\Omega ^+)^n, \mbox{ where }
\boldsymbol\psi \!=\!E_{\Omega ^+\to\Omega }(\boldsymbol\psi ^+)\!\in \!\mathring{H}^1(\Omega )^n.}
\end{align*}
This definition agrees well with the condition that ${\boldsymbol\varphi ({\boldsymbol\psi ^-)}={\bf 0}}$ for any $\boldsymbol\psi ^-\in {\mathcal D}(\Omega ^-)^n$, and shows that for a fixed $E_{\Omega ^+\to\Omega}$, any functional $\boldsymbol\varphi $ from the space \eqref{equivalence-set} can be identified with a functional $\boldsymbol\phi \in\big({H}_{\Gamma ^+}^1(\Omega ^+)^n\big)'$.

To prove that the functional $\boldsymbol\phi$ does not depend on the extension operator $E_{\Omega ^+\to\Omega}$, let us consider two such extension operators, $E_{\Omega ^+\to\Omega}'$ and $E_{\Omega ^+\to\Omega}''$ from ${H}_{\Gamma ^+}^1(\Omega ^+)^n$ to $\mathring{H}^1(\Omega )^n$, generating for a fixed $\boldsymbol\varphi $ two functionals, $\boldsymbol\phi' :=\boldsymbol\varphi \circ E_{\Omega ^+\to \Omega }'$ and $\boldsymbol\phi'' :=\boldsymbol\varphi \circ E_{\Omega ^+\to \Omega }''$.
Then
\begin{align*}
\boldsymbol\phi' (\boldsymbol\psi ^+)-\boldsymbol\phi''(\boldsymbol\psi ^+)
&=\boldsymbol\varphi \left(E_{\Omega^+\to\Omega}'\boldsymbol\psi ^+\right)
-\boldsymbol\varphi \left(E_{\Omega^+\to\Omega}''\boldsymbol\psi ^+\right)\\
&=\boldsymbol\varphi \left(E_{\Omega^+\to\Omega}'\boldsymbol\psi ^+
-E_{\Omega^+\to\Omega}''\boldsymbol\psi ^+\right),
\, \ \ \forall \, \boldsymbol\psi ^+\in {H}_{\Gamma ^+}^1(\Omega ^+)^n.
\end{align*}
Denoting $\boldsymbol\psi_0:=E_{\Omega^+\to\Omega}'\boldsymbol\psi ^+-E_{\Omega^+\to\Omega}''\boldsymbol\psi ^+$, we obtain that $\boldsymbol\psi_0\in\mathring{H}^1(\Omega )^n$ and $\boldsymbol\psi_0=\mathbf 0$ in $\Omega^+$, implying that $r_{_{\Omega^-}}\boldsymbol\psi_0\in\mathring{H}^1(\Omega^-)^n$ and hence $\gamma_{_{\Omega^-}}\boldsymbol\psi_0=\mathbf 0$.
Thus, there exists $\boldsymbol\Psi_0\in \widetilde H^1(\Omega^-)^n\subset\widetilde H^1(\Omega)^n\subset H^1(\R^n)^n$ such that $\boldsymbol\psi_0=r_{_{\Omega}}\boldsymbol\Psi_0$ and a sequence $\{\boldsymbol\Psi_i\}_{i\in {\mathbb N}}\subseteq {\mathcal D}(\Omega ^-)^n$ converging to $\boldsymbol\Psi_0$ in $\widetilde H^1(\Omega^-)^n$.
By \eqref{E3.24} then
$\boldsymbol\varphi(\boldsymbol\psi_0)
=\boldsymbol\varphi(r_{_{\Omega}}\boldsymbol\Psi_0)
=\lim_{i\to\infty}\boldsymbol\varphi(r_{_{\Omega}}\boldsymbol\Psi_i)
=\lim_{i\to\infty}\boldsymbol\varphi(\boldsymbol\Psi_i)=0\,,$
and, hence, the asserted independence property follows.

Conversely, assume that $\boldsymbol \phi \in \left({H}_{\Gamma ^+}^1(\Omega ^+)^n\right)'$ and let $r_{\Omega \to \Omega ^+}$ be the restriction operator from the space $\mathring{H}^1(\Omega )^n$ to ${H}_{\Gamma ^+}^1(\Omega ^+)^n$.
Then the mapping $\boldsymbol \varphi :=\boldsymbol \phi \circ r_{\Omega \to \Omega ^+}:\mathring{H}^1(\Omega )^n\to {\mathbb R}$
is linear and bounded, i.e., $\boldsymbol \varphi \in {H^{-1}(\Omega )^n}$.
In addition, for any ${\boldsymbol\psi^-} \in {\mathcal D}(\Omega ^-)^n$, we have that ${\boldsymbol\psi ^-}\in \mathring{H}^1(\Omega )^n$, and accordingly that
$\boldsymbol \varphi \left({\boldsymbol\psi ^-}\right)=\boldsymbol \phi ({\bf 0})=0\,,$
where the last equality is provided by the linearity of the mapping $\boldsymbol \phi:{H}_{\Gamma ^+}^1(\Omega ^+)^n\to {\mathbb R}.$ Consequently, $\boldsymbol \varphi $ belongs to the space defined in \eqref{equivalence-set}, and $\boldsymbol \phi $ can be identified with $\boldsymbol \varphi $ through the relations
$\boldsymbol \varphi :=\boldsymbol \phi \circ r_{\Omega \to \Omega ^+}:\mathring{H}^1(\Omega )^n\to {\mathbb R}$
and
$\boldsymbol\phi =\boldsymbol\varphi \circ E_{\Omega ^+\to \Omega }:{H}_{\Gamma ^+}^1(\Omega ^+)^n\to \mathbb R$.
\qed
\end{proof}

\section{Abstract variational problems and well-posedness results}
\label{B-B-theory}

A main role in our analysis of variational problems is played by the following well-posedness result from \cite{Babuska}, \cite[Theorem 1.1]{Brezzi},
(cf., also \cite[Theorem 2.34]{Ern-Gu} and \cite{Brezzi-Fortin}). 
\begin{theorem}
\label{B-B}
Let $X$ and ${\mathcal M}$ be two real Hilbert spaces. Let $a(\cdot ,\cdot):X\times X\to {\mathbb R}$ and $b(\cdot ,\cdot):X\times {\mathcal M}\to {\mathbb R}$ be bounded bilinear forms. Let $V$ be the subspace of $X$ defined by
\begin{align}
\label{V}
V:=\left\{v\in X: b(v,q)=0\,,\ \forall \, q\in {\mathcal M}\right\}.
\end{align}
Assume that $a(\cdot ,\cdot ):V\times V\to {\mathbb R}$ is coercive, {\Rd that is,} there exists a constant $c_a>0$ such that
\begin{align}
\label{coercive}
a(w,w)\geq c_a\|w\|_X^2\,,\ \ \forall \, w\in V,
\end{align}
and that $b(\cdot ,\cdot)\!:\!X\!\times \!{\mathcal M}\!\to \!{\mathbb R}$ satisfies
the condition
\begin{align}
\label{inf-sup-sm}
&\inf _{q\in {\mathcal M}\setminus \{0\}}\sup_{v\in X\setminus \{0\}}\frac{b(v,q)}{\|v\|_X\|q\|_{\mathcal M}}\geq c_b \,,
\end{align}
with some constant $c_b>0$. Let $f\in X'$ and $g\in {\mathcal M}'$. Then the variational problem
\begin{equation}
\label{mixed-variational}
\left\{\begin{array}{ll}
a(u,v)+b(v,p)&=f(v), \ \ \forall \, v\in X,\\
b(u,q)&=g(q), \ \ \forall \, q\in {\mathcal M},
\end{array}
\right.
\end{equation}
with the unknown $(u,p)\in X\times {\mathcal M}$, is well-posed, which means that \eqref{mixed-variational} has a unique solution $(u,p)$ in $X\times {\mathcal M}$ and there exists a constant $C>0$ depending on $c_a$ and $c_b$, such that
\begin{align}
\label{mixed-C}
\|u\|_{X}+\|p\|_{{\mathcal M}}\leq C\left(\|f\|_{X'}+\|g\|_{{\mathcal M}'}\right).
\end{align}
\end{theorem}

\begin{rem}\label{RB.2}
The linearity and well-posedness of problem \eqref{mixed-variational} under conditions of  Theorem $\ref{B-B}$ imply that the solution of the problem can be represented in the form
$(u,p)=\bs{\mathfrak U}(f,g),$
where the solution operator $\bs{\mathfrak U}:X'\times\mathcal M'\to X\times\mathcal M$ is  linear and continuous.
\end{rem}

We need also the following useful result (see \cite[Theorem A.56, Remark 2.7]{Ern-Gu}.
\begin{lemma}
\label{surj-inj-inf-sup}
Let $X$ and ${\mathcal M}$ be reflexive Banach spaces. Let $b(\cdot ,\cdot):X\times {\mathcal M}\to {\mathbb R}$ be a bounded bilinear form.
Let $B:X\to {\mathcal M}'$ be the bounded linear operator defined by\\
$
\langle Bv,q\rangle =b(v,q),\ \forall \, v\in X,\ \forall \, q\in {\mathcal M},
$
where $\langle \cdot ,\cdot \rangle $ denotes the duality pairing between the dual spaces
${\mathcal M}'$ and ${\mathcal M}$. Let $V:={\rm{Ker}}\, B$. Then the following assertions are equivalent:
\begin{itemize}
\item[$(i)$]
There exists a constant ${c_b>0}$ 
such that $b(\cdot ,\cdot)$ satisfies the inf-sup condition \eqref{inf-sup-sm}.
\item[$(ii)$]
The operator $B:{X/V}\to {\mathcal M}'$ is an isomorphism and
$\|Bw\|_{{\mathcal M}'}\geq {c_b}\|w\|_{X/V}$, $w\in X/V.$
\end{itemize}
\end{lemma}

\section{\bf Useful norm estimates}\label{A-C}

In this appendix we provide several estimates, embeddings, and identities (some of them well known), used in the analysis of the Navier-Stokes problems.
Let $\Omega$ denote a bounded {\bn Lipschitz} domain in ${\mathbb R}^n$, $n\in\{2,3\}$.

$\bullet $ By the Sobolev embedding theorem (see, e.g., \cite[Theorem 6.3]{Adams2003}),
the space $H^1(\Omega)^n$ is {\bn compactly} embedded in $L^{4}(\Omega)^n$
and {\bn there} exists a constant $c_1{\bn=c_1(\Omega,n)}>0$ such that
\begin{align}\label{SET}
\|{\bf v}\|_{L^{4}(\Omega)^n}\le c_1\|{\bf v}\|_{H^1(\Omega)^n}\,, \quad \forall \, {\bf v}\in H^1(\Omega)^n\,.
\end{align}
Due to the equivalence in $\mathring{H}^1(\Omega )^n$ of the semi-norm $\|{\nabla (\cdot )} \|_{L^2(\Omega )^{n\times n}}$ with the norm \mbox{$\|\cdot \|_{{H}^1(\Omega )^n}$} given by \eqref{Sobolev}, estimate \eqref{SET} also implies
\begin{align}
\label{L4}
\|{\bf v}\|_{L^4(\Omega )^n}\leq c_0\|\nabla {\bf v}\|_{L^2(\Omega )^{n\times n}}\,, \quad \forall \ {\bf v}\in \mathring{H}^1(\Omega )^n,
\end{align}
with some constant $c_0=c_0(\Omega ,n)\gr>0$.

$\bullet $
By the H\"{o}lder inequality, we obtain for all ${\bf v}_1,{\bf v}_2, {\bf v}_3\!\in \!{H}^1(\Omega)^n$,
\begin{align}
\left|\left\langle({\bf v}_1\cdot \nabla ){\bf v}_2,{\bf v}_3\right\rangle _{\Omega}\right|
&\leq \|{\bf v}_1\|_{L^{4}(\Omega)^n}\|{\bf v}_3\|_{L^{4}(\Omega)^n} \|\nabla {\bf v}_2\|_{L^{2}(\Omega )^{n\times n}}\nonumber\\
&\leq c_1\|{\bf v}_1\|_{L^{4}(\Omega)^n}\|{\bf v}_3\|_{H^1(\Omega)^{n}} \|\nabla {\bf v}_2\|_{L^{2}(\Omega )^{n\times n}}.
\label{P-0}
\end{align}

This also means that for all ${\bf v}_1,{\bf v}_2, {\bf v}_3\!\in \!{H}^1(\Omega)^n$,
\begin{align}
\left|\left\langle\mathring E_\Omega[({\bf v}_1\cdot \nabla ){\bf v}_2],{\bf v}_3\right\rangle _{\Omega}\right|
&\leq \|\, |\mathring E_\Omega{\bf v}_1| \, |{\bf V}_3|\,\|_{L^{2}(\R^n)}
\|\mathring E_\Omega\nabla {\bf v}_2\|_{L^{2}(\R^n )^{n\times n}}\nonumber\\
&\leq \|{\bf v}_1\|_{L^{4}(\Omega)^n}\|{\bf v}_3\|_{L^{4}(\Omega)^n} \|\nabla {\bf v}_2\|_{L^{2}(\Omega )^{n\times n}}\nonumber\\
&\leq c_1\|{\bf v}_1\|_{L^{4}(\Omega)^n}\|{\bf v}_3\|_{H^1(\Omega)^{n}} \|\nabla {\bf v}_2\|_{L^{2}(\Omega )^{n\times n}},
\label{P-0E}
\end{align}
where ${\bf V}_3\in H^1(\R^n)^n$ is such that $r_{_\Omega}{\bf V}_3={\bf v}_3$.
This implies that $\mathring E_\Omega[({\bf v}_1\cdot \nabla ){\bf v}_2]$ belongs to the space $\widetilde{H}^{-1}(\Omega)^n=\big({H}^1(\Omega)^n\big)'$.
Moreover, in view of \eqref{SET},
for all ${\bf v}_1,{\bf v}_2\in {H}^1(\Omega)^n$,
\begin{align}
\label{P-01til}
\left\|\mathring E_{\Omega}[({\bf v}_1\cdot \nabla ){\bf v}_2]\right\|_{\widetilde{H}^{-1}(\Omega)^n}
\le {\bn c_1^2}\|{\bf v}_1\|_{{H}^{1}(\Omega)^n} \|{\bf v}_2\|_{{H}^{1}(\Omega)^n}.
\end{align}

Taking ${\bf v}_3\in \mathring{H}^1(\Omega)^n$ in \eqref{P-0}, it follows that the term
$({\bf v}_1\cdot \nabla ){\bf v}_2$ belongs to the dual of the space $\mathring{H}^1(\Omega)^n$, that is, to the space ${H}^{-1}(\Omega)^n$ {\bn and for all ${\bf v}_1,{\bf v}_2\in {H}^1(\Omega)^n$,}
\begin{align}
\label{P-01}
\!\!\!\!\left\|({\bf v}_1\cdot \nabla ){\bf v}_2\right\|_{{H}^{-1}(\Omega)^n}
&\le c_1\|{\bf v}_1\|_{L^{4}(\Omega)^n} \|\nabla {\bf v}_2\|_{L^{2}(\Omega )^{n\times n}}\nonumber\\
&\le c_1\|{\bf v}_1\|_{L^4(\Omega)^n}\|{\bf v}_2\|_{H^1(\Omega)^n}
\le c_1^2\|{\bf v}_1\|_{{H}^{1}(\Omega)^n} \|{\bf v}_2\|_{{H}^{1}(\Omega)^n}.
\end{align}

$\bullet $
The dense embedding of the space {\bn$\mathcal D(\overline\Omega)^n$ into ${H}^1(\Omega)^n$}, the divergence theorem and estimate \eqref{P-01} {\rd imply} the following identity for any ${\bf v}_1,{\bf v}_2, {\bf v}_3\in {H}^1(\Omega)^n$
{\bn
\begin{align}
\label{P-051gamma}
&\left\langle({\bf v}_1\cdot \nabla ){\bf v}_2,{\bf v}_3\right\rangle _{\Omega}
=\int_{\Omega}\nabla\cdot\left({\bf v}_1({\bf v}_2\cdot {\bf v}_3)\right)d{\bf x}
-\left\langle(\nabla \cdot{\bf v}_1){\bf v}_3
+({\bf v}_1\cdot \nabla ){\bf v}_3,{\bf v}_2\right\rangle _{\Omega}\nonumber
\\
&=\left\langle\gamma_{\Omega}{\bf v}_1\cdot \bs\nu,\gamma_{\Omega}{\bf v}_2
\cdot\gamma_{\Omega}{\bf v}_3\right\rangle _{\partial\Omega}
-\left\langle(\nabla \cdot{\bf v}_1){\bf v}_3
+({\bf v}_1\cdot \nabla ){\bf v}_3,{\bf v}_2\right\rangle _{\Omega}\,,
 \end{align}
where $\bs\nu$ is the normal vector on $\partial\Omega$ directed outward $\Omega$.

To obtain an alternative versions of estimate \eqref{P-0E}, which does not involve
$\|\nabla {\bf v}_2\|_{L^{2}(\Omega)^{n\times n}}$, let use \eqref{P-051gamma} and take into account that
$\gamma_{\Omega}{\bf v}_1,\gamma_{\Omega}{\bf v}_2,\gamma_{\Omega}{\bf v}_3\in {H}^{1/2}(\partial\Omega)^n$.
Further, we employ that for the Lipschitz domain $\Omega\in\R^n$, $n=2,3$, the space ${H}^{1/2}(\partial\Omega)^n$ is continuously embedded in $L^3(\partial\Omega)^n$ (e.g., by the embeddings in \cite[Section 2.2.4, Corollary 2(i)]{Runst-Sickel} for $\R^{n-1}$ that can be extended to Lipschitz surfaces by standard arguments, cf. also a more general statement
in \cite[Proposition 3.8]{MM2013Spr}).
Thus, there exists a constant $c_2=c_2(\partial\Omega,n)$, such that
\begin{align}\label{SETgamma}
\|\bs\phi\|_{L^{3}(\partial\Omega)^n}\le c_2\|\bs \phi\|_{H^{1/2}(\partial\Omega)^n}\,, \quad \forall \, \bs\phi\in H^{1/2}(\partial\Omega)^n\,.
\end{align}
Hence by the H\"older inequality we have for all ${\bf v}_1,{\bf v}_2,{\bf v}_3\!\in \!{H}^1(\Omega)^n$,
\begin{align}
\left|\left\langle\gamma_{\Omega}{\bf v}_1\cdot \bs\nu,\gamma_{\Omega}{\bf v}_2
\cdot\gamma_{\Omega}{\bf v}_3\right\rangle _{\partial\Omega}\right|
\leq \|\, |\gamma_{\Omega}{\bf v}_1| \, |\gamma_{\Omega}{\bf v}_2|\,\|_{L^{3/2}(\partial\Omega)}
\|\gamma_{\Omega}{\bf v}_3\|_{L^{3}(\partial\Omega)^{n}}\nonumber\\
\leq \|\gamma_{\Omega}{\bf v}_1\|_{L^{3}(\partial\Omega)^{n}}\|\gamma_{\Omega}{\bf v}_2\|_{L^{3}(\partial\Omega)^{n}}
\|\gamma_{\Omega}{\bf v}_3\|_{L^{3}(\partial\Omega)^{n}}\nonumber\\
\leq c^2_2\|\gamma_{\Omega}{\bf v}_1\|_{H^{1/2}(\partial\Omega)^n}\|\gamma_{\Omega}{\bf v}_2\|_{L^{3}(\partial\Omega)^{n}}
\|\gamma_{\Omega}{\bf v}_3\|_{H^{1/2}(\partial\Omega)^n}.
\label{P-0gamma}
\end{align}
Then from \eqref{P-051gamma} and \eqref{P-0gamma} we obtain for any ${\bf v}_1,{\bf v}_2,{\bf v}_3\!\in \!{H}^1(\Omega)^n$
\begin{align}
&\left|\left\langle({\bf v}_1\cdot \nabla ){\bf v}_2,{\bf v}_3\right\rangle _{\Omega}\right|
\leq \|\gamma_{\Omega}{\bf v}_1\|_{L^{3}(\partial\Omega)^{n}}\|\gamma_{\Omega}{\bf v}_2\|_{L^{3}(\partial\Omega)^{n}}
\|\gamma_{\Omega}{\bf v}_3\|_{L^{3}(\partial\Omega)^{n}}\nonumber\\
&+ \|\nabla\cdot {\bf v}_1\|_{L^{2}(\Omega)^{n\times n}} \|{\bf v}_2\|_{L^{4}(\Omega)^n} \|{\bf v}_3\|_{L^{4}(\Omega)^n}
+ \|{\bf v}_1\|_{L^{4}(\Omega)^n}\|{\bf v}_2\|_{L^{4}(\Omega)^n}
\|\nabla {\bf v}_3\|_{L^{2}(\Omega)^{n\times n}}\nonumber\\
&\leq c^2_2\|\gamma_{\Omega}\|^2
\|{\bf v}_1\|_{H^{1}(\Omega)^n}\|\gamma_{\Omega}{\bf v}_2\|_{L^{3}(\partial\Omega)^{n}}
\|{\bf v}_3\|_{H^{1}(\Omega)^n}\nonumber\\
&\qquad+ 2c_1\|{\bf v}_1\|_{H^1(\Omega)^{n}}\|{\bf v}_2\|_{L^{4}(\Omega)^n}\|{\bf v}_3\|_{H^1(\Omega)^{n}}\nonumber\\
&=
\|{\bf v}_1\|_{H^{1}(\Omega)^n}
(c^2_2\|\gamma_{\Omega}\|^2
\|\gamma_{\Omega}{\bf v}_2\|_{L^{3}(\partial\Omega)^{n}}
+ 2c_1\|{\bf v}_2\|_{L^{4}(\Omega)^n})\|{\bf v}_3\|_{H^{1}(\Omega)^n},
\label{P-0Egamma}
\end{align}
where $\|\gamma_{\Omega}\|:=\|\gamma_{\Omega}\|_{H^{1}(\Omega)^n\to H^{1/2}(\partial\Omega)^n}$ is the norm of the trace operator.
Similar to \eqref{P-01til} and by using the previous estimate, we obtain for all ${\bf v}_1,{\bf v}_2\in {H}^1(\Omega)^n$ that
\begin{align}
\label{P-01til-gamma}
\hspace{-0.75em}\left\|\mathring E_{\Omega}[({\bf v}_1\cdot \nabla ){\bf v}_2]\right\|_{\widetilde{H}^{-1}(\Omega)^n}\le
\|{\bf v}_1\|_{H^{1}(\Omega)^n}
(c^2_2\|\gamma_{\Omega}\|^2 \|\gamma_{\Omega}{\bf v}_2\|_{L^{3}(\partial\Omega)^{n}}
+ 2c_1\|{\bf v}_2\|_{L^{4}(\Omega)^n}).
\end{align}

$\bullet $
For ${\bf v}_3\in \mathring{H}^1(\Omega)^n$, \eqref{P-051gamma}  simplifies to
}
\begin{align}
\label{P-051}
\left\langle({\bf v}_1\cdot \nabla ){\bf v}_2,{\bf v}_3\right\rangle _{\Omega}
&=-\left\langle(\nabla \cdot{\bf v}_1){\bf v}_3
+({\bf v}_1\cdot \nabla ){\bf v}_3,{\bf v}_2\right\rangle _{\Omega}\,,
\ \forall\, {\bf v}_1,{\bf v}_2\in {H}^1(\Omega)^n,\ {\bf v}_3\in \mathring{H}^1(\Omega)^n, \end{align}
In view of \eqref{P-051} we also obtain the identity
\begin{align}
\label{antisym}
\!\!\!\!\!\!
\left\langle({\bf v}_1\cdot \nabla ){\bf v}_2,{\bf v}_3\right\rangle _{\Omega}
&\!\!=\!-\left\langle({\bf v}_1\cdot \nabla ){\bf v}_3,{\bf v}_2\right\rangle _{\Omega}\,,
\, \forall\ {\bf v}_1\in {H}_{\rm{div}}^1(\Omega)^n,\
{\bf v}_2\in {H}^1(\Omega)^n,\, {\bf v}_3\in \mathring{H}^1(\Omega)^n\,,
\end{align}
and hence the well known formula
\begin{equation}
\label{P-5a}
\left\langle ({\bf v}_1\cdot \nabla ){\bf v}_2,{\bf v}_2\right\rangle _{\Omega}=0\,, \ \forall \,
{\bf v}_1\!\in \!{H}_{\rm{div}}^1(\Omega)^n,\ {\bf v}_2\!\in \!\mathring{H}^1(\Omega)^n.
\end{equation}

On the other hand, in the more general cases \eqref{P-051gamma} and \eqref{P-051} imply
\begin{align}
\label{P-5gamma-g}
2\left\langle ({\bf v}_1\cdot \nabla ){\bf v}_2,{\bf v}_2\right\rangle _{\Omega}
=\left\langle\gamma_{\Omega}{\bf v}_1\cdot \bs\nu,\gamma_{\Omega}{\bf v}_2
\cdot\gamma_{\Omega}{\bf v}_2\right\rangle _{\partial\Omega}
-\left\langle(\nabla \cdot{\bf v}_1),{\bf v}_2\cdot{\bf v}_2\right\rangle _{\Omega}\\
\forall \,
{\bf v}_1,{\bf v}_2\!\in \!{H}^1(\Omega)^n,\nonumber
\end{align}
\begin{equation}
\label{P-5g}
2\left\langle ({\bf v}_1\cdot \nabla ){\bf v}_2,{\bf v}_2\right\rangle _{\Omega}
=-\left\langle(\nabla \cdot{\bf v}_1),{\bf v}_2\cdot{\bf v}_2\right\rangle _{\Omega}\,, \ \forall \,
{\bf v}_1\!\in \!{H}^1(\Omega)^n,\ {\bf v}_2\!\in \!\mathring{H}^1(\Omega)^n.
\end{equation}

Similar arguments to those for \eqref{P-0} and identity \eqref{antisym} imply the estimate
\begin{align}
\label{C7}
\left|\left\langle({\bf v}_1\cdot \nabla ){\bf v}_2,{\bf v}_3\right\rangle _{\Omega}\right|
=\left|\left\langle({\bf v}_1\cdot \nabla ){\bf v}_3,{\bf v}_2\right\rangle _{\Omega}\right|
&\leq \|{\bf v}_1\|_{L^{4}(\Omega)^n}\|{\bf v}_2\|_{L^{4}(\Omega)^n} \|\nabla {\bf v}_3\|_{L^{2}(\Omega )^{n\times n}}
\end{align}
for all ${\bf v}_1\in {H}_{\rm{div}}^1(\Omega)^n$, ${\bf v}_2\in {H}^1(\Omega)^n$,
${\bf v}_3\in \mathring{H}^1(\Omega)^n$.
Therefore
\begin{align}
\|({\bf v}_1\cdot \nabla ){\bf v}_2\|_{H^{-1}(\Omega)^n}
&\le |\!|\!|({\bf v}_1\cdot \nabla ){\bf v}_2|\!|\!|_{H^{-1}(\Omega)^n}
\le\|{\bf v}_1\|_{L^{4}(\Omega)^n}\|{\bf v}_2\|_{L^{4}(\Omega)^n}\nonumber\\
\label{P-0512b}
&\le c_1\|{\bf v}_1\|_{H^1(\Omega)^n}\|{\bf v}_2\|_{L^4(\Omega)^n}\
\forall\, {\bf v}_1\in {H}_{\rm{div}}^1(\Omega)^n,\, {\bf v}_2\in {H}^1(\Omega)^n.
\end{align}

$\bullet $
Let now Assumption $\ref{interface-Sigma}$ hold and $\Omega'$ be either $\Omega^+$ or $\Omega^-$.
Similar to \eqref{P-0E} we have that for any ${\bf v}_1,{\bf v}_2\!\in \!{H}^1(\Omega')^n$ and ${\bf v}_3\!\in \!{H}^1(\Omega)^n$
\begin{align}
\left|\left\langle\mathring E_{\Omega'}[({\bf v}_1\cdot \nabla ){\bf v}_2],{\bf v}_3\right\rangle _{\Omega}\right|
&
=\left|\left\langle({\bf v}_1\cdot \nabla ){\bf v}_2,{\bf v}_3\right\rangle _{\Omega'}\right|
\nonumber\\
&\leq \|{\bf v}_1\|_{L^{4}(\Omega')^n}\|{\bf v}_3\|_{L^{4}(\Omega')^n}
\|\nabla {\bf v}_2\|_{L^{2}(\Omega')^{n\times n}}
\nonumber\\
&\leq c'_1\|{\bf v}_1\|_{L^{4}(\Omega')^n}\|{\bf v}_3\|_{H^1(\Omega')^{n}} \|\nabla {\bf v}_2\|_{L^{2}(\Omega')^{n\times n}}\nonumber\\
&\leq c'_1\|{\bf v}_1\|_{L^{4}(\Omega')^n}\|{\bf v}_3\|_{H^1(\Omega)^{n}} \|\nabla {\bf v}_2\|_{L^{2}(\Omega')^{n\times n}}.
\label{P-0E'}
\end{align}
Taking ${\bf v}_3\in \mathring{H}^1(\Omega)^n$ in \eqref{P-0E'}, we find that
$r_{_\Omega}\mathring E_{\Omega'}[({\bf v}_1\cdot \nabla ){\bf v}_2]$ belongs to  ${H}^{-1}(\Omega)^n$ and
\begin{multline}
\label{P-01'3}
\left\|\mathring E_{\Omega'}[({\bf v}_1\cdot \nabla ){\bf v}_2]\right\|_{{H}^{-1}(\Omega)^n}
\le c'_1\|{\bf v}_1\|_{L^{4}(\Omega')^n} \|\nabla {\bf v}_2\|_{L^{2}(\Omega')^{n\times n}}\\
\hspace{-1em}\le c'_1\|{\bf v}_1\|_{L^{4}(\Omega')^n} \|{\bf v}_2\|_{{H}^{1}(\Omega')^n}
\le (c_1')^2\|{\bf v}_1\|_{{H}^{1}(\Omega')^n} \|{\bf v}_2\|_{{H}^{1}(\Omega')^n}\ \forall\, {\bf v}_1,{\bf v}_2\in {H}^1(\Omega')^n,
\end{multline}
where
$c_1'=c_1(\Omega',n)$, cf. \eqref{SET}. If, moreover, ${\bf v}_1\in {H}^1(\Omega)^n$, then \eqref{P-01'3} implies
\begin{multline}
\label{P-01'3-1}
\left\|\mathring E_{\Omega'}[(r_{_{\Omega'}}{\bf v}_1\cdot \nabla ){\bf v}_2]\right\|_{{H}^{-1}(\Omega)^n}
\le c'_1\|{\bf v}_1\|_{L^{4}(\Omega)^n} \|{\bf v}_2\|_{{H}^{1}(\Omega')^n}\\
\le (c'_1)^2\|{\bf v}_1\|_{{H}^{1}(\Omega)^n} \|{\bf v}_2\|_{{H}^{1}(\Omega')^n}
\ \forall\, {\bf v}_1\in {H}^1(\Omega)^n,\,{\bf v}_2\in {H}^1(\Omega')^n.
\end{multline}

$\bullet $
Let again Assumption $\ref{interface-Sigma}$ hold.
In order to obtain some alternative versions of estimates \eqref{P-0E'} and \eqref{P-01'3}, which do not involve
$\|\nabla {\bf v}_2\|_{L^{2}(\Omega')^{n\times n}}$, we implement \eqref{P-0Egamma} for $\Omega'$ and find that for all
${\bf v}_1,{\bf v}_2\!\in \!{H}^1(\Omega')^n$ and ${\bf v}_3\!\in \!{H}^1(\Omega)^n$
\begin{align}
&\left|\left\langle\mathring E_{\Omega'}[({\bf v}_1\cdot \nabla ){\bf v}_2],{\bf v}_3\right\rangle _{\Omega}\right|
=\left|\left\langle({\bf v}_1\cdot \nabla ){\bf v}_2,{\bf v}_3\right\rangle _{\Omega'}\right|\nonumber\\
&\le
\|{\bf v}_1\|_{H^{1}(\Omega')^n}
(c'^2_2\|\gamma_{\Omega'}\|^2\|\gamma_{\Omega'}{\bf v}_2\|_{L^{3}(\partial\Omega')^{n}}
+ 2c'_1\|{\bf v}_2\|_{L^{4}(\Omega')^n})\|{\bf v}_3\|_{H^{1}(\Omega)^n},
\label{P-0E'gamma}
\end{align}
where $c'_2=c_2(\partial\Omega',n)$.
If we take ${\bf v}_3\!\in \!\mathring{H}^1(\Omega)^n$, then \eqref{P-0E'gamma} implies
\begin{multline}
\label{P-01-gamma'}
\left\|\mathring E_{\Omega'}[({\bf v}_1\cdot \nabla ){\bf v}_2]\right\|_{{H}^{-1}(\Omega)^n}\le
\|{\bf v}_1\|_{H^{1}(\Omega')^n}
\Big(c'^2_2\|\gamma_{\Omega'}\|^2\|\gamma_{\Omega'}{\bf v}_2\|_{L^{3}(\partial\Omega')^{n}}\\
+ 2c'_1\|{\bf v}_2\|_{L^{4}(\Omega')^n}\Big), \quad \forall\ {\bf v}_1,{\bf v}_2\in {H}^1(\Omega')^n.
\end{multline}

\section*{\bf Acknowledgements}
The research has been supported by the grant EP/M013545/1: "Mathematical Analysis of Boundary-Domain Integral Equations for Nonlinear PDEs" from the EPSRC, UK. M. Kohr has been partially supported by the Babe\c{s}-Bolyai University research grant {AGC35124/31.10.2018}. 
W.L. Wendland has been partially supported by "Deutsche Forschungsgemeinschaft (DFG, German Research Foundation) under Germany's Excellence Strategy-EXC 2075-390740016".


\begin{thebibliography}{100}

\bibitem{Adams2003}
R.A. Adams, J.J.F. Fournier, Sobolev Spaces, Academic Press, Amsterdam, Boston, 2003.

\bibitem{Amrouche-Bellido}
C. Amrouche, M.A. Rodr\'{i}guez-Bellido, The Oseen and Navier-Stokes equations in a
non-solenoidal framework. Math. Meth. Appl. Sci. \textbf{39} (2016), 5066--5090.

\bibitem{Am-Ciarlet}
C. Amrouche, P.G. Ciarlet, C. Mardare, On a Lemma of Jacques-Louis Lions and its relation to other fundamental results. J. Math. Pures Appl. \textbf{104} (2015), 207--226.


\bibitem{Angot-2}
Ph. Angot, A fictitious domain model for the Stokes/Brinkman problem with jump
embedded boundary conditions. C. R. Acad. Sci. Paris, Ser. I \textbf{348} (2010), 697--702.

\bibitem{Angot-3}
Ph. Angot, On the well-posed coupling between free fluid and porous viscous flows, Appl. Math. Lett. \textbf{24} (2011), 803--810.

\bibitem{Babuska}
I. Babu\u{s}ka, The finite element method with Lagrangian multipliers. Numer. Math. \textbf{20} (1973), 179--192.

\bibitem{Bogovskii}
M.E. Bogovski\u{\i}, Solution of the first boundary value problem for an equation of continuity of an incompressible medium. Dokl. Akad. Nauk SSSR \textbf{248}(5) (1979), 1037--1040.

\bibitem{B-M-M-M}
K. Brewster, D. Mitrea, I. Mitrea, and M. Mitrea, Extending Sobolev functions with partially vanishing traces from locally $(\epsilon ,\delta)$-domains and applications to mixed boundary problems. J. Funct. Anal. \textbf{266} (2014), 4314--4421.

\bibitem{Brezzi}
F. Brezzi, On the existence, uniqueness and approximation of saddle-point problems arising from lagrangian multipliers. R.A.I.R.O. Anal. Numer. \textbf{R2} (1974), 129--151.

\bibitem{Brezzi-Fortin}
F. Brezzi, M. Fortin, Mixed and Hybrid Finite Element Methods. Springer Series in Comput. Math. {\bf 15}, Springer-Verlag, New York, 1991.

\bibitem{CMN-1}
O. Chkadua, S.E. Mikhailov, D. Natroshvili, Analysis of direct boundary-domain integral equations for a mixed BVP with variable coefficient. I. Equivalence and invertibility. J. Int. Equ. Appl. \textbf{21} (2009), 499--542.

\bibitem{CMN-2}
O. Chkadua, S.E. Mikhailov, D. Natroshvili, Analysis of direct boundary-domain integral equations for a mixed BVP with variable coefficient. II. Solution regularity and asymptotics. J. Int. Equ. Appl. \textbf{22} (2010), 19-37.

\bibitem{Ch-Mi-Na}
O. Chkadua, S.E. Mikhailov, D. Natroshvili, Localized boundary-domain singular integral equations based on harmonic parametrix for divergence-form elliptic PDEs with variable matrix coefficients. Integr. Equ. Oper. Theory. \textbf{76} (2013), 509--547.

\bibitem{Choi-Dong-Kim-JMFM}
{J. Choi, H. Dong, D. Kim, Green functions of conormal derivative problems for stationary Stokes system. J. Math. Fluid Mech. \textbf{20} (2018), 1745--1769.}

\bibitem{Choi-Lee}
J. Choi, K-A. Lee, The Green function for the Stokes system with measurable coefficients. Comm. Pure Appl. Anal. \textbf{16} (2017), 1989--2022.

\bibitem{Co}
M. Costabel, Boundary integral operators on Lipschitz domains: Elementary results. SIAM J. Math. Anal. \textbf{19} (1988), 613--626.

\bibitem{D-M}
M. Dindo\u{s}, M. Mitrea, The stationary Navier-Stokes system in nonsmooth manifolds: The Poisson problem in Lipschitz and $C^1$ domains. Arch. Rational Mech. Anal. \textbf{174} (2004), 1--47.

\bibitem{Dong-Kim}
{\bl H. Dong, D. Kim, $L_r$-Estimates for stationary Stokes system with coefficients measurable in one direction. Bull. Math. Sci. \textbf{9} (2019), 1950004.}

\bibitem{duffy}
B.R. Duffy, Flow of a liquid with an anisotropic viscosity tensor. J. Nonnewton. Fluid Mech. \textbf{4} (1978), 177-193.

\bibitem{Ern-Gu}
A. Ern, J.L. Guermond, Theory and Practice of Finite Elements. Springer, New York, 2004.

\bibitem{FP-Mik2019}
C. Fresneda-Portillo,  S.E. Mikhailov,  Analysis of boundary-domain integral equations to the mixed BVP for a compressible Stokes system with variable viscosity, Communic. Pure and Appl. Analysis, \textbf{18} (2019), 3059--3088. 

\bibitem{Galdi}
G.P. Galdi, An Introduction to the Mathematical Theory of the Navier--Stokes Equations. Steady-State Problems, Second Edition, Springer, New York 2011.

\bibitem{Ga-Sa}
G.N. Gatica, R. Oyarz\'{o}ua, F.J. Sayas, A twofold saddle point approach for the coupling of fluid flow with nonlinear porous media flow, IMA J. Numer. Anal. \textbf{32} (2012), 845--887.
\bibitem{Gi-Ra}
V. Girault, P. A. Raviart, Finite Element Methods for Navier-Stokes Equations, Theory and Algorithms. Springer Series in Comp. Math. {\bf 5}, Springer-Verlag, Berlin, 1986.

\bibitem{Gilbarg-Trudinger}
D. Gilbarg, N.S. Trudinger, Elliptic Partial Differential Equations of Second Order. Springer,
Berlin, 2001.

\bibitem{Grisvard1985}
P. Grisvard Elliptic Problems in Nonsmooth Domains, Pitman: Boston, London, Melbourne, 1985.

\bibitem{H-W}
G.C. Hsiao, W.L. Wendland, Boundary Integral Equations. Springer-Verlag, Heidelberg 2008.

\bibitem{K-L-W}
M. Kohr, M. Lanza de Cristoforis, W.L. Wendland, Nonlinear Neumann-transmission problems for Stokes and Brinkman equations on Euclidean Lipschitz domains. Potential Anal. \textbf{38} (2013), 1123--1171.

\bibitem{K-M-W}
M. Kohr, S.E. Mikhailov, W.L. Wendland, Transmission problems for the Navier-Stokes and Darcy-Forchheimer-Brinkman systems in Lipschitz domains on compact Riemannian manifolds. J. Math. Fluid Mech. \textbf{19} (2017), 203--238.

\bibitem{K-M-W-2}
M. Kohr, S.E. Mikhailov, W.L. Wendland, Potentials and transmission problems in weighted Sobolev spaces for anisotropic Stokes and Navier-Stokes systems with $L_\infty $ strongly elliptic coefficient tensor, Complex Var. Elliptic Equ., \textbf{65} (2020), 109--140.

\bibitem{KMW-DCDS2021}
M. Kohr, S.E. Mikhailov, W.L. Wendland, Dirichlet and transmission problems for anisotropic Stokes and Navier-Stokes systems with $L_\infty$ tensor coefficient
under relaxed ellipticity condition. Discrete Contin. Dyn. Syst.,  doi:10.3934/dcds.2021042.

\bibitem{KMW-LP}
M. Kohr, S.E. Mikhailov, W.L. Wendland, Layer potential theory for the anisotropic Stokes system with variable $L_\infty$ symmetrically elliptic tensor coefficient. Math. Meth. Appl. Sci., {\rd DOI: 10.1002/mma.7167}, to appear.

\bibitem{K-W}
M. Kohr, W.L. Wendland, Variational approach for the Stokes and Navier-Stokes systems with nonsmooth coefficients in Lipschitz domains on compact Riemannian manifolds. Calc. Var. Partial Differ. Equ. {\bf 57:165} (2018), 1--41.

\bibitem{K-W1}
M. Kohr, W.L. Wendland, Layer potentials and Poisson problems for the nonsmooth coefficient Brinkman system in Sobolev and Besov spaces. J. Math. Fluid Mech. \textbf{20} (2018), 1921--1965.

\bibitem{Korobkov}
M.V. Korobkov, K. Pileckas, R. Russo, On the flux problem in the theory of steady Navier-Stokes equations with non-homogeneous boundary conditions. Arch. Rational Mech. Anal. \textbf{207} (2013), 185--213.

\bibitem{LS1976}
O. A. Ladyzhenskaya, V. A. Solonnikov, Some problems of vector analysis and generalized formulations of boundary value problems for Navier-Stokes equations. Zap. Nauchn. Sem. LOMI. Leningrad. Otdel. Mat. Inst. Steklov, {\bf 59} (1976), 81-116. English transl. in J. Soviet Math. {\bf 10} (1978) No.2.

\bibitem{Leray}
J. Leray, \'{E}tude de diverses \'{e}quations int\'{e}grales non lin\'{e}aire et de quelques probl$\grave{e}$mes que pose l'hydrodynamique. J. Math. Pures Appl. \textbf{12} (1933), 1--82.

\bibitem{LiMa1}
J.-L. Lions, E. Magenes Non-Homogeneous Boundary Value Problems and Applications, Vol. 1. Springer: Berlin-Heidelberg-New York, 1972.

\bibitem{Malkin}
A. Ya. Malkin, A.I. Isayev, Rheology: Concepts, Methods, and Applications, 3rd Edition, ChemTec Publishing, Toronto, 2017.

\bibitem{Ma-Ni}
{A.L. Mazzucato, V. Nistor, Well-posedness and regularity for the elasticity equation with mixed boundary conditions on polyhedral domains and domains with cracks. Arch. Rational Mech. Anal. \textbf{195} (2010), 25--73.}

\bibitem{Lean}
W. McLean, Strongly Elliptic Systems and Boundary Integral Equations. Cambridge University Press, UK, 2000.

\bibitem{Mikh}
S.E. Mikhailov, Traces, extensions and co-normal derivatives for elliptic systems on Lipschitz domains. J. Math. Anal. Appl. \textbf{378} (2011), 324--342.

\bibitem{Mikh-3}
S.E. Mikhailov, Solution regularity and co-normal derivatives for elliptic systems with non-smooth coefficients on Lipschitz domains. J. Math. Anal. Appl. \textbf{400} (2013), 48-67.

\bibitem{Mikh-18}
S.E. Mikhailov, Analysis of segregated boundary-domain integral equations for BVPs with non-smooth coefficient on Lipschitz domains. Boundary Value Problems. 2018:87, 1--52.

\bibitem{M-M1}
I. Mitrea, M. Mitrea, The Poisson problem with mixed boundary conditions in Sobolev and Besov spaces in non-smooth domains. Trans. Amer. Math. Soc. \textbf{359} (2007), 4143--4182.

\bibitem{MM2013Spr}
I. Mitrea, M. Mitrea, Multi-Layer Potentials and Boundary Problems for Higher-Order Elliptic Systems in {L}ipschitz Domains, Springer, {Berlin - Heidelberg} (2013).

\bibitem{M-W}
M. Mitrea, M. Wright, Boundary value problems for the Stokes system in arbitrary Lipschitz domains. Ast\'{e}risque. \textbf{344} (2012), viii+241 pp.

\bibitem{Oleinik}
O.A. Oleinik, A.S. Shamaev, G.A. Yosifian, Mathematical Problems in Elasticity and Homogenization. Vol. 26, Horth-Holland, Amsterdam, 1992.

\bibitem{Rudin1991} W. Rudin, Functional Analysis. McGraw-Hill, New York, 1991.

\bibitem{Runst-Sickel}
T. Runst, W. Sickel, Sobolev Spaces of Fractional order, Nemytskij Operators, and Nonlinear Partial Differential Equations. De Gruyter, Berlin, 1996.

\bibitem{Sayas-book}
{F-J. Sayas, T.S. Brown, M.E. Hassell, Variational Techniques for Elliptic Partial Differential Equations. Theoretical Tools and Advanced Applications. CRC Press, Boca Raton, FL, 2019.}

\bibitem{Seregin}
G. Seregin, Lecture Notes on Regularity Theory for the Navier-Stokes Equations. World Scientific, London, 2015.

\bibitem{Schwartz1969}
J.T. Schwartz, Nonlinear Functional Analysis. Gordon and Breach, N.Y., London, Paris (1969).

\bibitem{Temam}
R. Temam, Navier-Stokes Equations. Theory and Numerical Analysis. AMS Chelsea Edition, American Mathematical Society, 2001.

\end{thebibliography}
\end{document}